\documentclass[reqno,10pt]{amsart}
\usepackage{amsfonts}
\usepackage{graphicx}
\usepackage{amsfonts,amsmath, amssymb}
\usepackage{marginnote}
\usepackage{color}
\usepackage{cite}
\usepackage{bm}
\usepackage{enumerate}
\numberwithin{equation}{section}

\theoremstyle{plain}
\newtheorem{theorem}{Theorem}[section]

\newtheorem{lemma}[theorem]{Lemma}
\newtheorem{proposition}[theorem]{Proposition}

\theoremstyle{remark}
\newtheorem{remark}[theorem]{Remark}

\def\a{\alpha}
\def\g{\gamma}

\def\f{\frac}

\usepackage{tikz}

\newcommand{\dd}{{\rm d}}

\newcommand{\R}{\mathbb{R}}

\newcommand{\T}{\mathbb{T}}

\newcommand{\CE}{\mathcal{E}}
\newcommand{\CF}{\mathcal{F}}

\newcommand{\CL}{\mathcal{L}}

\newcommand{\CG}{\mathcal{G}}

\newcommand{\pa}{\partial}
\newcommand{\ka}{\kappa}

\newcommand{\vep}{\varepsilon}

\begin{document}
\date{\today}
\title[Tollmien-Schlichting waves over Prandtl-Hartmann layer]{ Analysis on Tollmien-Schlichting wave in the  Prandtl-Hartmann Regime}

\author[C.-J. Liu]{Cheng-Jie Liu}
\address[C.-J. Liu]{School of Mathematical Sciences, LSC-MOE and Institute of Natural Sciences, Shanghai Jiao Tong University, Shanghai, China}
\email{liuchengjie@sjtu.edu.cn}

\author[T. Yang]{Tong Yang}
\address[T. Yang]{Department of Mathematics, City University of Hong Kong, Hong Kong, China}
\email{matyang@cityu.edu.hk}

\author[Z. Zhang]{Zhu Zhang}
\address[Z. Zhang]{Department of Mathematics, City University of Hong Kong, Hong Kong, China}
\email{zhuzhangpde@gmail.com}

\begin{abstract}
	In this paper, we study the instability induced by the Tollmien-Schlichting wave governed by the MHD system in the Prandtl-Hartmann regime. The interaction of the inviscid mode and viscous mode that leads to the instability
	is analyzed by the introduction of a new decomposition of the Orr-Sommerfeld operator on the velocity and magnetic fields. The critical Gevrey index for the instability is justified by constructing 
	the growing mode in the essential frequency and it is shown to be the same as the  incompressible  Navier-Stokes equations in the Prandtl regime.
	This result justifies rigorously the physical understanding that the transverse magnetic field to the boundary in the Prandtl-Hartmann regime has no extra stabilizing effect on the Tollmien-Schlichting wave.
\end{abstract}

\maketitle
\tableofcontents
\section{Introduction}

It is a classical problem in fluid mechanics about the stability and instability of different hydrodynamic patterns in various physical settings, in particular in the transition region from  laminar flow to turbulence that involves different
modes interaction in the boundary layer. The analysis study can be traced back to the early work by Lord Rayleigh and Heisenberg among many others. In fact, for inviscid flow, the classical criterion for stability was given by Rayleigh saying that a necessary condition for instability of shear flow profile is that the shear flow must have an inflection point, and it was later refined by Fjortoft, see \cite{DR,SG}. On the other hand, for viscous flow, 
except the case of the linear Couette flow, which is proved to be linearly stable for all Reynolds numbers by Romanov \cite{R}, all other profiles (including those which are inviscid stable) are shown to be linearly unstable for large Reynolds numbers \cite{DR,GGN1,SG}.

Despite the tremendous progress on the mathematical theory on the high Reynolds number limit, in particular for the Navier-Stokes equations, there are still many challenging and unsolved mathematical problems. This paper aims to study
a well-known physical phenomenon of the instability induced by Tollmien-Schlichting wave for the electrically
conducting fluid  in the transition region from laminar flow to turbulence. For this,
we consider the following 2D incompressible MHD system that is a fundamental model describing  behavior of electrically conducting fluid: 
\begin{align}
\label{1.1}
\left\{
\begin{array}{ll}
\partial_t \vec{U}+\vec{U}\cdot\nabla \vec{U}+\nabla p-\frac{1}{\hbox{Re}}\triangle \vec{U}=S \vec{H}\cdot\nabla \vec{H},\\
\partial_t \vec{H}-\hbox{curl}(\vec{U}\times \vec{H})+\frac{1}{\hbox{Rm}}\hbox{curl curl} \vec{H}=\vec{0},\\
\hbox{div} \vec{U}=0,\quad \hbox{div} \vec{H}=0,\quad (x,y)\in\Omega=\T\times\R_+,
\end{array}
\right.
\end{align}
where  the physical parameters $\hbox{Re}$ and $\hbox{Rm}$ are the hydrodynamic and magnetic Reynolds numbers respectively, and 
$S=\frac{\hbox{Ha}^2}{\hbox{Re}\hbox{Rm}}
$
with  $\hbox{Ha}$ being Hartmann number. Here $\vec{U}=(u,v)$ and $\vec{H}=(h,g)$ represent the velocity field and magnetic field respectively, and $p$ stands for the total pressure.

Precisely, we consider
the problem in the Prandtl-Hartmann regime by taking the following scaling for large Reynolds number:
\begin{align*}
\hbox{Re}\sim \frac{1}{\varepsilon}, \qquad \hbox{Ha}\sim\hbox{Rm}\sim \frac{1}{\sqrt{\varepsilon}},
\end{align*}
where $\varepsilon$ is a small parameter. Hence, the system \eqref{1.1} becomes
\begin{align}
\label{1.2}
\left\{
\begin{array}{ll}
\partial_t \vec{U}+ \vec{U}\cdot\nabla \vec{U}+\nabla p-\varepsilon\triangle \vec{U}=\sqrt{\varepsilon} \vec{H}\cdot\nabla \vec{H},\\
\partial_t \vec{H}-\hbox{curl}( \vec{U}\times \vec{H})+\sqrt{\varepsilon}\hbox{curl curl} \vec{H}=\vec{0},\\
\hbox{div} \vec{U}=0,\quad \hbox{div} \vec{H}=0,\quad t>0,\quad (x,y) \in\Omega.
\end{array}
\right.
\end{align}

Many physical literatures predict that under this scaling, the Hartmann boundary layer  develops near the boundary by the transversal  magnetic field, cf. the original work by Hartmann\cite{H}. For simplicity, we consider the case with the transverse background magnetic field $\vec{H}=(0,1)$. In this case, the classical Prandtl-Hartmann boundary layer profile 
\begin{align}\label{ha}
(\vec{U}_s, \vec{H}_s)=[(U_s(\frac{y}{\sqrt{\varepsilon}}),0),(H_s(\frac{y}{\sqrt{\varepsilon}}),1)]=[(u_\infty(1-e^{-\frac{y}{\sqrt{\varepsilon}}}),0), (h_{\infty}-u_\infty e^{-\frac{y}{\sqrt{\varepsilon}}},1)]
\end{align}
 is an exact steady solution to the above system, where the far fields $u_\infty$  and $h_{\infty}$ are two given
 	constants. 
 	
 	 Then it is very natural to investigate its stability/ instability properties along the dynamics of \eqref{1.2}. 
In the following discussion, we take $u_\infty=1$ without loss of generality.

One of the powerful analytic tools initiated by Orr and Sommerfeld is the spectral analysis by studying the Fourier normal mode behavior through the famous Orr-Sommerfeld equation  
derived from the linearization of the incompressible Navier-Stokes equations around a given shear flow profile. As for the stability and instability investigation, tremendous progress has been made since the pioneer work by   Heisenberg, C.C. Lin, Tollmien and Schlichting (see \cite{Hes,lin,SG} and the reference therein), after the fundamental study on the boundary layer around a solid body by Prandtl. The related theories influenced by the wing design for  airplanes  have crucial impact on the development of aerodynamics because of the importance in understanding on the transition from laminar flow to  turbulence that is also related to the separation of  boundary layer. A large number of  literatures have been devoted to the estimation of the critical Reynolds number for stability
based on the Navier-Stokes equations for  various shear flow pattens, such as   Poiseuille flow, Blasius profile, exponential suction profile, etc. Without viscosity, the Orr-Sommerfeld equation is the Rayleigh equation for inviscid flow. 
Hence, Orr-Sommerfeld equation can be viewed as a singular perturbation of Rayleigh equation in high Reynolds number.

In fact, there have been extensive mathematical studies on the construction of growing modes to the Navier-Stokes equation that is related to  the stability of boundary layer profile and validity of Prandtl ansatz. There are two destabilizing mechanisms of the boundary layer.
	\begin{itemize}
		 \item[(a)] Instability at the inviscid level. If the boundary layer profile possesses a spectral instability for the inviscid system, then such instability  persists at the viscous level for small viscosity $\vep\ll1$. In fact, in this case one can show that the unstable eigenvalue for the rescaled viscous system is of order $O(1)$, which leads to a strong ill-posedness of linearized system below analytic regularity \cite{G1} as well as a nonlinear instability in $L^\infty$-space \cite{GN,GN3}.
\item[(b)] Tollmien-Schlichting instability. Even though the boundary layer profile is stable for the inviscid system, there is still a spectral instability at the Navier-Stokes level due to the small viscosity. A recent important work by Grenier-Guo-Nguyen \cite{GGN} gave the precise description on such an unstable
eigenvalue to the classical Orr-Sommerfeld equation by deriving the pointwise bounds on the Green functions
associated to the corresponding Rayleigh and Airy operators. Roughly speaking, they decompose
\begin{equation}\nonumber
\begin{aligned}
\mbox{ Orr-Sommerfeld operator}&= \mbox{Rayleigh operator+ Diffusion operator,}\\
&= \mbox{Airy operator+ Regular operator}.
\end{aligned}
\end{equation}
By careful estimation and iteration, they gave a rigorous description of the unstable mode by showing that
if the Reynolds number $\mbox{Re}$ is sufficiently large, there exists an unstable solution to the rescaled Navier-Stokes equation around any monotone, concave boundary layer profile. The growing mode is localized in the frequency $[\text{Re}^{-\f18},\text{Re}^{-\f1{12}}]$ with growth rate of order $\text{Re}^{-\f18}$, which leads to an ill-posedness for the linearized system in the original variables below Gevrey regularity with index $\f32.$ This
analysis is consistent with the formation of Tollmien-Schlichting wave near the boundary due to the small viscosity
that is well documented in  physical literatures, e.g.  \cite{DR,SG}. See also the asymptotic analysis by Wasow \cite{Wa}. Note that the instability in this case is weaker than the previous one in the sense that the unstable eigenvalue for the rescaled system vanishes in the high Reynold number limit. Therefore, it is more difficult to bootstrap this linear instability to the nonlinear instability. Recently there are several 
important progress in this direction, see \cite{GN1,GN2}, where the instability of order $\text{Re}^{-\f14}$ has been justified in the nonlinear settings.
\end{itemize}

Note that the above instability mechanisms occur in the high tangential frequency regime. Hence,
 if initially data in this high frequency regime has small amplitude so that the growth induced
 by the instability is negligible over the time scale of order $O(1)$, one can still expect the stability of boundary layer profiles and hence the validity of Prandtl ansatz in short time. For the analytic data, the inviscid limit for the Navier-Stokes equation is well-understood and a lot of important progress has been made in this direction. In fact, the verification of the Prandtl ansatz was achieved by Sammartino and Caflisch \cite{SC2}. See also the work \cite{WWZ} by Wang-Wang-Zhang for a new proof based on the energy approach. If the initial vorticity supported away from boundary, Maekawa \cite{M} justified the inviscid limit for 2D Navier-Stokes equation and this result was generalized to 3D in \cite{FTZ}. Remarkably, by obtaining some uniform bounds on vorticity, Nguyen and Nguyen \cite{NN} directly justified the inviscid limit in $L^2$-topology, without using any boundary layer expansion. Very recently, this limit was proved in 2D by Kukavica, Vicol and Wang \cite{KVW} by assuming analyticity only near the boundary. See also \cite{W} for the 3D generalization. If  the boundary layer profile is
 assumed to be monotone and concave, Ger\'ard-Varet-Maekawa-Masmoudi\cite{GMM1} proved the Gevrey stability for Navier-Stokes equation with critical index $\f32$ and then justified the general Prandtl ansatz \cite{GMM2} at the same level of regularity. See also a recent improvement \cite{CWZ}. There is also some important progress for the steady case. We refer to \cite{GM,GI,IM1,IM2} and references therein.  Finally, we refer to \cite{DDGM,DG,GD,LWY} for the instabilities in Prandtl boundary layer model and some other
 physical models.

Back to MHD, it is  interesting to understand the effect of magnetic fields on the hydrodynamic instability mechanisms mentioned above. In fact, it is known that a strong tangential magnetic field can stabilize the boundary layer \cite{GP,LXY2}. For transversal magnetic field, however, there are many numerical experiments \cite{GH,LA,Rock,MA,PK} show that the laminar Hartmann layer  loses its stability through
a  transition to turbulence in the high Reynold number limit. The main purpose of this paper is to investigate the mechanism of instability from the mathematical point of view. In particular, as the first step in this direction, we aim to justify the presence of Tollmien-Schlichting instability at the linear level in the Hartmann layer.

To start with, let us  first derive the Orr-Sommerfeld system for the stream functions of the velocity and magnetic fields. Consider the linearization of the system \eqref{1.2} around a boundary layer profile $(\vec{U}_s, \vec{H}_s)=({U}_s(\frac{y}{\sqrt{\varepsilon}}), 0, {H}_s(\frac{y}{\sqrt{\epsilon}}),1)$:
\begin{align}
\label{1.3}
\left\{
\begin{array}{ll}
\partial_t \vec{U}+ U_s\partial_x\vec{U}+v\partial_yU_s \vec{e}_1
+\nabla p-\varepsilon\triangle \vec{U}-\sqrt{\varepsilon}(H_s\partial_x \vec{H}+g\partial_yH_s\vec{e}_1)
-\sqrt{\varepsilon}\partial_y \vec{H}=\vec{0},\\
\partial_t \vec{H}+U_s\partial_x \vec{H}
+v\partial_yH_s \vec{e}_1-H_s\partial_x \vec{U}-g\partial_y U_s\vec{e}_1
-\partial_y\vec{U}-\sqrt{\varepsilon}\triangle  \vec{H}=\vec{0},\\
\hbox{div} \vec{U}=0,\quad \hbox{div} \vec{H}=0,\quad t>0,\quad (x,y) \in\Omega.
\end{array}
\right.
\end{align}
In coherence with physical literature \cite{LA,Rock,PK}, we supplement \eqref{1.3} with the following no-slip boundary condition on velocity field and perfect conducting wall condition on magnetic field:
\begin{align}\label{bd}
\vec{U}|_{y=0}=(\pa_yh,g)|_{y=0}=\vec{0}.
\end{align}
Notice that under the boundary conditions \eqref{bd}, the divergence free condition of $\vec{H}$ is preserved along \eqref{1.3}. To study the linear system \eqref{1.3}, \eqref{bd}, we introduce the following rescaled variables:
$$\tau=\frac{t}{\sqrt{\vep}}, \, X=\frac{x}{\sqrt{\vep}}, \,  Y=\f{y}{\sqrt{\vep}}.
$$
Then \eqref{1.3} reads
\begin{align}
\label{1.4}
\left\{
\begin{array}{ll}
\partial_\tau \vec{U}+ U_s\partial_X\vec{U}+v\pa_YU_s \vec{e}_1
+\nabla_{X,Y} p-\sqrt{\varepsilon}\triangle_{X,Y} \vec{U}-\sqrt{\varepsilon}(H_s\partial_X \vec{H}+g\pa_YH_s\vec{e}_1)
-\sqrt{\varepsilon}\partial_Y \vec{H}=\vec{0},\\
\partial_\tau \vec{H}+U_s\partial_X \vec{H}
+v\pa_YH_s \vec{e}_1-H_s\partial_X \vec{U}-g\pa_YU_s\vec{e}_1
-\partial_Y\vec{U}-\triangle_{X,Y}  \vec{H}=\vec{0},\\
\hbox{div}_{X,Y} \vec{U}=0,\quad \hbox{div}_{X,Y} \vec{H}=0,\quad t>0,\quad (X,Y) \in\Omega_{\vep}\triangleq\vep^{-\f12}\T\times \mathbb{R}_+,\\
\vec{U}\mid_{Y=0}=(\partial_Y h, g)\mid_{Y=0}=\vec{0}.
\end{array}
\right.
\end{align}
Thanks to divergence free and boundary conditions of both $\vec{U}$ and $\vec{H}$, it is convenient to introduce the following stream functions $\Phi$ and $\Psi$ of $\vec{U}$ and $\vec{H}$ respectively:
$$
\begin{aligned}
\vec{U}&=\nabla_{X,Y}^{\perp}\Phi=(\pa_Y\Phi,-\pa_X\Phi),~\Phi|_{Y=0}=0,\\
\vec{H}&=\nabla_{X,Y}^{\perp}\Psi=(\pa_Y\Psi,-\pa_X\Psi),~\Psi|_{Y=0}=0.
\end{aligned}
$$
We further consider the solutions to linearized system \eqref{1.4} taken the following normal modes form:
\begin{equation}
\begin{aligned}\label{sof}
(\vec{U},\vec{H})=e^{i\a(X-c\tau)}(\pa_Y\Phi,-i\a\Phi,\pa_Y\Psi,-i\a\Psi)(Y),
\end{aligned}
\end{equation}
which lead to the following Orr-Sommerfeld equations:
\begin{equation}\label{eq1.1-1}
\left\{ \begin{aligned}
&-\sqrt{\vep}(\pa_Y^2-\a^2)^2\Phi+i\a(U_s-{c})(\pa_Y^2-\a^2)\Phi-i\a\pa_Y^2U_s\Phi\\
&\qquad\quad-i\a\sqrt{\vep}\left(  H_s(\pa_Y^2-\a^2)\Psi-\pa_Y^2H_s\Psi\right)-\sqrt{\vep}\pa_Y\left(\pa_Y^2-\a^2 \right)\Psi=0,\\
&-(\pa_Y^2-\a^2)\Psi+i\a(U_s-c)\Psi-i\a H_s\Phi-\pa_Y\Phi=0,\\
&\Phi|_{Y=0}=\pa_Y\Phi|_{Y=0}=\Psi|_{Y=0}=\pa_Y^2\Psi|_{Y=0}=0.
\end{aligned}\right.
\end{equation}
From the second equation one has
$$-\pa_Y\left(\pa_Y^2-\a^2 \right)\Psi=\pa_Y^2\Phi+i\a\pa_Y\left( H_s\Phi-(U_s-c)\Psi  \right).
$$
Then plugging it into \eqref{eq1.1-1}, we get the following equivalent system to \eqref{eq1.1-1}:
\begin{equation}\label{eq1.1}
\left\{ \begin{aligned}
&\frac{i}{n}(\pa_Y^2-\a^2)^2\Phi+(U_s-\hat{c})(\pa_Y^2-\a^2)\Phi-\pa_Y^2U_s\Phi\\
&\qquad\quad-\sqrt{\vep}  H_s(\pa_Y^2-\a^2)\Psi+\sqrt{\vep}\pa_Y^2H_s\Psi-\f{\a}{n}\pa_Y\left( (U_s-c)\Psi-H_s\Phi\right)-\f{ i\a^2}{n}\Phi=0,\\
&-(\pa_Y^2-\a^2)\Psi+i\a(U_s-c)\Psi-i\a H_s\Phi-\pa_Y\Phi=0,
\end{aligned}\right.
\end{equation}
with the boundary condition
\begin{equation}\label{BD1}
\Phi|_{Y=0}=\pa_Y\Phi|_{Y=0}=\Psi|_{Y=0}=0.
\end{equation}
Here we have used the notations
 \begin{align}\nonumber
n=\frac{\a}{\sqrt{\vep}}\text{ and }\hat{c}=c+\f{i}{n},
\end{align}
where $n$ is the frequency in the original variable. Note that by restricting the second equation in \eqref{eq1.1} on the boundary $\{Y=0\}$ and using the boundary conditions $\Phi|_{Y=0}=\pa_Y\Phi|_{Y=0}=\Psi|_{Y=0}=0$, the condition $\pa_Y^2\Psi|_{Y=0}=0$ automatically holds. Thus,  we omit it in the formulation \eqref{eq1.1} and \eqref{BD1}.
In the following, without confusion we denote the linear operator of the Orr-Sommerfeld system \eqref{eq1.1} by $\mbox{OS}(\Phi,\Psi)$. In some places we replace $\mbox{OS}(\Phi,\Psi)$ by $\mbox{OS}_c(\Phi,\Psi)$ to emphasize the dependence of $c$. In view of \eqref{sof}, if the system \eqref{eq1.1} together with \eqref{BD1} has a non-trivial solution $(\Psi,\Psi)$ with $\text{Im}c>0$, then the linearized system \eqref{1.4} is spectral unstable.  

With the above notations, we can now state the main result as follows.
\begin{theorem}\label{thm1.1}
	Let $(\vec{U}_s,\vec{H}_s)$ be the Hartmann layer profile given in \eqref{ha}. There is a positive constant $A_0>1$ such that for any $A\geq A_0$, there exists a positive constant $\vep_0\in (0,1)$,  such that for any $\vep\in (0,\vep_0)$ and $\a=A\vep^{\f18}$,  the linearized MHD system \eqref{1.3} together with \eqref{bd} has a solution $(\vec{U},\vec{H})$ in the form of
	\begin{align}\label{sol}
	(\vec{U},\vec{H})=e^{\frac{i\a}{\sqrt{\vep}}(x-c_\vep t)}(\pa_Y\Phi,-i\a\Phi,\pa_Y\Psi,-i\a\Psi)(Y),~Y:=\frac{y}{\sqrt{\vep}}
	\end{align}
	for some $c_\vep\in \mathbb{C}$ with $\a\text{Im}c_\vep\approx\vep^{\f14}.$
	Moreover, the result also holds for $\a=M\vep^{\beta}$ with $M>0$ and $ \beta\in (3/28,1/8).$
\end{theorem}
\begin{remark}
	In terms of  instability, the wave solution \eqref{sol} is localized in the frequency $\f{\a}{\sqrt{\vep}}\in (A_0\vep^{-\f38}, B_0\vep^{-\f{11}{28}})$  
	and it grows exponentially in time with the rate $\f{\a \text{Im}c_\vep}{\sqrt{\vep}}\approx \vep^{-\f14}$. Hence this shows
	the instability of linearized MHD system around the Prandtl-Hartmann layer with critical Gevrey index $\f32$ that is
	the same as the classical Navier-Stokes equations. Therefore, it reveals that the transversal magnetic field to the
	boundary in the Prandtl-Hartmann regime does not enhance the stability of the boundary layer.
	This is in sharp contrast to the stabilizing effect of the tangential magnetic field in the strong nonlinear
	boundary layer regime and the Shercliff regime \cite{GP,LXY2}.
\end{remark}
\begin{remark}
	Let us comment why the critical index for instability is the same as for classical Navier-Stokes equations  in the  model considered in this paper. As observed in  \cite{GP,XY}, the magnetic field creates a  damping term in the reduced boundary layer equation. It is natural to expect that such a feature persists in the full system \eqref{1.2}. In fact, compared with the classical Navier-Stokes system, the Lorenz force in the current scaling contributes an extra term $\f{i}{n}$ in the spectrum parameter $\hat{c}$ that corresponds to the damping effect of magnetic field. However, in view of the leading order dispersion relation \eqref{4.3-1.1}, this damping term is negligible so that it has no
	essential  effect  on the formation of Tollmien-Schlichting wave. 
	Nevertheless, in the full system the coupling of velocity and the magnetic fields complicates the analysis so that it is more involved to justify this intuitive understanding rigorously.
\end{remark}
\begin{remark}
In the proof given in Section 5, all arguments work for {$\beta\in (1/12, 1/8)$}, except for the last inequality in \eqref{7.31} that requires $\beta>3/28$. 
On the other hand, in view of \eqref{7.29} and \eqref{7.30}, the result also holds with
 $\beta\in {(1/12,3/28]}$ by constructing an approximate slow mode that has one order higher accuracy in $\a$ so  that the $L^2$-norm of error terms $E^s_{\beta,1},E^s_{\beta,2}$ is of order $\a^{3+\f12(1+\nu_0)}$ and the weighted $L^2$-norm of $E^{s}_{\beta,3}$ is of order $\a^{\f52}.$ 
 Since we  focus on the essential instability mechanism that appears in the regime with $\beta =\vep^{\f18}$ corresponding to the most unstable mode, we will not give the detailed analysis
   in the regime 
 with {$\beta \in  ({1/12}, {3/28}]$}.
\end{remark}
\begin{remark}
	The above theorem also holds when $U_s$ and $ H_s$  in the background solution  satisfy the
	following structural assumptions:
	\begin{itemize}
		\item 
		$U_s\in C^3(\overline{\mathbb{R}_+})$ and $H_s\in C^2(\overline{\mathbb{R}_+})$ and they satisfy
		\begin{align}\nonumber
		U_s(0)=0, ~\lim_{Y\rightarrow \infty}U_s(Y)=1,~ U_s'(0)=1,~\lim_{Y\rightarrow \infty}H_s(Y)=h_{\infty}.
		\end{align}
		\item (Monotonicity) There exist positive constants $s_0$, $s_1$ and $s_2$ such that \begin{align}
		s_1e^{-s_0Y}\leq \pa_YU_s(Y)\leq s_2e^{-s_0Y}~ \text{for any }Y>0.\label{BL2}
		\end{align}
		\item (Strongly concave condition) There exists positive constant $\sigma_0$ such that for any $Y>0$,
		\begin{align}\label{BL}
		-\sigma_0 \pa_Y^2U_s\geq (\pa_YU_s)^2,~ \sup_{Y\geq 0}\left(\left|\frac{\pa_Y^3U_s}{\pa_Y^2U_s}\right|+\left|\frac{\pa_Y^2U_s}{\pa_YU_s}\right|+\left|\frac{\pa_Y^2H_s}{\pa_YU_s}\right|+\left|\frac{\pa_YH_s}{\pa_YU_s}\right|+\left|\frac{1-U_s}{\pa_YU_s}\right|\right)\leq \sigma_0.
		\end{align}
	\end{itemize}
\end{remark}

In the following, we will briefly present the key ingredients in the proof  of Theorem \ref{thm1.1}
and the novelty of the new approach. For illustration, we only consider the problem in the regime
with  $\a\sim\vep^{\f18}$.\\

{\bf Strategy and novelty:}

\begin{itemize}
	\item First of all, even though the equations on the stream functions for the velocity and magnetic fields are coupled, one can still estimate the stream function of the magnetic field in terms of the stream function of the velocity field by a given boundary condition.
	\item To investigate the instability 
	induced by the Tollmien-Schlichting wave, as for the Navier-Stokes equations, we can start with an approximate solution to the Orr-Sommerfeld system by constructing the slow mode (inviscid mode) and fast mode (viscous mode). One can observe that in the Prandtl-Hartmann regime, the terms involving magnetic stream function $\Psi$ in the first equation of \eqref{eq1.1} is formally of high order. Therefore, it is permissible to construct these modes near those of classical Orr-Sommerfeld system for Navier-Stokes system constructed in \cite{GGN}. The interaction of these two modes by a linear combination gives an approximate growing mode solution $(\Phi_{\text{app}}(Y;c),\Psi_{\text{app}}(Y;c))$ to the system \eqref{eq1.1} with zero boundary conditions $\Phi_{\text{app}}(0;c)=\Psi_{\text{app}}(0;c)=0$ on the stream functions. Then to recover the no-slip boundary condition, we obtain the dispersion relation $\Gamma_0(c)=0$, where $\Gamma_0(c)=\pa_Y\Phi_{\text{app}}(0;c)$ is the boundary data of the first order derivative of approximate solution, see \eqref{4.3-1.1}. A formal expansion of $\Gamma_0(c)$ gives the asymptotic behavior of its zero point: 
		\begin{align}\nonumber
		c\approx (A+A^{-1}e^{\f14 \pi i})\vep^{\f18}+O(1)A^{-2}\vep^{\f18}+C_{A}\vep^{\f14}|\log \vep|, \text{ as }\vep\rightarrow 0^+,
		\end{align}
	for suitably large but fixed $A>1$.
	We will apply Rouch\'e's theorem to justify this formal expansion.

{To use Rouch\'e's Theorem in finding the critical point for instability of fluid mechanics equation can be traced back to the early paper\cite{Mo} and also the recent works  \cite{DDGM,DG} on some boundary layer models. In the present paper,} given that the underlying unstable eigenvalue vanishes in the high Reynold number limit, the instability under consideration is of  different nature compared with \cite{DDGM,DG,Mo}. The novelty therefore comes from a suitable choice of loops that shrink to the original point as $\vep\rightarrow 0^+$ and in the meantime surround the possible zero points of $\Gamma_0(c)$. Then by establishing some sharp pointwise estimates of slow and fast modes, we are able to show that $\Gamma_0(c)$ is approximated by its linearization around the leading part of the unstable eigenvalue. This concludes that $\Gamma_0$ has a simple root. Moreover, we can show that  $\Gamma_0(c)$ has a lower bound of order $O(1)$ on these loops, which survives in the limit $\vep\rightarrow 0^+.$ 
This idea will be crucially used to justify the existence of exact growing mode for the full original system.
\item The next step is to show the solvability of Orr-Sommerfeld system \eqref{eq1.1} in order to obtain an exact growing mode near the approximate one. The key point is to obtain some energy estimate for the resolvent problem. Using the idea introduced in \cite{CLWZ,CWZ,GM,GMM2}, we solve the resolvent problem by imposing the Navier-slip boundary condition on velocity field, so that the multiplier $\frac{\omega}{\pa_Y^2U_s}$ can be used to obtain a weighted $L^2$-estimate $\|\f{\omega}{|\pa_Y^2U_s|^{\f12}}\|_{L^2}$ on vorticity. One of the new observations used in this step is that due to the favorable boundary condition, the stability estimates obtained in \cite{CWZ,GMM2} can be extended to Gevrey function spaces with arbitrarily large index,
at least for strongly concave shear flow, cf.  Lemma \ref{lem6.1} and Remark \ref{rmk4.4}. The motivation of studying the resolvent estimates in this step is that we can control the $H^2$-norm of the remainder in terms of the  $L^2$-norms of the error terms generated by the slow and fast modes, together with a small factor in the  order of $O(\vep^{\f{1}{16}})$ when compared with $L^\infty$-norms. This small factor 
is used to obtain the desired results without  higher order expansion.
\item $\text{OS}_d-\text{OS}_s$ iteration: Above procedure however is not enough to close the resolvent estimate. A difficulty arises from the magnetic coupling. In fact, the terms coming from the Lorentz force exhibit a slow decay in $Y$ as $e^{-\a^{\f12}Y}$, which are not compatible with above weighted estimate.  To overcome this, we first observe that all of these slow decay terms in the equation of velocity stream function $\Phi$ are in the form of total derivatives. This observation leads us to the introduction of
 a new decomposition of the solution operator by taking account of the strong decay property of the background solution and the differential structure in the operator together with the smallness in the frequency regime under consideration. Then the exact solution to the Orr-Sommerfeld system can be obtained by a series of approximate solutions through interation. 
To illustrate this, note that we can have two ways to represent $\mbox{OS}$. The first one is 
\begin{equation}\label{OS1}
\mbox{OS}=\mbox{OS}_d+L_d,
\end{equation}
where $\mbox{OS}_d$ is defined later in \eqref{6-1.1} and it comes from \eqref{eq1.1} by retaining only Navier-Stokes part in the first equation. Notice that $\mbox{OS}_d$ in principle behaves like a regular viscous perturbation of Rayleigh's equation so that the weighted estimate on the vorticity can be obtained by applying the trick in \cite{G}. However, the remaining term $L_d(\cdot)\triangleq(\pa_YR_1(\cdot)+i\a R_2(\cdot),0)$ where $R_1$ and $R_2$ are defined in \eqref{6-2.2} that  can not be regarded as a source term of $\mbox{OS}_d$ due to slow decay of magnetic field. Hence,  we introduce another decomposition
\begin{equation}\label{OS2}
\mbox{OS}=\mbox{OS}_s+L_s,
\end{equation}
where $\mbox{OS}_s$ is defined in \eqref{6-2.1} corresponding to the divergence form of \eqref{eq1.1}. On one hand, due to the elliptic structure of $\mbox{OS}_s$, the remaining term $L_d$ in \eqref{OS1} is a suitable source term of $\mbox{OS}_s$. On the other hand, the remaining term $L_s$ in \eqref{OS2} has a strong decay in $Y$ as $e^{-Y}$ due to the background shear profile. Therefore it matches the operator $\mbox{OS}_d.$ With this decomposition, as an analogy of the celebrated Rayleigh-Airy iteration introduced in \cite{GGN,GMM1}, the solvability of $\mbox{OS}$ can be justified by constructing a series of
	solutions to the approximated operators $\mbox{OS}_s$ and $\mbox{OS}_d$. The iteration scheme as well as its convergence will be given in detail in Section 4.3. Finally, by the virtue of these resolvent estimates and some detailed estimates on error terms, we will prove that the $L^\infty$-norm of first order derivative of remainder $\|\pa_Y\Phi_R(\cdot~;c)\|_{L^\infty}$ is of order $\vep^{\f{1}{16}}$   uniformly in the region surrounded by the loops chosen suitably. Then we conclude the Theorem \ref{thm1.1} by Rouch\'e's Theorem. 
	\item Note that for the classical Orr-Sommerfeld equation derived from the incompressible Navier-Stokes equations, the operator $\mbox{OS}=\mbox{OS}_d$ so that it only takes two iteration to complete the solvability analysis. Hence, it provides another approach to the analysis on
	the instability   mechanism.
\end{itemize}


The rest of paper is organized as follows. In Section 2, we will  give estimates on the magnetic field with a given velocity field. We will then construct the approximate growing mode in Section 3. Section 4 is devoted to 
	the solvability of the Orr-Sommerfeld system together with the Navier slip boundary condition in order to resolve the remainder due to the approximation, and then  complete  the proof of Theorem \ref{thm1.1} for $\a=A\vep^{\f18}$. In Section 5, we will give the proof of Theorem \ref{thm1.1} for $\a\sim \vep^{\beta}$ with $\beta\in (3/28,1/8)$. In the Appendix, we list some properties of the Airy functions and some 
	useful estimates. We will also give a lemma of analyticity there.

{\bf Convention:} In the whole paper, for any $z\in \mathbb{C}\setminus \mathbb{R}_-$,  we take the principle analytic branch of $\log z$ and $z^k, k\in (0,1)$, i.e.
$$\log z\triangleq Log|z|+i\text{Arg}z,~ z^k\triangleq|z|^ke^{ik \text{Arg} z},~\text{Arg}z\in (-\pi,\pi].$$

{\bf Notations:} Throughout this paper, we denote by $C$ the generic positive constant and by $C_a,C_b,\cdots$ the generic positive constants depending on $a,~b,\cdots$ respectively. These constants may vary from line to line.  We say $A\lesssim B$ if there exists a generic constant $C$ such that $A\leq CB$, and $A\lesssim_\eta B$ if such a constant $C$ depends on $\eta$. Similarly, $A\gtrsim B$ means that there exists a positive constant $C$ such that $A\geq CB$, and $A\gtrsim_\eta B$ means that such a constant $C$ depends on $\eta$. Moreover, we use notation $A\sim B$ if $A\lesssim B$ and $A\gtrsim B$ and use notation $A\sim_\eta B$ if $A\lesssim_\eta B$ and $A\gtrsim_\eta B$. We denote by $\|\cdot\|_{L^2}$ the standard $L^2(\mathbb{R}_+)$-norm and by $\|\cdot\|_{L^\infty}$  the $L^\infty(\mathbb{R}_+)$-norm. For any $\eta>0$, we define $L^\infty_\eta(\mathbb{R}_+)$ as the weighted $L^\infty$-space with the norm  $\|f\|_{L^\infty_\eta}\triangleq \sup_{Y\in \mathbb{R}_+}\left|e^{\eta Y}f(Y)\right|$. For any nonnegative integer $k$, we set $W^{k,\infty}_\eta(\mathbb{R}_+)=\{f\in L^{\infty}_\eta(\mathbb{R}_+)|\pa_Y^jf(Y)\in L^\infty_\eta(\mathbb{R}_+),~j=1,2,\cdots,k  \}$.

\section{Estimation on magnetic field}
In this section, we will study the following magnetic equation with Dirichlet boundary condition 
\begin{equation}\left\{
\begin{aligned}\label{3.1}
&-(\pa_Y^2-\a^2)\varphi+i\a(U_s-c)\varphi=f, ~Y>0,\\
&\varphi|_{Y=0}=\varphi_b,
\end{aligned}\right.
\end{equation}
where $\varphi_b$ is a given constant  and $f$ is a source.
\begin{proposition}\label{prop3.1}
	There exist positive constants $\a_0\in (0,1)$ and $\gamma_0\in (0,1)$ such that
if $\a\in (0,\a_0)$ and $c$ lies in the disk $D_{\gamma_0}=\{c\in \mathbb{C}\mid |c|< \gamma_0\}$, then for any $\eta\in (0,\frac{\sqrt{2\a}}{4})$ and $f\in L^\infty_\eta(\mathbb{R}_+),$ there exists a unique solution $\varphi\in W^{2,\infty}_\eta(\mathbb{R}_+)$ satisfying
\begin{equation}
\begin{aligned}\label{3.2}
\|\varphi\|_{L^\infty_\eta}&\lesssim |\varphi_b|+\a^{-1}\|f\|_{L^\infty_\eta},\\
\|\pa_Y\varphi\|_{L^\infty_\eta}&\lesssim \a^{\f12}|\varphi_b|+\a^{-\f12}\|f\|_{L^\infty_\eta},\\
\|\pa_Y^2\varphi\|_{L^\infty_\eta}&\lesssim \a|\varphi_b|+\|f\|_{L^\infty_\eta}.
\end{aligned}	
\end{equation}
Moreover, if both $\varphi_b$ and $f$ are holomorphic in  $c$ in $D_{\gamma_0}$, then $\varphi$ is holomorphic in $D_{\gamma_0}$.
\end{proposition}
\begin{proof}
Noting that $U_s\rightarrow 1$ as $Y\rightarrow +\infty$,	
 we rewrite \eqref{3.1} as
\begin{equation}\left\{
\begin{aligned}
&-\pa_Y^2\varphi+(i\a+\a^2-i\a c)\varphi=f+i\a(1-U_s)\varphi, ~Y>0,\\
&\varphi|_{Y=0}=\varphi_b.
\end{aligned}\right.\nonumber
\end{equation}
Observe that $i\a+\a^2-i\a c\sim i\a$ for small $\a,c$.
Let $\xi\triangleq \sqrt{i\a+\a^2-i\a c},$ where the root is taken so that $\text{Re}\xi>0.$ Then  $\text{Re}\xi\sim \text{Re}\left(e^{\f{\pi i}{4}}\a^{\f12}\right)=\frac{\sqrt{2\alpha}}{2}.$ Thus there exist positive constants $\a_0$ and $\gamma_0$, such that for $\a\in (0,\a_0)$ and $c\in D_{\gamma_0}$,  $\f{\sqrt{2\a}}{3}<\text{Re}\xi<\frac{2\sqrt{2\alpha}}{3}$
holds. Now we decompose
\begin{align}\label{3.4}
\varphi(Y)=e^{-\xi Y}\varphi_b+\tilde{\varphi}(Y),
\end{align}
where $\tilde{\varphi}(Y)$ satisfies the following equation
\begin{equation}\left\{
\begin{aligned}\label{3.5}
&-\pa_Y^2\tilde{\varphi}+\xi^2\tilde{\varphi}=\tilde{f}+i\a(1-U_s)\tilde{\varphi}, ~Y>0,\\
&\tilde{\varphi}|_{Y=0}=0,
\end{aligned}\right.
\end{equation}
with $\tilde{f}\triangleq f+i\a(1-U_s)e^{-\xi Y}\varphi_b.$ To  solve \eqref{3.5}, we introduce the following iteration scheme:
\begin{equation}\left\{
\begin{aligned}
&-\pa_Y^2\tilde{\varphi}^{k+1}+\xi^2\tilde{\varphi}^{k+1}=\tilde{f}+i\a(1-U_s)\tilde{\varphi}^k, ~Y>0,\\
&\tilde{\varphi}^{k+1}|_{Y=0}=0,~ \tilde{\varphi}^0\equiv0.\nonumber
\end{aligned}\right.
\end{equation}
The sequence $\{\tilde{\varphi}^k\}_{k=0}^\infty$ can be solved inductively by the following mild formulation:
\begin{align}\label{3.7}
\tilde{\varphi}^{k+1}(Y)=\int_0^Ye^{-\xi(Y-Y')}\dd Y'\int_{Y'}^\infty e^{\xi(Y'-Y'')}\left( \tilde{f}+i\a(1-U_s)\tilde{\varphi}^k  \right)(Y'') \dd Y''.
\end{align}
For any $\eta\in (0,\frac{\sqrt{2\a}}{4})$, using the decay property of boundary layer profile $|1-U_s(Y)|\lesssim e^{-s_0Y}$, we have
\begin{equation}
\begin{aligned}
|\tilde{\varphi}^{k+1}(Y)|&\lesssim \int_0^Ye^{-\text{Re}\xi(Y-Y')}\dd Y'\int_{Y'}^\infty e^{\text{Re}\xi(Y'-Y'')}e^{-\eta Y''}\left(\|\tilde{f}\|_{L^\infty_\eta}+\a e^{-s_0Y''}\|\tilde{\varphi}^k\|_{L^\infty_\eta}   \right)\dd Y''\\
&\lesssim\int_0^Y e^{-\text{Re}\xi(Y-Y')-\eta Y'}\times \left(  \frac{\|\tilde{f}\|_{L^\infty_\eta}}{\text{Re}\xi+\eta} + \frac{\a\|\tilde{\varphi}^k\|_{L^\infty_\eta}}{\text{Re}\xi+\eta+s_0}e^{-s_0Y'} \right)\dd Y'\\
&\lesssim \frac{e^{-\eta Y}}{(\text{Re}\xi)^2-\eta^2}\|\tilde{f}\|_{L^\infty_\eta}+\a e^{-\eta Y }\|\tilde{\varphi}^k\|_{L^\infty_\eta}\int_0^Ye^{-(s_0+\eta-\text{Re}\xi)Y'}\dd Y'\\
&\leq C\alpha^{-1}e^{-\eta Y}\|\tilde{f}\|_{L^\infty_\eta}+C\alpha e^{-\eta Y}\|\tilde{\varphi}^k\|_{L^\infty_\eta}.\label{3.7-1}
\end{aligned}
\end{equation}
Here we have used the fact that $(\text{Re}\xi)^2-\eta^2\geq \f{2\a}{9}-\f{\a}{8}>\f{\a}{12}$. Taking $\alpha_0$  smaller if necessary, we obtain from \eqref{3.7-1} the following uniform-in-$k$ estimate:
$$\|\tilde{\varphi}^{k}\|_{L^\infty_\eta}\le 2C\a^{-1}\|\tilde{f}\|_{L^\infty_\eta},~k=0,1,\cdots.
$$
Applying the similar estimate to $\tilde{\varphi}^{k+1}-\tilde{\varphi}^k$ gives
$$\|\tilde{\varphi}^{k+1}-\tilde{\varphi}^k\|_{L^\infty_\eta}\leq C\a \|\tilde{\varphi}^{k}-\tilde{\varphi}^{k-1}\|_{L^\infty_\eta}\leq \f12\|\tilde{\varphi}^{k}-\tilde{\varphi}^{k-1}\|_{L^\infty_\eta}.
$$
This implies that $\{\tilde{\varphi}^k\}_{k=0}^\infty$ is a Cauchy sequence in $L^\infty_\eta$ that has a limit $\tilde{\varphi}.$ Taking the limit $k\rightarrow\infty$ in \eqref{3.7}, we have that $\tilde{\varphi}$ is a solution to \eqref{3.5} and $\tilde{\varphi}$ satisfies $$
\|\tilde{\varphi}\|_{L^\infty_\eta}\lesssim \a^{-1}\|\tilde{f}\|_{L^\infty_\eta}\lesssim \a^{-1}\|{f}\|_{L^\infty_\eta}+|\varphi_b|.
$$
Note that 
$$\pa_Y\tilde{\varphi}(Y)=-\xi \tilde{\varphi}(Y)+\int_Y^\infty e^{\xi(Y-Y')}\left( \tilde{f}+i\a(1-U_s)\tilde{\varphi}\right)(Y')\dd Y',
$$
which implies
$$\|\pa_Y\tilde{\varphi}\|_{L^\infty_\eta}\lesssim \left(|\xi|+\a\right)\|\tilde{\varphi}\|_{L^\infty_\eta}+\frac{1}{\text{Re}\xi+\eta}\|\tilde{f}\|_{L^\infty_\eta}\lesssim \a^{-\f12}\|f\|_{L^\infty_\eta}+\a^{\f12}|\varphi_b|.
$$
By using the equation \eqref{3.5}, we have
$$\|\pa_Y^2\tilde{\varphi}\|_{L^\infty_\eta}\lesssim \a\|\tilde{\varphi}\|_{L^\infty_\eta}+\|\tilde{f}\|_{L^\infty_\eta}\lesssim \|f\|_{L^\infty_\eta}+\a|\varphi_b|.
$$
We therefore conclude the estimate \eqref{3.2} by noting \eqref{3.4}.
Moreover, for $|c|\leq \gamma_0$ with $\gamma_0$ being sufficiently small, $\text{arg}(i\a+\a^2-i\a c)\sim \f{\pi}{2}$, which implies that $\xi$ is holomorphic in $c$. Then  for any loop $\Gamma$ inside $D_{\gamma_0}$, by using the analyticity of $f$ and $\varphi_b$ we can show from \eqref{3.7} inductively that $\int_\Gamma \tilde{\varphi}^{k}(Y;c)\dd c=0$. Since the convergence is uniform, we conclude that $\int_\Gamma \tilde{\varphi}(Y;c)\dd c=0$. Therefore, by Morera's Theorem we know that $\tilde{{\varphi}}(Y;c)$ is analytic and so is ${{\varphi}}(Y;c)$. The proof of Proposition \ref{prop3.1} is then completed.
\end{proof}
Next we will give some $L^2$-type estimates on the solution of \eqref{3.1} with homogeneous boundary data $\varphi_b=0$ for use in Section 4.
\begin{proposition}\label{prop3.2}
	There exist positive constants $\a_1\in (0,1)$ and $\g_1\in (0,1)$, such that if $\a\in (0,\a_1)$ and $c$ lies in the disk $D_{\gamma_1}=\{c\in\mathbb{C}\mid |c|< \gamma_1 \}$,
then for any $f\in L^2(\mathbb{R}_+)$, equation \eqref{3.1} with $\varphi_b=0$ has a unique solution $\varphi\in H^2(\mathbb{R}_+)\cap H^1_0(\mathbb{R}_+)$ and $\varphi$ satisfies 
	\begin{align}
	\|(\pa_Y\varphi,~\alpha\varphi)\|_{L^2}+\alpha^{\frac{1}{2}}\|\varphi\|_{L^2}&\lesssim\alpha^{-\frac12}\|f\|_{L^2}, \label{est-mag1}\\
	\|(\pa_Y^2-\alpha^2)\varphi\|_{L^2}&\lesssim\|f\|_{L^2}.\label{est-mag2}
		\end{align}
Moreover, $\varphi(\cdot~;c)$ is analytic in $c$ with value in $H^2(\mathbb{R}_+).$
\end{proposition}

\begin{proof}[Proof.]
	The solvability of elliptic problem \eqref{3.1} can be obtained via the similar idea in the proof of Proposition \ref{prop3.1}. We only show the a priori estimates. Taking the $L^2$-inner product on both sides of \eqref{3.1} with $\varphi$, we obtain by integration by parts that
	\begin{align}\label{est-mag1_0}
	\|(\pa_Y\varphi,~\alpha\varphi)\|_{L^2}^2+i\alpha\int_0^\infty(U_s-c)|\varphi|^2\dd Y=\langle f, \varphi\rangle. 
	\end{align}
	We take the real and imaginary part of \eqref{est-mag1_0} respectively to get
	\begin{align}\label{est-mag1-re}
	\|(\pa_Y\varphi, \alpha\varphi)\|_{L^2}^2+\alpha \text{Im}c\|\varphi\|_{L^2}^2&=\text{Re}\left(\langle f, \varphi\rangle\right)\leq |\langle f, \varphi\rangle|,
	\end{align}
	and
	\begin{align}\label{est-mag1-im}
	\alpha\int_{0}^{\infty}(U_s-\text{Re}c)|\varphi|^2\dd Y=\text{Im}\left(\langle f, \varphi\rangle\right)\leq |\langle f, \varphi\rangle|.
	\end{align}
Let $\gamma_1\in(0,\f12)$. Then for any $|c|<\gamma_1$,	by using the fact $\|Y^2(1-U_s)\|_{L^\infty}<\infty$, the left-hand side of \eqref{est-mag1-im} yields 
	\begin{align}\label{est-mag1_2}
	\begin{aligned}
	\alpha\int_{0}^{\infty}(U_s-\text{Re}c)|\varphi|^2\dd Y&=\alpha\int_{0}^{\infty}(1-\text{Re}c)|\varphi|^2\dd Y+\alpha\int_{0}^{\infty}(U_s-1)|\varphi|^2 \dd Y\\
	&\geq\alpha(1-\text{Re}c)\|\varphi\|_{L^2}^2-\alpha\|Y^2(1-U_s)\|_{L^\infty}\|Y^{-1}\varphi\|_{L^2}^2\\
	&\geq\alpha(1-\text{Re}c)\|\varphi\|_{L^2}^2-C\alpha\|\pa_Y\varphi\|_{L^2}^2,
	\end{aligned}\end{align}
	where we have used the Hardy inequality in the last inequality.
	We plug \eqref{est-mag1_2} into  \eqref{est-mag1-im}, and combine it with \eqref{est-mag1-re} to obtain
	\begin{align}
	\|(\pa_Y\varphi,~\alpha\varphi)\|_{L^2}^2+\alpha(1-|c|)\|\varphi\|_{L^2}^2\lesssim\alpha\|\pa_Y\varphi\|_{L^2}^2+|\langle f, \varphi\rangle|\lesssim\alpha\|\pa_Y\varphi\|_{L^2}^2+\|f\|_{L^2}\|\varphi\|_{L^2}.\nonumber
	\end{align}
	Then by   the smallness of $\alpha_1$, this  implies
	\begin{align*}
	\|(\pa_Y\varphi,~\alpha\varphi)\|_{L^2}^2+\alpha\|\varphi\|_{L^2}^2\lesssim\alpha^{-1}\|f\|_{L^2}^2,~ \text{for any }\a\in (0,\a_1).
	\end{align*}
	Thus we obtain \eqref{est-mag1}.	Then, by \eqref{3.1} and \eqref{est-mag1} it is straightforward to obtain  
	\begin{align*}
	\|(\pa_Y^2-\alpha^2)\varphi\|_{L^2}&\leq\alpha\|(U_s-c)\varphi\|_{L^2}+\|f\|_{L^2}\lesssim\|f\|_{L^2}
	\end{align*}
	that is  \eqref{est-mag2}. Finally, we will show the analyticity. Define an operator
	$$\begin{aligned}
	&T: H^2(\mathbb{R}_+)\cap H^1_0(\mathbb{R}_+)\mapsto L^2(\mathbb{R}_+)\\
	&T(\varphi)=\frac{i}{\a}(\pa_Y^2-\a^2)\varphi+U_s\varphi.
	\end{aligned}
	$$
	We have shown that $D_{\gamma_1}$ belongs to the resolvent set of $T$. From \eqref{est-mag1} and \eqref{est-mag2}, we obtain that
	$$\|(-c+T)^{-1}f\|_{H^2}\leq C\a^{-1}\|i\a f\|_{L^2}\leq C\|f\|_{L^2}.
	$$
Therefore, for any $c_0\in D_{\gamma_1}$ and any $|c-c_0|<\min\{\f1{2C},\f12 d(c_0,\pa D_{\gamma_1})  \}$, we can write 
$$\varphi(c)=(-c+T)^{-1}f=\sum_{n=0}^\infty (c-c_0)^n(-c_0+T)^{-n-1}f,
$$
which is absolutely convergent in $H^2$. Therefore, $\varphi$ is analytic in $c$ and the proof of Proposition \ref{prop3.1} is completed.		

\end{proof}
\section{Approximation of unstable modes}
\subsection{Slow mode}
In this subsection, we will construct the slow mode $\Phi_{\text{app}}^s(Y;{c})$ for \eqref{eq1.1} that captures the inviscid motion in
the fluid. This mode can be approximated by a solution to the Rayleigh equation:
\begin{align}\label{4.1.1}
\text{Ray}_\a(\psi)\triangleq(U_s-\hat{c})(\pa_Y^2-\a^2)\psi-\pa_Y^2U_s\psi=0,~Y>0,~\hat{c}=c+\f{i}{n}.
\end{align}
The solution to \eqref{4.1.1} has been well studied, for example in \cite{GGN} for small $\a$. For completeness, 
we list some key points in the construction and one can refer to \cite{GGN} for more details. 
For $\a=0$ and Im$c>0$, \eqref{4.1.1} has two independent solutions
\begin{align}\label{4.1.2}
\psi_{0,1}(Y;c)=U_s(Y)-\hat{c},~ \psi_{0,2}(Y;c)=(U_s(Y)-\hat{c})\int_1^Y\f{1}{(U_s(X)-\hat{c})^2}\dd X.
\end{align}
For $\a>0$ but small, the corresponding approximate solutions to \eqref{4.1.1} can be given by
\begin{align}
\psi_{\a,j}(Y)=e^{-\a Y}\psi_{0,j}(Y),~j=1,2.\nonumber
\end{align}
It is straightforward to check that these two solutions satisfy
\begin{align}
\text{Ray}_\a(\psi_{\a,j})=-2\a(U_s-\hat{c})\pa_Y\psi_{0,j}e^{-\a Y},~j=1,2.\nonumber
\end{align}
To obtain a better approximation, as \cite{GGN}, one can define the approximate Green's kernel by
\begin{equation}
\label{4.1.4}
G_{\a}(X,Y)\triangleq -(U_s(X)-\hat{c})^{-1}\left\{ \begin{aligned}
&e^{-\a(Y-X)}\psi_{0,1}(Y)\psi_{0,2}(X),~Y>X,\\
&e^{-\a(Y-X)}\psi_{0,1}(X)\psi_{0,2}(Y),~Y<X,
\end{aligned}
\right.
\end{equation}
and introduce the corrector 
\begin{align}\label{4.1.4-1}
\Phi_1^s(Y;{c})\triangleq2\int_0^\infty G_\a(X,Y)(U_s-\hat{c})\pa_Y\psi_{0,1}(X)e^{-\a X}\dd X
\end{align}
to absorb $O(\alpha)$ error generated by $\psi_{\a,1}$. Thus, the slow mode $\Phi_{\text{app}}^s(Y;{c})$ is given by
\begin{align}\label{4.1.5}
\Phi_{\text{app}}^s(Y;{c})=\psi_{\a,1}(Y;{c})+\a \Phi_{1}^s(Y;{c}).
\end{align}
A straightforward calculation yields that 
\begin{equation}
\begin{aligned}
\text{Ray}_\a(\Phi_{\text{app}}^s)=&-2\a^2(U_s-\hat{c})(\pa_Y\Phi_{1}^s+\a\Phi_1^s)\\
=&4\a^2\psi_{\a,1}\pa_Y\psi_{0,1}\int_0^YU_s'(X)\psi_{0,2}(X)\dd X+4\a^2\psi_{\a,1}\pa_Y\psi_{0,2}\int_Y^\infty U_s'(X)\psi_{0,1}(X)\dd X.\label{4.1.5-1}
\end{aligned}
\end{equation}
From \eqref{4.1.5-1}, $\Phi_{\text{app}}^s$ formally solves \eqref{4.1.1} up to $O(\a^2)$. In the following lemma, we state some properties  and estimates on  the boundary data of the slow mode $\Phi_{\text{app}}^s$ which will be used later.
\begin{lemma}\label{lem4.1}
	1. For each $Y\geq0$, $\Phi_{\text{app}}^s(Y;{c})$ is holomorphic  in the upper-half complex plane $\{c\in\mathbb{C}\mid \text{Im}c>0  \}$.\\
	2. At $Y=0$, we have
	\begin{align}
	\Phi_{\text{app}}^s(0;{c})&=-\hat{c}-\a \psi_{0,2}(0)(1-2\hat{c}),\label{4.1.6}\\
	\pa_Y\Phi_{\text{app}}^s(0;{c})&=1+\a \hat{c}+\a(1-2\hat{c})(\a \psi_{0,2}(0)-\pa_Y\psi_{0,2}(0)).\label{4.1.6-1}
	\end{align}
Moreover, there exists a positive constant $\gamma_2\in (0,\tilde{\gamma})$ where $\tilde{\gamma}$ is given in Lemma \ref{lem5.3}, such that if $c\in \{\text{Im}c>0,~|\hat{c}|\leq \gamma_2\}$, we have
\begin{align}
	\Phi_{\text{app}}^s(0;{c})&=-\hat{c}+\a+O(1)\a |\hat{c}\log \text{Im}\hat{c}|,\label{4.1.7}\\
	\pa_Y\Phi_{\text{app}}^s(0;{c})&=1+O(1)\a | \log\text{Im}\hat{c}|\label{4.1.7-1}.
\end{align}
\end{lemma}
\begin{proof}
Note that $\hat{c}=c+\frac{i}{n}$. For $c$ lies in the upper-half plane, the denominators in the formula \eqref{4.1.2} and \eqref{4.1.4} do not vanish. This gives the analyticity of $\Phi_{\text{app}}^{s}(Y;{c})$. To show \eqref{4.1.6}, we compute
\begin{align}
\Phi_1^s(Y;{c})=-2\psi_{0,1}(Y;{c})e^{-\a Y}\int_0^YU_s'(X)\psi_{0,2}(X)\dd X-2\psi_{0,2}(Y;{c})e^{-\a Y}\int_Y^\infty U_s'(X)\psi_{0,1}(X)\dd X.\label{4.1.8}
\end{align}
Taking \eqref{4.1.8} at $Y=0$ gives
\begin{align}
\Phi_{1}^s(0;{c})=-2\psi_{0,2}(0)\int_0^\infty(U_s(X)-\hat{c})U_s'(X)\dd X=-\psi_{0,2}(0)(U_s(X)-\hat{c})^2\big|_{X=0}^{X=\infty}=-\psi_{0,2}(0)(1-2\hat{c}).\nonumber
\end{align}
Thus we have $\Phi_{\text{app}}^s(0;c)=\psi_{\a,1}(0;c)+\a\Phi_{1}^s(0;c)=-\hat{c}-\a\psi_{0,2}(0)(1-2\hat{c})$, which is \eqref{4.1.6}. Moreover, for $Y\in [0,1]$, $U_s'(Y)$ does not vanishes. Thus  integration by parts yields that
$$\begin{aligned}
\psi_{0,2}(0)&=\hat{c}\int_0^1 \f{1}{(U_s(X)-\hat{c})^2}\dd X=
\hat{c}\int_0^1\f{\dd}{\dd Y'}\left(\frac{-1}{U_s(X)-\hat{c}}\right)\frac{1}{U_s'(X)}\dd X\\
&=-\f{1}{U_s'(0)}-\frac{\hat{c}}{U_s'(1)(U_s(1)-\hat{c})}-\hat{c}\int_0^1\frac{1}{U_s(X)-\hat{c}}\frac{U_s''(X)}{(U_s'(X))^2}\dd X.
\end{aligned}
$$
Since $U_s'(0)=1$, the first term is equal to $-1$. For $|\hat{c}|\leq \frac{U_s(1)}{2}$, we have $|U_s(1)-\hat{c}|\geq \f{U_s(1)}{2}.$ So the second term satisfies
$$\left|\frac{\hat{c}}{U_s'(1)(U_s(1)-\hat{c})}
\right|\leq \frac{2|\hat{c}|}{U_s'(1)U_s(1)}\lesssim|\hat{c}|.$$
For the last term, by using \eqref{5.8-2}, we obtain for $|\hat{c}|<\tilde{\gamma}$ that
$$\left|\hat{c}\int_0^1\frac{1}{U_s(X)-\hat{c}}\frac{U_s''(X)}{(U_s'(X))^2}\dd X\right|\lesssim |\hat{c}|\int_0^1\f{1}{|U_s(x)-\hat{c}|}\dd X\lesssim |\hat{c}\log \text{Im}\hat{c}|.$$
Let $\gamma_2=\min\{\frac{U_s(1)}{2}, \tilde{\g} \}$. Then for any $c\in\{\text{Im}{c}>0,~|\hat{c}|\leq \g_2\}$, we have
$\psi_{0,2}(0)=-1+O(1)|\hat{c}\log \text{Im}\hat{c}|.
$
Substituting it into \eqref{4.1.6}, we  obtain \eqref{4.1.7}. By using the explicit formula \eqref{5.6} and integral estimate in Lemma \ref{lem5.3}, the estimates \eqref{4.1.6-1} and \eqref{4.1.7-1} can be obtained similarly.  We omit the detail for brevity. The proof of Lemma \ref{lem4.1} is completed.
\end{proof}

\subsection{Fast mode}
To capture the viscous effect and the no-slip  boundary data, one needs to investigate the fast mode in
the sub-layer with respect to the Prandtl layer scale. 
For this, set $z=\delta^{-1}Y$ as the boundary sub-layer variable, where $\delta=e^{-\f{1}{6}\pi i}n^{-\f13}$. Then we 
look for an approximate solution to \eqref{eq1.1} in the form of
\begin{align}\label{4.2.1}
(\Phi_{\text{app}}^f,\Psi_{\text{app}}^f)(Y)=(\phi^f,\delta\varphi^f)(\delta^{-1}Y).
\end{align}
Plugging this ansatz into the Orr-Sommerfeld equations \eqref{eq1.1} and taking the leading order, we obtain
\begin{equation}\left\{
\begin{aligned}
&\pa_z^4\phi^f-(z+z_0)\pa_z^2\phi^f=0,\\
&-\pa_z^2\varphi^f-\pa_z\phi^f=0,
\end{aligned}\nonumber
\right.
\end{equation}
for $z=e^{\f{1}{6}\pi i}n^{\f13}Y\in e^{\f{1}{6}\pi i}\mathbb{R}_+$. Here we have used $$
U(Y)-\hat{c}\sim Y-\hat{c}=\delta(z+z_0),~\text{ where }
z_0\triangleq -\delta^{-1}\hat{c}=e^{-\f{5}{6}\pi i}n^{\f13}\hat{c}.$$ Therefore, we can expect that the profile $(\phi^f, \varphi^f)$ satisfies:
$$
\pa_z^2\phi^f(z)\sim Ai(z+z_0),~ \pa_z^2\varphi^f(z)\sim -\pa_z\phi^f(z),
$$
where $Ai(z)$ is the classical Airy function defined in \eqref{A1}. Inspired by the above formal argument, we define 
\begin{align}\label{4.2.3}
\phi^f(z)\triangleq\frac{Ai(2,z+z_0)}{Ai(2,z_0)}, ~
\varphi^f(z)\triangleq-\f{Ai(3,z+z_0)}{Ai(2,z_0)},
\end{align}
where $Ai(2,z)$ and $Ai(3,z)$ given in \eqref{A2} are the second and third order primitive functions of classical Airy function respectively. By a straightforward computation, we know that the fast mode $(\Phi_{\text{app}}^f,\Psi_{\text{app}}^f)$ satisfies 
\begin{equation}\nonumber
\left\{
\begin{aligned}
&\f{i}{n}\pa_Y^4\Phi^f_{\text{app}}+(Y-\hat{c})\pa_Y^2\Phi^f_{\text{app}}=0,\\
&-\pa_Y^2\Psi^f_{\text{app}}-\pa_Y\Phi^f_{\text{app}}=0,\quad Y>0.
\end{aligned}
\right.
\end{equation}
In addition, at $Y=0$ it holds that
$$
\Phi_{\text{app}}^f(Y;{c})|_{Y=0}\equiv1.
$$

\subsection{Approximate eigenvalue}
We are now ready to construct the approximate growing mode that represents the Tollmien-Schlichting wave in the
boundary layer profile. Set
\begin{align}\label{4.3-1.3}
\Phi_{\text{app}}(Y;{c})\triangleq \Phi_{\text{app}}^s(Y;{c})-\Phi_{\text{app}}^s(0;{c})\Phi_{\text{app}}^f(Y;{c}),\nonumber\\
\Psi_{\text{app}}(Y;{c})\triangleq \Psi_{\text{app}}^s(Y;{c})-\Phi_{\text{app}}^s(0;{c})\Psi_{\text{app}}^f(Y;{c}).
\end{align}
Here  $\Psi_{\text{app}}^s(Y;{c})$ is the solution to \eqref{3.1} with $$\varphi_{b}=\Phi_{\text{app}}^s(0;{c})\Psi_{\text{app}}^f(0;{c})\text{ and } f=i\a H_s\Phi_{\text{app}}^s(Y;{c})+\pa_Y\Phi_{\text{app}}^s(Y;{c}).$$ 
From the construction, the approximate solution satisfies the zero Dirichlet boundary conditions on the stream functions: 
$$\Phi_{\text{app}}(0;{c})=\Psi_{\text{app}}(0;{c})=0.$$ 
To recover the no-slip boundary condition,  we define
\begin{align}\label{4.3-1.1}
\Gamma_0({c})\triangleq \pa_Y\Phi_{\text{app}}(0;{c})=\pa_Y\Phi_{\text{app}}^s(0;c)-\delta^{-1}\Phi_{\text{app}}^s(0;{c})\frac{Ai(1,z_0({c}))}{Ai(2,z_0({c}))}.
\end{align}
Then $(\Phi_{\text{app}},\Psi_{\text{app}})(Y;c)$ satisfies the boundary conditions in \eqref{BD1} if and only if $\Gamma_0({c})=0$.
Let
\begin{align}\label{c}
c_*\triangleq (A+A^{-1}e^{\f{1}{4}\pi i})\vep^{\f18}.
\end{align}
The following proposition gives the existence of                  approximate eigenvalue near $c_*$.
\begin{proposition}\label{prop4.1}
Let $\theta\in (0,1)$ be a fixed number. There is a positive constant $A_0>1$ such that if $A\geq A_0$, then there exists a positive constant $
\vep_1\in (0,1)$, such that for any $\a=A\vep^{\f18}$ with $\vep\in (0,\vep_1)$, the function $\Gamma_0({c})$ defined in \eqref{4.3-1.1} admits a unique zero point ${c}_{\text{app}}$ in the disk $D_{*}=\{{c}\in \mathbb{C}\mid |\hat{c}-c_*|\leq A^{-1-\theta}\vep^{\f18}\}$.
Moreover, on the circle $\pa D_*,$ it holds that
\begin{align}\label{lb}
|\Gamma_0(c)|\geq \f12 A^{-\theta}.
\end{align}
\end{proposition}
\begin{proof}
	For convenience, we zoom in the disk $D_*$ by introducing $\hat{c}=\vep^{\f18}h$, $h_*=A+A^{-1}e^{\f{1}{4}\pi i}$ and $\hat{\Gamma}_0(h)\triangleq \Gamma_0(\vep^{\f18}h-\f{i}{n})$. 
	It suffices to study the zeros of $\hat{\Gamma}_0(h)$ in the disk $D_h=\{h\in\mathbb{C}\mid |h-h_*|\leq A^{-1-\theta}\}.$ Recall that $$z_0=-\delta^{-1}\hat{c}=e^{-\frac{5}{6}\pi i}A^{\f13}h.$$
Then by taking $A_0\geq1$ sufficiently large, it is straightforward to check that for $A\geq A_0$,
\begin{align}
|z_0|=A^{\f43}\left(1+O(1)A^{-2}\right),~h\in D_h,\label{4.3-1.4}
\end{align}
and there exists a positive constant $\tau_0>0$, such that
\begin{align}
-\frac{5}{6}\pi<\arg z_0\leq -\f{5}{6}\pi+\tau_0A^{-2}.\label{4.3-1.5}
\end{align}
Therefore, from the asymptotic behavior \eqref{EA1} of Airy profiles we can obtain
\begin{align}\label{4.3-1.2}
\frac{Ai(1,z_0)}{Ai(2,z_0)}=-z_0^{\f12}+O(1)|z_0|^{-1}=-A^{\f23}e^{-\f{5\pi i}{12}}+O(1)A^{-\f43}.
\end{align}
Thus by taking the  leading order in \eqref{4.3-1.1} and using the asymptotic behavior \eqref{4.1.7} and \eqref{4.1.7-1} in Lemma \ref{lem4.1} for small $\hat{c}$, we have
 $$\hat{\Gamma}_0(h)\sim 1+\delta^{-1}(\a-\hat{c})A^{\f23}e^{-\f{5\pi i}{12}}=1+e^{-\f{1 }{4}\pi i}A(A-h),~ \text{for }h\in D_h.$$
Based on this, we define a reference mapping:
$$\hat{\Gamma}_{\text{ref}}(h)=1+e^{-\f{1 }{4}\pi i}A(A-h).
$$
On  one hand, we know that $\hat{\Gamma}_{\text{ref}}(h)$ has a unique zero point $h=h_*$ in the whole complex plane. Moreover, on the circle $\{h\in \mathbb{C}
\mid |h-h_*|=A^{-1-\theta}\}$ it holds that
$$|\hat{\Gamma}_{\text{ref}}(h)|=A^{-\theta}.
$$
On the other hand, we can estimate the difference between $\hat{\Gamma}_0$ and $\hat{\Gamma}_{\text{ref}}$ as follows.
$$\begin{aligned}
\left| \hat{\Gamma}_0(h)-\hat{\Gamma}_{\text{ref}}(h)     \right|\leq& \left| 1-e^{\f{1}{6}\pi i}A^{\f13}(A-h)\frac{Ai(1,z_0)}{Ai(2,z_0)}-\hat{\Gamma}_{\text{ref}}(h) \right|+C_{A}\vep^{\f18}|\log\vep|\\
\leq&\left| 1+e^{\f{1}{6}\pi i}A^{\f13}(A-h) z_0^{\f12} -\hat{\Gamma}_{\text{ref}}(h)\right|+CA^{-2}+C_{A}\vep^{\f18}|\log\vep|\\
\leq & C_1A^{-2}+C_{1,A}\vep^\f{1}{8}|\log\vep|.
\end{aligned}
$$
Here the positive constant $C_1$ does not depend on either $\vep$ or $A$ and $C_{1,A}$ may depend on $A$ but not on $\vep$. Since $\theta\in (0,1)$, taking $A_0$ suitably large and then taking $\vep_1$ sufficiently small such that
$$ C_1A^{-2}+C_{1,A}\vep^\f{1}{8}|\log\vep|
\leq \f12 A^{-\theta}
$$
for $A\geq A_0$ and $\vep\in (0,\vep_1)$, we obtain that $\left| \hat{\Gamma}_0(h)-\hat{\Gamma}_{\text{ref}}(h)\right|<|\hat{\Gamma}_{\text{ref}}(h)|$ on the circle $\{h\in \mathbb{C}
\mid |h-h_0|=A^{-1-\theta}\}$. In particular, we have $$|\Gamma_0(c)|=|\hat{\Gamma}_0(h)|\geq |\hat{\Gamma}_{\text{ref}}(h)|-|\hat{\Gamma}_0(h)-\hat{\Gamma}_{\text{ref}}(h)|\geq \f12 A^{-\theta},$$ that is \eqref{lb}. Moreover, note that $Ai(k,z), k=0,1,2$ are entire functions and $Ai(2,z_0)\neq 0$ due to \eqref{4.3-1.2}. Both $\hat{\Gamma}_0$ and $\hat{\Gamma}_{\text{ref}}$ are holomorphic in the disk $D_h$. Therefore, Proposition \ref{prop4.1} is concluded by Rouch\'e's Theorem.
\end{proof}
\subsection{Error estimates}
 In this subsection, we will give the precise representation of the  error terms generated by the approximate growing mode $(\Phi_{\text{app}},\Psi_{\text{app}})(Y;{c})$ defined in \eqref{4.3-1.3} and then
 provide some estimates.
 By a straightforward computation, $(\Phi_{\text{app}},\Psi_{\text{app}})(Y;{c})$ satisfies
 \begin{equation}\left\{
 \begin{aligned}\label{5.5-2}
 &\text{OS}_c(\Phi_{\text{app}},\Psi_{\text{app}})=(E^s(Y;c)+E^f(Y;c),F^f(Y;c)),~Y>0,\\
 &\Phi_{\text{app}}(Y;c)|_{Y=0}=\Psi_{\text{app}}(Y;c)|_{Y=0}=0.
 \end{aligned}\right.
 \end{equation}
 Moreover, from Proposition \ref{prop4.1} we know that there exists a unique $c_{\text{app}}\in D_*$, such that $\pa_Y\Phi_{\text{app}}(0;c)=0.$ In what follows we specify the error terms on the right hand side of \eqref{5.5-2}. First of all, let us focus on the error term $E^s(Y;{c})$ generated by  slow mode $(\Phi_{\text{app}}^s,\Psi_{\text{app}}^s)(Y;{c})$. For latter use, we divide  $E^s(Y;{c})$ into following three groups: the first two groups consist of terms that can be written either in derivatives of $Y$ or $x$,  and the remaining one consists of terms with strong decay in $Y.$ That is, 
 \begin{equation}
 \begin{aligned}\label{5.5}
 E^{s}(Y;{c})=\pa_YE^s_{1}(Y;{c})+i\a E^s_2(Y;{c})+E^s_{3}(Y;{c}), 
 \end{aligned}
 \end{equation}
 where 
 $$
 \begin{aligned}
 E_{1}^s(Y;{c})\triangleq& \frac{i}{n}\left(\pa_Y^3\Phi_{\text{app}}^s(Y;{c})-2\a^2\pa_Y\Phi_{\text{app}}^s(Y;{c})\right)-\sqrt{\vep}H_s\pa_Y\Psi_{\text{app}}^s(Y;{c})\\
 &-\f{\a}{n}\left( (U_s-{c})\Psi_{\text{app}}^s(Y;{c})-H_s\Phi_{\text{app}}^s(Y;{c}) \right),\\
 E_{2}^s(Y;{c})\triangleq&\frac{\a^3}{n}\Phi_{\text{app}}^s(Y;{c})-i\a\sqrt{\vep}H_s\Psi_{\text{app}}^s(Y;{c})-\f{\a}{n}\Phi_{\text{app}}^s(Y;{c}),\\
 E_{3}^s(Y;{c})\triangleq&-2\a^2(U_s-\hat{c})(\pa_Y\Phi_{1}^s(Y;{c})+\a\Phi_{1}^s(Y;{c}))+\sqrt{\vep}\pa_YH_s\pa_Y\Psi_{\text{app}}^s(Y;{c})+\sqrt{\vep}\pa_Y^2H_s\Psi_{\text{app}}^s(Y;{c}).
 \end{aligned}
 $$
The error terms $(E^f(Y;c),F^f(Y;c))$ generated by the approximate fast mode are also written in the similar form:
\begin{equation}\label{5.5-1}
E^f(Y;{c})=\pa_YE_1^f(Y;{c})+i\a E_2^f(Y;c)+E_3^f(Y;{c}),
\end{equation}
where
\begin{align}
E^f_1(Y;{c})\triangleq&-\Phi_{\text{app}}^s(0;{c})\left(\f{-2\a^2i}{n}\pa_Y\Phi_{\text{app}}^f-\sqrt{\vep}H_s\pa_Y\Psi_{\text{app}}^f-\f{\a}{n}(U_s-{c})\Psi_{\text{app}}^f+\f{\a}{n}H_s\Phi_{\text{app}}^f\right),\nonumber\\
E^f_2(Y;{c})\triangleq&-\Phi_{\text{app}}^s(0;{c})\left( \f{\a^3}{n}\Phi_{\text{app}}^f+i\a(U_s-\hat{c})\Phi_{\text{app}}^f -\f{\a}{n}\Phi_{\text{app}}^f-i\a\sqrt{\vep}H_s\Psi_{\text{app}}^f  \right),\nonumber\\
E^f_3(Y;{c})\triangleq&-\Phi_{\text{app}}^s(0;{c})\left( (U_s(Y)-U_s'(0)Y )\pa_Y^2\Phi_{\text{app}}^f-\pa_Y^2U_s\Phi_{\text{app}}^f+\sqrt{\vep}\pa_YH_s\pa_Y\Psi_{\text{app}}^f+\sqrt{\vep}\pa_Y^2H_s\Psi_{\text{app}}^f\right),\nonumber
\end{align}
and
\begin{equation}\label{5.5-3}
F^f(Y;{c})\triangleq-\Phi_{\text{app}}^s(0;{c})\left( \a^2\Psi_{\text{app}}^f+i\a(U_s-{c})\Psi_{\text{app}}^f-i\a H_s\Phi_{\text{app}}^f  \right).
\end{equation}
The main purpose in this subsection is to establish the following two Propositions for some estimates on these error terms.
Recall that $A_0>1$ and $\vep_1\in (0,1)$ are the numbers given in Proposition \ref{prop4.1} in which the disk $D_*$ has been defined.
 \begin{proposition} \label{prop4.2} Let $A\geq A_0$ be a fixed number. There exists $\vep_2\in (0,\vep_1),$ such that for $\vep\in (0,\vep_2),$  $\a=A\vep^{\f18}$ and ${c}\in D_{*}$, we have
 	\begin{align}
 	\|E_1^s(\cdot~;{c})\|_{L^2}+\|E_2^s(\cdot~;{c})\|_{L^2}\lesssim_{A} \vep^{\f{5}{16}},\label{5.8}\\
 	\||U_s''|^{-\f12}E_3^s(\cdot~;{c})\|_{L^2}\lesssim_{A} \vep^{\f{3}{16}}.\label{5.9}
 	\end{align}
 \end{proposition}
\begin{proposition}\label{prop4.3}
Under the same assumption as in Proposition \ref{prop4.2}, it holds that
\begin{align}
\|E_1^f(Y;{c})\|_{L^2}+\|E_2^f(Y;{c})\|_{L^2}+\|F^f(Y;{c})\|_{L^2}\lesssim_A\vep^{\f{5}{16}},\label{3.4-2}\\
\||U_s''|^{-\f12}E_3^f(Y;{c})\|_{L^2}\lesssim_A\vep^{\f{3}{16}}.\label{3.4-3}
\end{align}
\end{proposition}
To prove Proposition \ref{prop4.2} and \ref{prop4.3}, we need some preparation.
First, we prove the following lemma about some pointwise estimate on $\psi_{0,2}$ defined in \eqref{4.1.2} which will be frequently used in the error estimates.
\begin{lemma}\label{lem5.2}
For any ${c}\in \{ \text{Im}c>0,~ |\hat{c}|\leq \g_2\}$ where $\g_2\in (0,1)$ is the number given in Lemma \ref{lem4.1}, the following pointwise estimates on $\psi_{0,2}$ hold.
\begin{enumerate}[(1)]
	\item For $Y\geq 1$, we have
	\begin{equation}
	\begin{aligned}\label{5.1}
	&|\psi_{0,2}(Y)|\lesssim (1+Y), ~|\pa_Y\psi_{0,2}(Y)|\lesssim 1,\\
	&|\pa_Y^2\psi_{0,2}(Y)|+|\pa_Y^3\psi_{0,2}(Y)|\lesssim U_s'(Y)(1+Y).
	\end{aligned}
	\end{equation}
	\item For $Y\in [0,1]$, we have
	\begin{equation}
	\begin{aligned}\label{5.2}
	&|\psi_{0,2}(Y)|\lesssim 1, ~\pa_Y\psi_{0,2}(Y)=-\frac{U_s''(Y)}{(U_s'(Y))^2}\log (U_s(Y)-\hat{c})+O(1),\\
	&\pa_Y^2\psi_{0,2}(Y)= -\frac{U_s''(Y)}{U_s'(Y)(U_s(Y)-\hat{c})}-\frac{(U_s''(Y))^2}{(U_s'(Y))^3}\log (U_s(Y)-\hat{c})+O(1),\\
	&\pa_Y^3\psi_{0,2}(Y)=\frac{U_s''(Y)}{(U_s(Y)-\hat{c})^2} -\frac{U_s'''(Y)}{U_s'(Y)(U_s(Y)-\hat{c})}-\frac{U_s''(Y)U_s'''(Y)}{(U_s'(Y))^3}\log (U_s(Y)-\hat{c})+O(1).
	\end{aligned}
	\end{equation}
\end{enumerate}
\end{lemma}
\begin{proof}
	Note that
	\begin{equation}\label{5.3}
	\begin{aligned}
	\pa_Y\psi_{0,2}(Y)&=U_s'(Y)\int_1^Y\f{1}{(U_s(X)-\hat{c})^2}\dd X+\f{1}{U_s(Y)-\hat{c}},\\
	\pa_Y^2\psi_{0,2}(Y)&=U_s''(Y)\int_1^Y\f{1}{(U_s(X)-\hat{c})^2}\dd X,\\
	\pa_Y^3\psi_{0,2}(Y)&=U_s'''(Y)\int_1^Y\f{1}{(U_s(X)-\hat{c})^2}\dd X+\frac{U_s''(Y)}{(U_s(Y)-\hat{c})^2}.
	\end{aligned}
	\end{equation}
	For $Y\geq 1,$ we have $U_s(Y)\geq U_s(1)>0$. Thus, for $c\in\{\text{Im}c>0,~|\hat{c}|\leq \g_2\}$, we have
	$|\hat{c}|\leq \frac{U_s(1)}{2}$, which further implies that	
	$$\frac{1}{|U_s(Y)-\hat{c}|}\lesssim 1, ~\text{for }Y\geq 1.$$
	From \eqref{5.3} and using the structural condition of boundary layer $|\f{U_s''}{U_s'}|+|\f{U_s'''}{U_s'}|\lesssim1,$ we can directly obtain \eqref{5.1}. For $Y\in[0,1]$, it suffices to consider the integration $\int_1^Y\frac{1}{(U_s(X)-\hat{c})^2}\dd X.$ Since $U_s'(Y)>0$, we can use integration by parts to obtain 
	$$
	\begin{aligned}
	\int_1^Y\frac{1}{(U_s(X)-\hat{c})^2}\dd X&=\int_1^Y\f{\dd }{\dd X}\left(\frac{-1}{U_s(X)-\hat{c}}\right)\frac{1}{U_s'(X)}\dd X\\
	&=-\frac{1}{U_s'(Y)(U_s(Y)-\hat{c})}+\frac{1}{U_s'(1)(U_s(1)-\hat{c})}-\int_1^Y\frac{1}{U_s(X)-\hat{c}}\frac{U_s''(X)}{(U_s'(X))^2}\dd X.
	\end{aligned}
	$$
For the second term, it is straightforward to have
$$\left|\frac{1}{U_s'(1)(U_s(1)-\hat{c})}\right|\lesssim1,~ \text{for }c\in \{ \text{Im}c>0,~ |\hat{c}|\leq \g_2\}.
$$
For the last term, we use integration by parts again to get
	$$
	\begin{aligned}
	&\int_1^Y\frac{1}{U_s(X)-\hat{c}}\frac{U_s''(X)}{(U_s'(X))^2}\dd X=\int_1^Y\f{\dd \left(\log (U_s(X)-\hat{c})\right)}{\dd X}\frac{U_s''(X)}{(U_s'(X))^3}\dd X\\
	&~=\left(\log (U_s(Y)-\hat{c})\right)\frac{U_s''(Y)}{(U_s'(Y))^3}-\left(\log (U_s(1)-\hat{c})\right)\frac{U_s''(1)}{(U_s'(1))^3}-\int_1^Y \left(\log (U_s(X)-\hat{c})\right)\f{\dd}{\dd X}\left(\frac{U_s''(X)}{(U_s'(X))^3}\right)\dd X.
	\end{aligned}
	$$
By using \eqref{5.8-1} with $\kappa=1$ and the fact that $ U_s'\sim 1$ and $|U_s''|, |U_s'''|\lesssim 1$ for $Y\in [0,1]$, we have
$$\left|\int_1^Y \left(\log (U_s(X)-\hat{c})\right)\f{\dd}{\dd X}\left(\frac{U_s''(X)}{(U_s'(X))^3}\right)\dd X\right|\lesssim \int_0^1|\log(U_s(X)-\hat{c})|\dd X\lesssim 1,
$$
and 
$$\left|\left(\log (U_s(1)-\hat{c})\right)\frac{U_s''(1)}{(U_s'(1))^3}\right|\lesssim |\log U_s(1)|+1\lesssim 1.
$$
Therefore, it holds that
	\begin{align}\label{5.4}
	\int_1^Y\frac{1}{(U_s(X)-\hat{c})^2}\dd X=-\frac{1}{U_s'(Y)(U_s(Y)-\hat{c})}-\left(\log (U_s(Y)-\hat{c})\right)\frac{U_s''(Y)}{(U_s'(Y))^3}+R(Y;{c}),
	\end{align}
	where $|R(Y;{c})|\lesssim 1$ uniformly for $Y\in[0,1]$ and ${c}\in \{ \text{Im}c>0,~ |\hat{c}|\leq \g_2\}$. With \eqref{5.4}, we can obtain \eqref{5.2} directly from \eqref{5.3}. The proof of Lemma \ref{lem5.2} is then completed.
\end{proof}
With Lemma \ref{lem5.2}, we can  obtain some estimates on the corrector $\Phi_{1}^s(Y;{c})$ defined in \eqref{4.1.4-1}.

\begin{lemma}\label{lem5.4}
	For any $\a\in (0,1)$ and $c\in$ $\{\text{Im}c>0,~|\hat{c}|\leq \g_2 \}$, we have
\begin{equation}
	\begin{aligned}\label{5.7}
	&\|\Phi_{1}^s(\cdot~;{c})\|_{L^\infty_\a}\lesssim 1,\\
	&\|\pa_Y\Phi_{1}^s(\cdot~;{c})\|_{L^\infty_\a}\lesssim 1+|\log \text{Im}\hat{c}|,~\|\pa_Y\Phi_{1}^s(\cdot~;{c})\|_{L^2}\lesssim 1,\\
	&\|\pa_Y^2\Phi_{1}^s(\cdot~;{c})\|_{L^\infty_\a}\lesssim 1+| \text{Im}\hat{c}|^{-1},~\|\pa_Y^2\Phi_{1}^s(\cdot~;\hat{c})\|_{L^2}\lesssim 1+| \text{Im}\hat{c}|^{-\f12},\\
	&\|\pa_Y^3\Phi_{1}^s(\cdot~;{c})\|_{L^\infty_\a}\lesssim 1+| \text{Im}\hat{c}|^{-2},~\|\pa_Y^3\Phi_{1}^s(\cdot~;{c})\|_{L^2}\lesssim 1+| \text{Im}\hat{c}|^{-\f32},\\
	&\||U_s''|^{-\f12}(\pa_Y+\a)\Phi_{1}^s(\cdot~;{c})\|_{L^2}\lesssim 1+\a^{-\f12}.
	\end{aligned}
\end{equation}
\end{lemma}
\begin{remark}
  As mentioned in the Introduction, the $L^2$-norms  of $\pa_Y^2\Phi_1^s$ and $\pa_Y^3\Phi_1^s$ 
  absord one  order of $(Im\hat{c})^{-\f12}$ compared with the corresponding $L^\infty$-norms. This is  crucially used  in later remainder estimates. 
\end{remark}
{\bf Proof of Lemma \ref{lem5.4}:}	Recall some explicit formula \eqref{5.6} related to $\Phi_{1}^s(Y;{c}).$ With the pointwise estimates \eqref{5.1} for $Y\geq 1$, \eqref{5.2} for $Y\in[0,1]$ and the integration estimates over $[0,1]$ in Lemma \ref{lem5.3}, \eqref{5.7} can be obtained by a tedious but straightforward computation. For illustration, we only show the last inequality in \eqref{5.7}. 
	Note that for $Y>0$, it holds 
	$$\begin{aligned}
	\pa_Y\Phi_1^s(Y;{c})+\a\Phi_{1}^s(Y;{c})=&-2\pa_Y\psi_{0,1}(Y)e^{-\a Y}\int_0^YU_s'(X)\psi_{0,2}(X)\dd X\\
	&-2
	\pa_Y\psi_{0,2}(Y)e^{-\a Y}\int_Y^\infty U_s'(X)\psi_{0,1}(X)\dd X.
	\end{aligned}
	$$
	We then split the integral domain with respective to $Y$ into $\{Y\geq 1\}$ and $\{0\leq Y\leq 1\}$. For the first domain, by using \eqref{5.1} and the following facts
	$$\left|\int_0^Y U'(X)\psi_{0,2}(X)\dd X\right|\lesssim 1,~\left|\int_Y^\infty U'(X)\psi_{0,1}(X)\dd X\right|\lesssim 1-U_s(Y)\lesssim U_s'(Y),
	$$
 we have
	$$\begin{aligned}
	\int_1^\infty|U_s''|^{-1}\left|\pa_Y\Phi_1^s(Y;{c})+\a\Phi_{1}^s(Y;{c})\right|^2\dd Y\lesssim&\left\|\frac{(U_s')^2}{U_s''}\right\|_{L^\infty}\|e^{-\a Y}\|_{L^2}^2\lesssim \a^{-1}.
	\end{aligned}
	$$
	Here we have used the strong concave condition \eqref{BL} in the last inequality. For the second domain, we use \eqref{5.2} and \eqref{5.8-1} with $\ka=2$ to obtain
	$$
	\int_0^1|U_s''|^{-1}\left|\pa_Y\Phi_1^s(Y;{c})+\a\Phi_{1}^s(Y;{c})\right|^2\dd Y\lesssim 1+\int_0^1|\log (U_s(Y)-\hat{c})|^2\dd Y\lesssim 1.
	$$
The last inequality in \eqref{5.7} is obtained by combining these two  estimates. The proof of Lemma \ref{lem5.4} is then completed.  \qed

Next lemma is about estimates on the approximate slow mode of magnetic field defined in \eqref{4.3-1.3}.
 
\begin{lemma}\label{lem5.6}
Under the same assumption as in Proposition \ref{prop4.2}, the approximate slow mode of magnetic field $\Psi_{\text{app}}^s(Y;{c})$ satisfies the following estimates:
	\begin{align}
	\|\Psi_{\text{app}}^s(\cdot~;{c})\|_{L^\infty_\a}+	\a^{\f12}\|\Psi_{\text{app}}^s(\cdot~;{c})\|_{L^2}&\lesssim \a^{-1},\label{5.11-1}\\
	\|\pa_Y\Psi_{\text{app}}^s(\cdot~;{c})\|_{L^\infty_\a}+\a^{\f12}\|\pa_Y\Psi_{\text{app}}^s(\cdot~;{c})\|_{L^2}&\lesssim \a^{-\f12},\label{5.11-2}\\
	\|\pa_Y^2\Psi_{\text{app}}^s(\cdot~;{c})\|_{L^\infty_\a}+	\a^{\f12}\|\pa_Y^2\Psi_{\text{app}}^s(\cdot~;{c})\|_{L^2}&\lesssim 1.\label{5.11-3}
	\end{align}
\end{lemma}
\begin{proof}
	First, we take $\vep_1$ sufficiently small such that for any $\vep\in (0,\vep_1)$, it holds that
	$$D_*\subsetneqq\{\text{Im}c>0,~ |\hat{c}|\leq \g_2\}\cap\{|c|\leq \g_0\} ~\text{and } \a\in (0,\a_0),
	$$
	where $\g_0$ and $\a_0$ are given in Proposition \ref{prop3.1}. Recall that $\Psi_{\text{app}}^s(Y;{c})$ satisfies  \eqref{3.1} with boundary data $\varphi_b=\Phi_{\text{app}}^s(0;{c})\Psi_{\text{app}}^f(0;{c})$ and inhomogeneous source term $f=i\a H_s\Phi_{\text{app}}^s(Y;{c})+\pa_Y\Phi_{\text{app}}^s(Y;{c})$. As in \eqref{4.3-1.2}, we use
 the asymptotic behavior of Airy profiles \eqref{EA1} to obtain
	$$\left| \Psi_{\text{app}}^f(0;{c}) \right|\lesssim |\delta|\left|\frac{Ai(3,z_0)}{Ai(2,z_0)}\right|\lesssim|\delta||z_0|^{-\f12}\lesssim \a.
	$$
	Here we have used \eqref{4.3-1.4} and the fact that $|\delta|=n^{-\f13}={\vep^{\f16}}{\a^{-\f13}}=A^{-\f43}\a$. Then combining this with \eqref{4.1.7} gives $|\varphi_b|\lesssim \a^2.$ Moreover, by using previous bounds \eqref{5.7}, we have, for any $c\in D_*$, that
	$$\begin{aligned}
	&\|\pa_Y\Phi_{\text{app}}^s(Y;{c})\|_{L^\infty_\a}+\a\|\Phi_{\text{app}}^s(Y;{c})\|_{L^\infty_\a}\\
	&\qquad\qquad\lesssim 	\|\pa_Y\psi_{\a,1}(Y;{c})\|_{L^\infty_\a}+\a\|\psi_{\a,1}(Y;{c})\|_{L^\infty_\a}+\a\left(\|\pa_Y\Phi_{1}^s(Y;{c})\|_{L^\infty_\a}+\a\|\Phi_{1}^s(Y;{c})\|_{L^\infty_\a} \right)\\
	&\qquad\qquad\lesssim
	 1+\a|\log \text{Im}\hat{c}|\lesssim 1.
	\end{aligned}
	$$
Now  by taking $\vep_1$  smaller  if necessary such that $0<\a<\frac{\sqrt{2\a}}{4}$,
	we can apply \eqref{3.2} with $\eta=\a$ to $\Psi_{\text{app}}^s$  to have
	\begin{align}
	\a\|\Psi_{\text{app}}^s(\cdot~;{c})\|_{L^\infty_\a}+
	\a^{\f12}\|\pa_Y\Psi_{\text{app}}^s(\cdot~;{c})\|_{L^\infty_\a}+
	\|\pa_Y^2\Psi_{\text{app}}^s(\cdot~;{c})\|_{L^\infty_\a}\lesssim 1.\nonumber
	\end{align}
	The $L^2$-estimates can be obtained directly from the following inequalities
	$$\|\pa_Y^k\Psi_{\text{app}}^s\|_{L^2}\lesssim \a^{-\f12}\|\pa_Y^k\Psi_{\text{app}}^s\|_{L^\infty_\a},~k=0,1,2.
	$$
	The proof of Lemma \ref{lem5.6} is  completed.
\end{proof}
Now we are ready to prove Proposition \ref{prop4.2}.\\

{\bf Proof of Proposition \ref{prop4.2}:} Recall from \eqref{4.1.5} that $\Phi_{\text{app}}^s=\psi_{\a,1}+\a\Phi_{1}^s$. By a direct computation, we have
	\begin{align}\label{5.9-1}
	\a^{\f12}\|\psi_{\a,1}\|_{L^2}+\|\psi_{\a,1}\|_{L^\infty_\a}+\sum_{k=1}^3\|\pa_Y^k\psi_{\a,1}\|_{L^2\cap L^\infty_\a}\lesssim 1.
	\end{align}
	Then applying \eqref{5.7}--\eqref{5.9-1} to $E_1^s(Y;{c})$, we have
	\begin{equation}
	\begin{aligned}\label{5.9-2}
	\|E_1^s(\cdot~;{c})\|_{L^2}\lesssim& \frac{1}{n}\left( \|\pa_Y^3\psi_{\a,1}\|_{L^2}+\a\|\pa_Y^3\Phi_1^s\|_{L^2} +\a^2\|\pa_Y\psi_{\a,1}\|_{L^2}+\a^3\|\pa_Y\Phi_1^s\|_{L^2}\right)\\
	&+\vep^{\f12}\|\pa_Y\Psi_{\text{app}}^s\|_{L^2}+\frac{\a}{n}\left( \|\Psi_{\text{app}}^s\|_{L^2}+\|\psi_{\a,1}\|_{L^2}+\a\|\Phi_1^s\|_{L^2}  \right)\\
	\lesssim &\frac{1}{n}\left(1+\f{\a}{|\text{Im}\hat{c}|^{\f32}}+\a^2+\a^3  \right)+\f{\vep^{\f12}}{\a}+\f{\a}{n}\left(\f{1}{\a^{\f32}}+\f{1}{\a^{\f12}}+\a^{\f12}\right).
	\end{aligned}
	\end{equation}
	Note that for $\a=A\vep^{\f18}$ and ${c}\in D_*$, we have
	$$n^{-1}\sim_A \vep^{\f38},~ \text{Im}\hat{c}\sim_A \vep^{\f18}.
	$$
	By taking $\vep_2\in (0,\vep_3)$ suitably small, we can deduce for any $\vep\in (0,\vep_2)$ that
	$$
	\|E_1^s(\cdot~;{c})\|_{L^2}\lesssim_A\vep^{\f38}\left( 1 +\vep^{-\f{1}{16}}+\vep^{\f14}+\vep^{\f38}\right)+\vep^{\f38}+\vep^{\f12}\left( \vep^{-\f3{16}}+\vep^{-\f{1}{16}}+\vep^{\f1{16}} \right)\lesssim_A\vep^{\f{5}{16}}.
	$$
	Similarly, we have
	$$
	\begin{aligned}
	\|E_2^s(\cdot~;{c})\|_{L^2}&\lesssim \f{\a^3}{n}\|\Phi_{\text{app}}^s\|_{L^2}+\a\vep^{\f12}\|\Psi_{\text{app}}^s\|_{L^2}+\frac{\a}{n}\|\Phi_{\text{app}}^s\|_{L^2}\\
	&\lesssim \frac{\a^{\f52}}{n}+\frac{\vep^{\f12}}{\a^{\f12}}+\frac{\a^{\f12}}{n}\lesssim_A \vep^{\f{7}{16}}.
	\end{aligned}
	$$
	This completes the proof of \eqref{5.8}. Finally, by using the last inequality in \eqref{5.7}, \eqref{5.11-2} and structure of boundary layer profile \eqref{BL}, we obtain
	$$
	\begin{aligned}
	\||U_s''|^{-\f12}E_{3}^s(\cdot~;{c})\|_{L^2}\lesssim& \a^2\|U_s''|^{-\f12}(\pa_Y\Phi_1^s+\a\Phi_1^s)\|_{L^2}+\vep^{\f12}\left\|\f{\pa_YH_s}{\pa_YU_s}\right\|_{L^\infty}\left\|\f{\pa_YU_s}{|\pa_Y^2U_s|^\f12}\right\|_{L^\infty}\|\pa_Y\Psi_{\text{app}}^s\|_{L^2}\\
	&+ \vep^{\f12}\left\|\f{\pa_Y^2H_s}{\pa_YU_s}\right\|_{L^\infty}\left\|\f{\pa_YU_s}{|\pa_Y^2U_s|^\f12}\right\|_{L^\infty}\|\Psi_{\text{app}}^s\|_{L^2}\\
	\lesssim &\a^{\f32}+\frac{\vep^{\f12}}{\a}+\frac{\vep^{\f12}}{\a^{\f32}}\lesssim_A \vep^{\f{3}{16}},
	\end{aligned}
	$$
	that is \eqref{5.9}. The proof of Proposition \ref{prop4.2} is completed.\qed\\
	
Now we turn to prove Proposition \ref{prop4.3}.
First we will show the following pointwise estimates on the  approximate fast mode $(\Phi_{\text{app}}^f,\Psi_{\text{app}}^f)$ defined in \eqref{4.2.1}. 
\begin{lemma}\label{lem5.7}
	Let $A\geq A_0$ be a fixed number. There exists a positive constant $\tau_1$, such that if $\vep\in (0,\vep_1)$, $\a=A\vep^{\f18}$ and 
	 ${c}\in D_*$, the following pointwise estimates hold for any $Y>0$ and $k=0,1,2$:
	\begin{align}
	& |\pa_Y^k\Phi_{\text{app}}^f(Y;{c})|\lesssim_A n^{\f{k}3}e^{-\tau_1n^{\f13}Y},\label{5.10}\\
	& |\pa_Y^k\Psi_{\text{app}}^f(Y;{c})|\lesssim_A n^{\f{k-1}3}e^{-\tau_1n^{\f13}Y}\label{5.11}.
	\end{align}
\end{lemma}
\begin{proof}
Recall that $z_0=-\delta^{-1}\hat{c}=e^{\f{1 }{6}\pi i}n^{\f13}(-\hat{c})$ and $z+z_0=e^{\f{1 }{6}\pi i}n^{\f13}(Y-\hat{c}),~ Y>0$. As in \eqref{4.3-1.4} and \eqref{4.3-1.5}, we can deduce for $A\geq A_0$ and ${c}\in D_*$ that
$$|z_0|=A^{\f43}(1+O(A^{-2})),
$$
and there exists $\tau_0>0$, such that
$$-\f{5\pi}{6}<\arg z_0<-\f{5\pi}{6}+\tau_0A^{-2}.
$$
Thus,  both $z_0$ and $z_0+z$ belong to the half plane $\{z\in \mathbb{C}\mid  \arg z\in(-\f{5\pi}{6},\f{\pi}{6}) \}$. Therefore, by applying the asymptotic behavior of Airy profile \eqref{EA1}  to \eqref{4.2.1}
and \eqref{4.2.3}, we have
\begin{equation}\label{5.12}
\begin{aligned}
\left| \pa_Y^k\Phi_{\text{app}}^f(Y)\right|&\lesssim |\delta|^{-k}\left|  \f{Ai(2-k,z+z_0)}{Ai(2,z_0)} \right|\\
&\lesssim n^{\f{k}{3}}|z_0|^{\f54}|z+z_0|^{-\f{5-2k}{4}}
\left| \exp\left( -\f23( z+z_0)^{\f{3}{2}}+\f23(z_0)^{\f32}  \right)\right|,\\
\left| \pa_Y^k\Psi_{\text{app}}^f(Y)\right|&\lesssim |\delta|^{-k+1}\left|  \f{Ai(3-k,z+z_0)}{Ai(2,z_0)} \right|\\
&\lesssim n^{\f{k-1}{3}}|z_0|^{\f54}|z+z_0|^{-\f{7-2k}{4}}
\left| \exp\left( -\f23( z+z_0)^{\f{3}{2}}+\f23(z_0)^{\f32}  \right)\right|.
\end{aligned}
\end{equation}
Noting that $|z_0|\sim A^{\f43}$ and $c$ belongs to $D_*$,
$$|z+z_0|=n^{\f13}|Y-\hat{c}|\geq n^{\f13}\text{Im}\hat{c}\gtrsim \left(A\vep^{-\f38}\right)^{\f13}A^{-1}\vep^{\f18}\gtrsim A^{-\f23},
$$
we further obtain from \eqref{5.12} that
\begin{equation}\label{5.13}
\begin{aligned}
\left| \pa_Y^k\Phi_{\text{app}}^f(Y)\right|&\lesssim_A n^{\f{k}{3}}
\left| \exp\left( -\f23( z+z_0)^{\f{3}{2}}+\f23(z_0)^{\f32}  \right)\right|,\\
\left| \pa_Y^k\Psi_{\text{app}}^f(Y)\right|
&\lesssim_A n^{\f{k-1}{3}}
\left| \exp\left( -\f23( z+z_0)^{\f{3}{2}}+\f23(z_0)^{\f32}  \right)\right|.
\end{aligned}
\end{equation}
It remains to give a lower bound of real part of $ ( z+z_0)^{\f{3}{2}}-(z_0)^{\f32}$. In fact, this lower bound has been studied in \cite{GMM1}, even for more general cases. For completeness, we follow the argument
in \cite{GMM1} to give a  brief presentation. Let $f(Y)\triangleq e^{\f{1}{6}\pi i}(Y-\hat{c})$. Then 
$$( z+z_0)^{\f{3}{2}}-(z_0)^{\f32}=n^{\f12}\left( (f(Y))^{\f32}-(f(0))^{\f32}   \right).
$$
For $t\in [0,1]$,  $\arg f(tY)\in (-\f{5\pi}{6},\f{\pi}{6})$. Thus 
$$ (f(Y))^{\f32}-(f(0))^{\f32}=\int_0^1\f{\dd}{\dd t}\left( f(tY)\right)^{\f32}\dd t=\f{3}{2}Y\int_0^1\left( e^{\f{1}{3}\pi i} f(tY) \right)^{\f12}\dd t.
$$
Since $e^{\f{1}{3}\pi i}\arg f(tY)\in (-\f{\pi}{2},\f{\pi}{2})$, we obtain 
$$\text{Re}\int_0^1\left( e^{\f{\pi i}{3}} f(tY) \right)^{\f12}\dd t\geq \frac{\sqrt{2}}{2}\int_0^1|f(tY)|^{\f12}\dd t\geq \frac{\sqrt{2}}{2}\int_0^1|tY-\hat{c}|^{\f12}\dd t\geq \hat{\tau}_1|\hat{c}|^{\f12},
$$
for some positive constant $\hat{\tau}_1>0$. It then implies that
$$\f{2}{3}\text{Re}\left( ( z+z_0)^{\f{3}{2}}-(z_0)^{\f32}  \right)=\f{2}{3}n^{\f12}\text{Re}\left( (f(Y))^{\f32}-(f(0))^{\f32}\right)\geq {\hat{\tau}_1} n^{\f12}|\hat{c}|^{\f12}Y.
$$
For ${c}\in D_*$, we have $|n|^{\f16}|\hat{c}|^{\f12}\gtrsim \left( A\vep^{-\f38} \right )^{\f16}\left( A\vep^{\f18} \right )^{\f12}\gtrsim A^{\f23}$. Thus 
$$\f{2}{3}\text{Re}\left( ( z+z_0)^{\f{3}{2}}-(z_0)^{\f32}  \right)\geq {\hat{\tau}_1}A^{\f23} n^{\f13}Y.
$$
By plugging this bound back into \eqref{5.13} and setting $\tau_1={\hat{\tau}_1}A^{\f23}$, \eqref{5.10} and \eqref{5.11} follow. The proof of Lemma \ref{lem5.7} is completed.
\end{proof}
With these pointwise estimates, we can obtain the following weighted $L^2$-estimates directly.
\begin{lemma}\label{lem5.8}
	Let $\beta\geq 0$ and $k=0,1,2$. Under the same assumption in Lemma \ref{lem5.7}, it holds that
	\begin{align}
	&\left\| |U_s''|^{-\f12} Y^\beta \pa_Y^k\Phi_{\text{app}}^f(\cdot~;c)\right\|_{L^2}\lesssim_A n^{\f{k-\beta}{3}-\f16},\label{5.18}\\
	&\left\| |U_s''|^{-\f12} Y^\beta \pa_Y^k\Psi_{\text{app}}^f(\cdot~;c)\right\|_{L^2}\lesssim_A n^{\f{k-1-\beta}{3}-\f16}.\label{5.19}
	\end{align}
\end{lemma}
\begin{proof}
From the pointwise estimate \eqref{5.10} and structural conditions \eqref{BL2} and \eqref{BL}, we have
$$\begin{aligned}
\left||U_s''|^{-\f12} Y^\beta \pa_Y^k\Phi_{\text{app}}^f\right|&\lesssim_A n^{\f{k-\beta}{3}}\left\|\f{|U_s'|}{|U_s''|^{\f12}} \right\|_{L^\infty}\left\|\f{e^{-s_0Y}}{|U_s'|} \right\|_{L^\infty}| n^{\f13}Y|^\beta e^{-\tau_1n^{\f13}Y+s_0Y}\\
&\lesssim_A n^{\f{k-\beta}{3}}| n^{\f13}Y|^\beta e^{-\f{\tau_1}{2}n^{\f13}Y}.
\end{aligned}
$$	
Therefore, it holds that
$$\left\||U_s''|^{-\f12} Y^\beta \pa_Y^k\Phi_{\text{app}}^f\right\|_{L^2}\lesssim_A n^{\f{k-\beta}{3}}\left\| |n^{\f13}Y|^\beta e^{-\f{\tau_1}{2}n^{\f13}Y}\right\|_{L^2}\lesssim_{A} n^{\f{k-\beta}{3}-\f16},
$$	
that is \eqref{5.18}. \eqref{5.19} can be proved similarly and then this completes
the  proof of the lemma.
\end{proof}
Now we are ready to give the proof of Proposition \ref{prop4.3}.\\

{\bf Proof of Proposition \ref{prop4.3}:}
First we consider $E_{1}^f(Y;{c})$. For $\a=A\vep^{\f18}$ and ${c}\in D_*$, we have
$\a\sim |c|\sim n^{-\f13} \sim\vep^{\f18}$. Thus from \eqref{4.1.7} we obtain that
$$|\Phi_{\text{app}}^s(0;{c})|\lesssim_A\vep^{\f18}.
$$
Then by using \eqref{5.18} and \eqref{5.19} in Lemma \ref{lem5.8}, we can estimate $\|E_1^f(\cdot~;{c})\|_{L^2}$ as follows.
$$\begin{aligned}
\|E_1^f(\cdot~;{c})\|_{L^2}&\lesssim_A \vep^{\f18}\bigg(\f{\a^2}{n}\|\pa_Y\Phi_{\text{app}}^f\|_{L^2}+{\vep}^{\f12}\|\pa_Y\Psi_{\text{app}}^f\|_{L^2}\\&\qquad\qquad+\f{\a}{n}\left\|\f{U_s}{Y}\right\|_{L^\infty}\|Y\Psi_{\text{app}}^f\|_{L^2}
+\frac{\a|c|}{n}\|\Psi_{\text{app}}^f\|_{L^2}+\f{\a}{n}\|\Phi_{\text{app}}^f\|_{L^2} \bigg)\\
&\lesssim_A \vep^{\f18}\left( \frac{\a^2}{n^{\f56}} +\f{\vep^{\f12}}{n^{\f16}}+\f{\a}{n^{\f{11}{6}}}+\frac{\a|c|}{n^{\f32}} +\f{\a}{n^{\f76}}\right)\\
&\lesssim_A \vep^{\f18}\left( \vep^{\f14+\f5{16}} +\vep^{\f12+\f1{16}}+\vep^{\f18+\f{11}{16}}+\vep^{\f{1}{8}+\f{7}{16}}\right)\lesssim_A \vep^{\f{11}{16}}.
\end{aligned}
$$
Similarly, we can obtain
$$
\begin{aligned}
\|E_2^f(\cdot~;{c})\|_{L^2}&\lesssim_A \vep^{\f18}\bigg( \f{\a^3}{n}\|\Phi_{\text{app}}^f\|_{L^2}+\a\left\| \f{U_s}{Y} \right\|_{L^\infty}\|Y\Phi_{\text{app}}^f\|_{L^2}\\
&\qquad\qquad+\a|\hat{c}|\|\Phi_{\text{app}}^f\|_{L^2} +\f{\a}{n} \|\Phi_{\text{app}}^f\|_{L^2}+\a\vep^{\f12}\|\Psi_{\text{app}}^f\|_{L^2} \bigg)\\
&\lesssim_A \vep^{\f18}\left(  \f{\a^3}{n^{\f76}}+\f{\a}{n^{\f12}}+\f{\a|\hat{c}|}{n^{\f16}}+\frac{\a\vep^{\f12}}{n^{\f12}}  \right)\lesssim_A\vep^{\f18}\left(  \vep^{\f38+\f7{16}}+\vep^{\f18+\f3{16}}+\vep^{\f{1}{4}+\f{1}{16}}+\vep^{\f58+\f3{16}}  \right)\\
&\lesssim_{A} \vep^{\f7{16}},
\end{aligned}
$$
and
$$
\begin{aligned}
\|F^f(\cdot~;{c})\|_{L^2}&\lesssim_A \vep^{\f18}\left( \a^2\|\Psi_{\text{app}}^f\|_{L^2}+\a\left\|\f{U_s}{Y}\right\|_{L^\infty}\|Y\Psi_{\text{app}}^f\|_{L^2}+\a|c|\|\Psi_{\text{app}}^f\|_{L^2}+\a\|\Phi_{\text{app}}^f\|_{L^2} \right)\\
&\lesssim_A \vep^{\f18}\left( \f{\a^2}{n^{\f12}}+\f{\a}{n^{\f56}}+\f{\a|c|}{n^{\f12}}+\f{\a}{n^{\f16}}  \right)\lesssim_A\vep^{\f18}\left( \vep^{\f14+\f3{16}}+\vep^{\f18+\f5{16}}+\vep^{\f18+\f1{16}}  \right)\\
&\lesssim_{A} \vep^{\f5{16}}.
\end{aligned}
$$
By combining these estimates, we obtain \eqref{3.4-2}. Finally, for $E_3^f(Y;{c})$, by using $|U_s(Y)-U_s'(0)Y|\lesssim Y^2$ and \eqref{5.18} with $k=\beta=2,$ we obtain 
$$
\begin{aligned}
	\||U_s''|^{-\f12}E_3^f(\cdot~;{c})\|_{L^2}&\lesssim_A \vep^{\f18}\bigg( \||U_s''|^{-\f12}Y^2\pa_Y^2\Phi_{\text{app}}^f\|_{L^2}+\|\Phi_{\text{app}}^f\|_{L^2}\\&\qquad\qquad+\vep^{\f12}\left\|\f{H_s'}{|U_s''|^{\f12}}\right\|_{L^\infty}\|\pa_Y\Psi_{\text{app}}^f\|_{L^2}+\vep^{\f12}\left\|\f{H_s''}{|U_s''|^{\f12}}\right\|_{L^\infty}\|\Psi_{\text{app}}^f\|_{L^2} \bigg)\\
	&\lesssim_{A} \vep^{\f18}\left( \f1{n^{\f16}}+\f{\vep^{\f12}}{n^{\f16}}+\f{\vep^{\f12}}{n^{\f12}}   \right)\lesssim_A\vep^{\f3{16}},
\end{aligned}
$$
that is \eqref{3.4-3}. The proof of Proposition \ref{prop4.3} is then completed.

\section{Construction of unstable modes}
In this section, we will construct the solution to the Orr-Sommerfeld system \eqref{eq1.1} with \eqref{BD1} based on the approximate growing mode
solution constructed in the previous section. For this, we first consider
\begin{equation}\label{6.1}
\left\{ \begin{aligned}
&\frac{i}{n}(\pa_Y^2-\a^2)^2\Phi+(U_s-\hat{c})(\pa_Y^2-\a^2)\Phi-\pa_Y^2U_s\Phi\\
& ~~~-\sqrt{\vep}  H_s(\pa_Y^2-\a^2)\Psi+\sqrt{\vep}\pa_Y^2H_s\Psi-\f{\a}{n}\pa_Y\left( (U_s-c)\Psi-H_s\Phi\right)-\f{ i\a^2}{n}\Phi=f_1,\\
&-(\pa_Y^2-\a^2)\Psi+i\a(U_s-c)\Psi-i\a H_s\Phi-\pa_Y\Phi=f_2,
\end{aligned}\right.
\end{equation}
with the Navier-slip boundary condition on the velocity field
\begin{align}\label{6.2}
\Phi|_{Y=0}=(\pa_Y^2-\a^2)\Phi|_{Y=0}=\Psi|_{Y=0}=0.
\end{align}
Here $(f_1,f_2)$ is a given inhomogeneous term. 
 Recall that $A_0$ and  the disk $D_*$  are given in Proposition \ref{prop4.1}. Define a weighted $L^2$-space 
$$L^2_w(\mathbb{R}_+)\triangleq\{f\in L^2(\mathbb{R}_+)\mid \|f\|_{L^2_w}\triangleq\||U_s''|^{-\f12}f\|_{L^2}<\infty\}.
$$

The solvability of  \eqref{6.1} together with \eqref{6.2} is given by the following
\begin{proposition}\label{prop5.1}
	Let $A\geq A_0$ be a fixed number. There exists $\vep_3\in (0,1)$, such that for $\vep\in (0,\vep_3)$, if $\a=A\vep^{\f18}$ and $c \in D_*$, then the following two statements hold.
\begin{enumerate}
\item If $\|f_1(\cdot~;c)\|_{L^2_w}+\|f_2(\cdot~;c)\|_{L^2}<\infty,$ then there exists a solution $(\Phi,\Psi)(\cdot~;c)\in H^2(\mathbb{R}_+)\cap H^1_0(\mathbb{R}_+)$ and $(\Phi,\Psi)(\cdot~;c)$ satisfies
\begin{equation}\label{6.4}
\begin{aligned}
\|(\pa_Y\Phi,~\a\Phi)\|_{L^2}+&\|(\pa_Y^2-\a^2)\Phi\|_{L^2}+\left|\pa_Y\Phi(0)\right|\\
&+\a\|\Psi\|_{L^2}+\a^{\f12}\|(\pa_Y\Psi,~\a\Psi)\|_{L^2}+\|(\pa_Y^2-
\a^2)\Psi\|_{L^2}\lesssim_A\frac{1}{\text{Im}\hat{c}}\|f_1\|_{L^2_w}+\|f_2\|_{L^2}.
\end{aligned}
\end{equation}
Moreover, if $f_1(\cdot~;c)$	and $f_2(\cdot~;c)$ are analytic in $D_*$ with values in $L^2_{w}(\mathbb{R}_+)$ and $L^2(\mathbb{R}_+)$ respectively, then the mapping
$$\Gamma_R(c)\triangleq \pa_Y\Phi(0;c): D_*\mapsto \mathbb{C}
$$
is analytic in $D_*$ as well.
\item  If $f_1=\pa_Yg_1$ or $f_1=i\a g_1$ for some $g_1\in L^2(\mathbb{R}_+)$, then there exists a solution $(\Phi,\Psi)(\cdot~;c)\in H^2(\mathbb{R}_+)\cap H^1_0(\mathbb{R}_+)$ and $(\Phi,\Psi)(\cdot~;c)$ satisfies
\begin{equation}\label{6.5}
\begin{aligned}
\|(\pa_Y\Phi,~\a\Phi)\|_{L^2}+&\|(\pa_Y^2-\a^2)\Phi\|_{L^2}+\left|\pa_Y\Phi(0)\right|\\
&+\a\|\Psi\|_{L^2}+\a^{\f12}\|(\pa_Y\Psi,~\a\Psi)\|_{L^2}+\|(\pa_Y^2-
\a^2)\Psi\|_{L^2}\lesssim_A\frac{1}{(\text{Im}\hat{c})^2}\|g_1\|_{L^2}+\|f_2\|_{L^2}.
\end{aligned}
\end{equation}
Moreover, if both $g_1(\cdot~;c)$	and $f_2(\cdot~;c)$ are analytic in $D_*$ with values in $L^2(\mathbb{R}_+)$, then the mapping
$$\Gamma_R(c)\triangleq \pa_Y\Phi(0;c): D_*\mapsto \mathbb{C}
$$
is analytic in $D_*$ as well.
\end{enumerate}
\end{proposition}
Note that as mentioned in the Introduction, we will decompose the operator $\mbox{OS}$ according to the structure of
the source term by $\mbox{OS}_s$ and $\mbox{OS}_d$. Hence, the solvability of $\mbox{OS}$ depends on the solvability of $\mbox{OS}_s$ and
$\mbox{OS}_d$ given in the following two subsections.

\begin{remark}\label{rmk4.2}
	As one can see from the proof, the argument also holds for a wider range of $\beta$ when
	$\a\sim \vep^\beta$  and  $\a \text{Im}\hat{c}\sim \vep^{\f14}$. First of all, to justify the expansion in \eqref{4.1.7} requires $\beta> 1/12$
	and $\text{Re} \, \hat{c}\sim \a$, so that $\a |\hat{c}\log \text{Im}\, \hat{c }|<\!< \text{Im}\,  \hat{c}$.
	In addition, we require that $c$ lies in $\Sigma_d\cap \Sigma_s$. Here $\Sigma_d$ and $\Sigma_s$ belong to resolvent sets of $\mbox{OS}_d$ and $\mbox{OS}_s$ given in Lemma \ref{lem6.1} and Lemma \ref{lem6.2} respectively.  Moreover, in view of \eqref{small1} and \eqref{small}, we require the smallness of the factor in front of $\CE_{k-1}$ in order to achieve the convergence of iteration. This leads to 
	require the smallness of the following quantity
	\begin{align}
\frac{1}{\a^{\f12}n(\text{Im}\hat{c})^2}\ll1.\nonumber
	\end{align}
This always holds for $\a\sim\vep^{\beta}$ with $\beta\in[\f1{12},\f1{8}]$. 
To be more precise, in Section 5, we will  prove the statement  of Theorem \ref{thm1.1} in the regime where $\a\sim \vep^{\beta}$ with $\beta\in (3/28,1/8).$
\end{remark}

\subsection{Solvability of $\mbox{OS}_d$}
In this subsection, we will consider the operator $\mbox{OS}_d$  when the source term has strong decay in $Y$.
For this,  we study the following system of equations $\mbox{OS}_d(\phi,\varrho)=(q_1,q_2)$:
\begin{equation}
\left\{
\begin{aligned}\label{6-1.1}
&\frac{i}{n}(\pa_Y^2-\a^2)^2\phi+(U_s-\hat{c})(\pa_Y^2-\a^2)\phi-U_s''\phi=q_1,~Y>0,\\
&-(\pa_Y^2-\a^2)\varrho+(U_s-c)\varrho-i\a H_s\phi-\pa_Y\phi=q_2,\\
&\phi|_{Y=0}=(\pa_Y^2-\a^2)\phi|_{Y=0}=\varrho|_{Y=0}=0.
\end{aligned}\right.
\end{equation}
Here $(q_1,q_2)$ is a given inhomogeneous source term. The result can be stated in  the following lemma.

\begin{lemma}\label{lem6.1}
There exist $\a_2\in (0,\a_1)$ and $\tau_2\in (0,1)$, such that for any $\a\in (0,\a_2)$ and $c\in \Sigma_d\triangleq\{c\in \mathbb{C}\mid \text{Im}\hat{c}\geq \tau_2^{-1}n^{-1},~|c|< \gamma_1 \}$, where $\g_1$ is given in Proposition \ref{prop3.2}, if $\|q_{1}\|_{L^2_w}+\|q_2\|_{L^2}<\infty,$ then \eqref{6-1.1} has a unique solution $(\phi,\varrho)\in H^2(\mathbb{R}_+)\cap H^1_0(\mathbb{R}_+)$ and $(\phi,\varrho)$ satisfies the following estimates
	\begin{align}
	\|(\pa_Y^2-\a^2)\phi\|_{L^2_w}+\|(\pa_Y\phi,~\a\phi)\|_{L^2}&\leq\frac{C}{\text{Im}\hat{c}}\|q_1\|_{L^2_w},\label{6-1.2}\\
	\|\pa_Y(\pa_Y^2-\a^2)\phi\|_{L^2_w}&\leq \frac{C n^{\f12}}{(\text{Im}\hat{c})^{\f12}}\|q_1\|_{L^2_w},\label{6-1.3}\\
	\|(\pa_Y^2-\a^2)\varrho\|_{L^2}+\a^{\f12}\|(\pa_Y\varrho,~\a\varrho)\|_{L^2}+\a\|\varrho\|_{L^2}&\leq C\left(  \f{1}{\text{Im}\hat{c}}\|q_1\|_{L^2_w}+\|q_2\|_{L^2} \right).\label{6-1.4} 
	\end{align}
Moreover, the solution operator
$$\mbox{OS}_d^{-1}(c): L^2_w(\mathbb{R}_+)\times L^2(\mathbb{R}_+)\mapsto H^2(\mathbb{R}_+)\times H^2(\mathbb{R}_+)
$$
is analytic in $c$.
\end{lemma}
\begin{remark}\label{rmk4.4}
	Lemma \ref{lem6.1} indicates that if we replace the no-slip boundary condition by the Navier boundary condition, then the corresponding homogeneous Orr-Sommerfeld system of equations has no unstable eigenvalue with Gevrey regularity. 
\end{remark}
{\bf Proof of Lemma \ref{lem6.1}}:
 We first show the a priori estimates. Denote $\omega=(\pa_Y^2-\a^2)\phi. $
 Note that $U_s''<0$ from the strong concave condition in \eqref{BL}. Taking inner product of \eqref{6-1.1} and $\frac{\omega}{U_s''}$, we deduce that
\begin{equation}\label{6-1.5}
\begin{aligned}
\|(\pa_Y\phi,~\a\phi)\|_{L^2}^2+\int_0^\infty\frac{(U_s-\hat{c})}{U_s''}|\omega|^2\dd Y+\langle \frac{i}{n}(\pa_Y^2-\a^2)\omega, \f{\omega}{U_s''}\rangle=\langle q_1, \f{\omega}{U_s''}\rangle.
\end{aligned}
\end{equation}	
We use  integration by parts and  the boundary condition $\omega|_{Y=0}=0$ to obtain 
$$
\langle \frac{i}{n}(\pa_Y^2-\a^2)\omega, \f{\omega}{U_s''}\rangle=\f{i}{n}\left\|(\pa_Y\omega,~\a\omega)\right\|_{L^2_w}^2+\f{i}{n}\langle \pa_Y\omega,~\frac{\omega U_s'''}{(U_s'')^2}    \rangle.
$$
Then the imaginary and real part of \eqref{6-1.5} give respectively
\begin{align}
\frac{1}{n}\left\|(\pa_Y\omega,~\a\omega)\right\|_{L^2_w}^2+\text{Im}\hat{c}\|\omega\|_{L^2_w}^2&=\text{Im}\left(\langle q_1, \f{\omega}{U_s''}\rangle\right)+\text{Im}\left( -\f{i}{n}\langle \pa_Y\omega,~\frac{\omega U_s'''}{(U_s'')^2}    \rangle\right),\label{6-1.6}\\
\|(\pa_Y\phi,~\a\phi)\|_{L^2}^2+\int_0^\infty\frac{U_s-\text{Re}\hat{c}}{U_s''}|\omega|^2\dd Y&=\text{Re}\left(\langle q_1, \f{\omega}{U_s''}\rangle\right)+\text{Re}\left( -\f{i}{n}\langle \pa_Y\omega,~\frac{\omega U_s'''}{(U_s'')^2}    \rangle\right).\label{6-1.7}
\end{align}
By using Cauchy-Schwarz and the structural condition of boundary layer \eqref{BL}, we have
\begin{equation}\left| \langle q_1, \f{\omega}{U_s''}\rangle\right|\leq \|q_1\|_{L^2_w}\|\omega\|_{L^2_w},\label{6-1.7.1}
\end{equation}
and
\begin{align}\f{1}{n}\left|\langle \pa_Y\omega,~\frac{\omega U_s'''}{(U_s'')^2}    \rangle\right|\leq \frac{1}{n}\left\|\frac{U_s'''}{U_s''}\right\|_{L^\infty}\|\pa_Y\omega\|_{L^2_w}\|\omega\|_{L^2_w}\lesssim \frac{1}{n}\|\pa_Y\omega\|_{L^2_w}\|\omega\|_{L^2_w}.\label{6-1.7.2}
\end{align}
Then from \eqref{6-1.6}, we have
\begin{equation}
\begin{aligned}
\frac{1}{n}\left\|(\pa_Y\omega,~\a\omega)\right\|_{L^2_w}^2+\text{Im}\hat{c}\|\omega\|_{L^2_w}^2\lesssim& \|q_1\|_{L^2_w}\|\omega\|_{L^2_w}+\frac{1}{n}\|\pa_Y\omega\|_{L^2_w}\|\omega\|_{L^2_w}\label{6-1.8}\\
\leq &\frac{1}{2n}\|\pa_Y\omega\|_{L^2_w}^2+\text{Im}\hat{c}\left(\frac{1}{2}+\frac{C}{n\text{Im}\hat{c}}\right)\|\omega\|_{L^2_w}^2+\frac{C}{\text{Im}\hat{c}}\|q_1\|_{L^2_w}^2.
\end{aligned}
\end{equation}
Taking $\tau_2>0$ suitably small so that $\frac{C}{n\text{Im}\hat{c}}\leq C\tau_2<\f14$ for $c\in\Sigma_d,$  the first two terms on the right hand side of \eqref{6-1.8} can be absorbed by the terms on the left hand side. We then have
\begin{align}
\|\omega\|_{L^2_w}\lesssim \frac{1}{\text{Im}\hat{c}}\|q_1\|_{L^2_w},~ \left\|(\pa_Y\omega,~\a\omega)\right\|_{L^2_w}\lesssim \frac{n^{\f12}}{(\text{Im}\hat{c})^{\f12}}\|q_1\|_{L^2_w}.\label{6-1.9}
\end{align}
Moreover, from \eqref{6-1.7}, \eqref{6-1.7.1} and \eqref{6-1.7.2}, we have
$$\begin{aligned}
\|(\pa_Y\phi,~\a\phi)\|_{L^2}^2&\lesssim\|\omega\|_{L^2_w}^2+\f{1}{n}\|\pa_Y\omega\|_{L^2_w}\|\omega\|_{L^2_w}+\|q_1\|_{L^2_w}\|\omega\|_{L^2_w}\\
&\lesssim\left(\frac{1}{(\text{Im}\hat{c})^2}+ \frac{1}{n^{\f12}(\text{Im}\hat{c})^\f{3}{2}}+\frac{1}{\text{Im}\hat{c}}\right)\|q_1\|_{L^2_w}^2\\
&\lesssim \frac{1}{(\text{Im}\hat{c})^2}\|q_1\|_{L^2_w}^2,
\end{aligned}
$$
where we have used \eqref{6-1.9}. Combining this with \eqref{6-1.9}, we obtain \eqref{6-1.2} and \eqref{6-1.3}. For \eqref{6-1.4}, we can apply the estimates \eqref{est-mag1} and \eqref{est-mag2}  to $\varrho$ with $f=q_2+i\a H_s\phi+\pa_Y\phi$. Then we have
$$\begin{aligned}
\|(\pa_Y^2-\a^2)\varrho\|_{L^2}+\a^{\f12}\|(\pa_Y\varrho,~\a\varrho)\|_{L^2}+\a\|\varrho\|_{L^2}\lesssim& \|q_2\|_{L^2}+\|\pa_Y\phi\|_{L^2}+\a\|\phi\|_{L^2}\\
\lesssim& \f{1}{\text{Im}\hat{c}}\|q_1\|_{L^2_w}+\|q_2\|_{L^2},
\end{aligned}
$$
which is \eqref{6-1.4}.
The uniqueness follows from the a priori estimates. For existence,  from the previous argument, the  estimates \eqref{6-1.2}-\eqref{6-1.3} indeed hold for all $c$ in  $\{c\in \mathbb{C}\mid Im\hat{c}\geq \tau_2^{-1}n^{-1},~ |c|\leq K \}$ for any given $K$, with constant $C$ in \eqref{6-1.2} and \eqref{6-1.3} depending on $K$. Also for suitably large $\text{Im}\hat{c}$, it is not difficult to construct the solution $\phi$ to the first equation in \eqref{6-1.1}. Then by applying the method of continuity, we can show the existence of $\phi$ solving $\eqref{6-1.1}_1$ for $c$  in $\Sigma_d$. With $\phi$ in hand, the existence of $\varrho$ is guaranteed by Proposition \ref{prop3.2}. Therefore, the existence part of statement follows. 

Finally, as for the analyticity, we rewrite the first equation in \eqref{6-1.1} as
$$\begin{aligned}
-c\omega+\frac{i}{n}(\pa_Y^2-\a^2)\omega-(U_s-\frac{i}{n})\omega-U_s''\Delta_\a^{-1}\omega=q_1,~ \omega|_{Y=0}=0,
\end{aligned}
$$
where $\Delta_\a^{-1}: L^2(\mathbb{R}_+)\mapsto H^2(\mathbb{R}_+)\cap H^1_0(\mathbb{R}_+)$ is the solution operator to
$$(\pa_Y^2-\a^2)\phi=\omega,~\phi|_{Y=0}=0.
$$
Denote the operator
$$\begin{aligned}
&\CL:H^2_w(\mathbb{R}_+)\cap H^1_0(\mathbb{R}_+)\rightarrow L^2_w(\mathbb{R}_+)\\
&\CL(\omega)=\frac{i}{n}(\pa_Y^2-\a^2)\omega-(U_s-\frac{i}{n})\omega-U_s''\Delta_\a^{-1}\omega.
\end{aligned}
$$
Here the weighted space $H_w^2(\mathbb{R}_+):=\{\omega\in L^2_w(\mathbb{R}_+)|\pa_Y^j\omega\in L^2_w(\mathbb{R}_+),~j=1,2\}$.
Then from \eqref{6-1.2} we know that for $c$ in $\Sigma_d$, the resolvent $(-c+\CL)^{-1}$ exists hence is analytic in $c$ with values in $B(L^2_w)$,  which is the space of linear bounded operators on $L^2_w$ . Also, $\omega=(-c+\CL)^{-1}q_1$ is analytic in $c$ with values in $L^2_w(\mathbb{R}_+)$. Note that $\phi=\Delta_{\a}^{-1}\omega$. By using the boundedness of $\Delta_{\a}^{-1}:L^2(\mathbb{R}_+)\mapsto H^2(\mathbb{R}_+)\cap H^1_0(\mathbb{R}_+)$ and Lemma \ref{lem.ap1}, the analyticity of $\phi$ can be obtained as well. Finally, given that $\phi$ is analytic in $c$ with value in $H^2(\mathbb{R}_+)$, the analyticity of $\varrho$ follows from the analyticity of the solution operator constructed in Proposition \ref{prop3.2} and an application of Lemma \ref{lem.ap1}. Therefore, the solution operator $\mbox{OS}_d^{-1}(c)$ is analytic and the proof of Lemma \ref{lem6.1} is completed.\qed 

\subsection{Solvability of  $\mbox{OS}_s$} In this subsection, we  study the part of $\mbox{OS}$ that has
differential structure. For this, we consider $\mbox{OS}_s(\xi,\vartheta)=(h_1,h_2),$ defined by
\begin{equation}
\left\{
\begin{aligned}\label{6-2.1}
&\frac{i}{n}(\pa_Y^2-\a^2)^2\xi+\pa_Y\left((U_s-\hat{c})\pa_Y\xi \right)-\a^2(U_s-\hat{c})\xi+\pa_YR_1(\xi,\vartheta)+i\a R_2(\xi,\vartheta)=h_1,~Y>0,\\
&-(\pa_Y^2-\a^2)\vartheta+(U_s-c)\vartheta-i\a H_s\xi-\pa_Y\xi=h_2,\\
&\xi|_{Y=0}=(\pa_Y^2-\a^2)\xi|_{Y=0}=\vartheta|_{Y=0}=0,
\end{aligned}\right.
\end{equation}
with a given inhomogeneous source term $(h_1,h_2)$.
Here 
\begin{align}\label{6-2.2}
R_1(\xi,\vartheta)=-\sqrt{\vep}H_s\pa_Y\vartheta+\sqrt{\vep}\pa_YH_s\vartheta-\frac{\a}{n}(U_s-\hat{c})\vartheta+\f{\a}{n}H_s\xi,~ R_2(\xi,\vartheta)=-\frac{\a}{n}\xi-i\a\sqrt{\vep}H_s\vartheta,
\end{align}
represent part of effects from the magnetic field. The following lemma clarifies the solvability of \eqref{6-2.1}. Recall $\a_1$ and $\gamma_1$ are the numbers given in Proposition \ref{prop3.2}. 
\begin{lemma}\label{lem6.2}
	There exists  $\tau_3\in (0,1)$, such that for any $\a\in (0,\a_1)$ and $c\in \Sigma_s\triangleq\{c\in \mathbb{C}\mid \text{Im}\hat{c}\geq \tau_3^{-1}\a^{-\f12}n^{-1},~|c|< \gamma_1 \}$, if $\|(h_1,h_2)\|_{L^2}<\infty,$ then \eqref{6-2.1} has a unique solution $(\xi,\vartheta)\in H^2(\mathbb{R}_+)\cap H^1_0(\mathbb{R}_+)$ and $(\xi,\vartheta)$ satisfies the following estimates
	\begin{align}
	&\|(\pa_Y\xi,~\a\xi)\|_{L^2}\lesssim\frac{1}{\a\text{Im}\hat{c}}\|h_1\|_{L^2}+\frac{1}{\a^{\f12}n\text{Im}\hat{c}}\|h_2\|_{L^2},\label{6-2.3}\\
	&\|(\pa_Y^2-\a^2)\xi\|_{L^2}\lesssim \frac{n^{\f12} }{\a(\text{Im}\hat{c})^{\f12}}\|h_1\|_{L^2}+\frac{1}{\left(\a n\text{Im}\hat{c}\right)^{\f12}}\|h_2\|_{L^2},\label{6-2.4}\\
		&\|(\pa_Y^2-\a^2)\vartheta\|_{L^2}+\a^{\f12}\|(\pa_Y\vartheta,~\a\vartheta)\|_{L^2}+\a\|\vartheta\|_{L^2}\lesssim \frac{1 }{\a\text{Im}\hat{c}}\|h_1\|_{L^2}+\|h_2\|_{L^2}.\label{6-2.5}
	\end{align}
Moreover, if we further have $h_1=\pa_Yg_1$ or $h_1=i\a g_1$ for some $g_1\in L^2(\mathbb{R}_+)$, then it holds that
	\begin{align}
&\|(\pa_Y\xi,~\a\xi)\|_{L^2}\lesssim\frac{1}{\text{Im}\hat{c}}\|g_1\|_{L^2}+\frac{1}{\a^{\f12}n\text{Im}\hat{c}}\|h_2\|_{L^2},\label{6-2.7}\\
&\|(\pa_Y^2-\a^2)\xi\|_{L^2}\lesssim \frac{n^{\f12} }{(\text{Im}\hat{c})^{\f12}}\|g_1\|_{L^2}+\frac{1}{\left(\a n\text{Im}\hat{c}\right)^{\f12}}\|h_2\|_{L^2},\label{6-2.8}\\
&\|(\pa_Y^2-\a^2)\vartheta\|_{L^2}+\a^{\f12}\|(\pa_Y\vartheta,~\a\vartheta)\|_{L^2}+\a\|\vartheta\|_{L^2}\lesssim \frac{1 }{\text{Im}\hat{c}}\|g_1\|_{L^2}+\|h_2\|_{L^2}.\label{6-2.9} 
\end{align}
Furthermore, the solution operator
	$$\mbox{OS}_s^{-1}(c): L^2(\mathbb{R}_+)\times L^2(\mathbb{R}_+)\mapsto H^2(\mathbb{R}_+)\times H^2(\mathbb{R}_+)
	$$
	is analytic in $c$.
\end{lemma}
Before proving this Lemma, we first bound $L^2$-norm of $R_1$ and $R_2$ defined in \eqref{6-2.2}.
\begin{lemma}
	Assume that $\vartheta|_{Y=0}=0$. Then it holds 
	\begin{align}\label{6-2.11}
	\|R_1(\xi,\vartheta)\|_{L^2}+	\|R_2(\xi,\vartheta)\|_{L^2}\lesssim \frac{1}{n}\left(\|(\pa_Y\xi,~\a\xi\|_{L^2}+\|(\pa_Y\vartheta,~\a\vartheta\|_{L^2} \right).
	\end{align}
\end{lemma}
\begin{proof}
We directly compute
$$\begin{aligned}
\|R_1(\xi,\vartheta)\|_{L^2}&\lesssim \sqrt{\vep}\left(\|\pa_Y\vartheta\|_{L^2}+\|Y\pa_YH_s\|_{L^\infty}\left\|\f{\vartheta}{Y}\right\|_{L^2}\right)+\frac{1}{n} \|\a\xi,~\a\vartheta\|_{L^2}\\
&\lesssim \frac{\a}{n}\|\pa_Y\vartheta\|_{L^2}+\frac{1}{n} \|\a\xi,~\a\vartheta\|_{L^2},
\end{aligned}
$$
where we have used the Hardy inequality $\|\f{\vartheta}{Y}\|_{L^2}\leq 2\|\pa_Y\vartheta\|_{L^2}$ and $\sqrt{\vep}=\f\a{n}$. Similarly, we find
$$
\|R_2(\xi,\vartheta)\|_{L^2}\lesssim \frac{\a}{n}\|\xi\|_{L^2}+\frac{\a^2}{n}\|\vartheta\|_{L^2}.
$$
The bound \eqref{6-2.11} directly follows from the fact $\a\lesssim1$.
\end{proof}
{\bf Proof of Lemma \ref{lem6.2}}: Let us focus on the a priori estimates. Multiplying both sides of the first equation in \eqref{6-2.1} by $\bar{\xi}$ gives
\begin{equation}\label{6-2.12}
\begin{aligned}
	\frac{i}{n}\|(\pa_Y^2-\a^2)\xi\|_{L^2}^2-\int_0^\infty(U_s-\hat{c})\left(|\pa_Y\xi|^2+\a^2|\xi|^2\right)\dd Y=\langle h_1,\xi\rangle+ \langle R_1,\pa_Y\xi\rangle+\langle R_2,i\a\xi\rangle.
\end{aligned}
\end{equation}
The imaginary part of \eqref{6-2.12} implies that
\begin{equation}\label{6-2.13}
\begin{aligned}
	&\frac{1}{n}\|(\pa_Y^2-\a^2)\xi\|_{L^2}^2+\text{Im}\hat{c}\|(\pa_Y\xi,~\a\xi)\|_{L^2}^2\\
	&\qquad=\text{Im}(\langle h_1,\xi\rangle)+\text{Im} (\langle R_1,\pa_Y\xi\rangle)+\text{Im}(\langle R_2,i\a\xi\rangle)\\
	&\qquad\leq\frac{\text{Im}\hat{c}}{4}\|(\pa_Y\xi,~\a\xi)\|_{L^2}^2+ \f{C}{\text{Im}\hat{c}}\|(R_1,R_2)\|_{L^2}^2+|\text{Im}(\langle h_1,\xi\rangle)|\\
	&\qquad\leq \frac{\text{Im}\hat{c}}{4}\|(\pa_Y\xi,~\a\xi)\|_{L^2}^2+ \f{C}{n^2\text{Im}\hat{c}}\left(\|(\pa_Y\xi,~\a\xi)\|_{L^2}^2+\|(\pa_Y\vartheta,~\a\vartheta)\|_{L^2}^2\right)+|\text{Im}(\langle h_1,\xi\rangle)|.
\end{aligned}
\end{equation}
Here we have used \eqref{6-2.11} in the last inequality. By applying \eqref{est-mag1} and \eqref{est-mag2} to $\vartheta$ with $f=h_2+\pa_Y\xi+i\a H_s\xi$, we have
\begin{equation}\label{6-2.14}
\begin{aligned}
\|(\pa_Y^2-\a^2)\vartheta\|_{L^2}+\a^{\f12}\|(\pa_Y\vartheta,~\a\vartheta)\|_{L^2}+\a\|\vartheta\|_{L^2}\lesssim \|h_2\|_{L^2}+\|(\pa_Y\xi,~\a\xi)\|_{L^2}.
\end{aligned}
\end{equation} 
With this inequality, we deduce from \eqref{6-2.13} that
\begin{equation}
\begin{aligned}
&\frac{1}{n}\|(\pa_Y^2-\a^2)\xi\|_{L^2}^2+\text{Im}\hat{c}\|(\pa_Y\xi,~\a\xi)\|_{L^2}^2\\
&\qquad\qquad\leq \left( \f{Im\hat{c}}{4}+\frac{C}{\a n^2\text{Im}\hat{c}}  \right)\|(\pa_Y\xi,~\a\xi)\|_{L^2}^2+\frac{C}{\a n^2\text{Im}\hat{c}}\|h_2\|_{L^2}^2+|\text{Im}(\langle h_1,\xi\rangle)|.	\nonumber
\end{aligned}
\end{equation}
By taking $\tau_3\in (0,1)$ suitably small so that
$$\frac{C}{\a n^2\text{Im}\hat{c}}\leq \frac{C\text{Im}\hat{c}}{\a n^2(\text{Im}\hat{c})^2}\leq C\tau_3^2\text{Im}\hat{c}<\frac{1}{4}\text{Im}\hat{c}~\text{ for }c\in \Sigma_s,
$$
we get
\begin{align}\label{6-2.15}
\frac{1}{n}\|(\pa_Y^2-\a^2)\xi\|_{L^2}^2+\text{Im}\hat{c}\|(\pa_Y\xi,~\a\xi)\|_{L^2}^2\leq \frac{C}{\a n^2\text{Im}\hat{c}}\|h_2\|_{L^2}^2+|\text{Im}(\langle h_1,\xi\rangle)|.
\end{align}
To bound the last term in the above inequality, we note that
\begin{align}\label{6-2.16}
|\text{Im}(\langle h_1,\xi\rangle)|\leq \|h_1\|_{L^2}\|\xi\|_{L^2}
\end{align}
for general $h_1\in L^2(\mathbb{R}_+)$ and  \begin{align}\label{6.2-17}
|\text{Im}(\langle h_1,\xi\rangle)|\leq|\langle g_1,\pa_Y\xi\rangle|\leq  \|g_1\|_{L^2}\|\pa_Y\xi\|_{L^2},
\end{align}
for $h_1=\pa_Yg_1$ with some $g_1\in L^2(\mathbb{R}_+)$. Then by \eqref{6-2.15} and \eqref{6-2.16}, we obtain for general $h_1$ that 
$$
\begin{aligned}
\|(\pa_Y\xi,~\a\xi)\|_{L^2}&\lesssim \frac{1}{\a\text{Im}\hat{c}}\|h_1\|_{L^2}+\f{1}{\a^{\f12}n\text{Im}\hat{c}}\|h_2\|_{L^2},\\
\|(\pa_Y^2-\a^2)\xi\|_{L^2}
&\lesssim \frac{n^{\f12}}{\a(\text{Im}\hat{c})^{\f12}}\|h_1\|_{L^2}+\f{1}{(\a n\text{Im}\hat{c})^{\f12}}\|h_2\|_{L^2},
\end{aligned}
$$
which give \eqref{6-2.3} and \eqref{6-2.4} respectively. Moreover, from \eqref{6-2.14} and  \eqref{6-2.3}, we obtain that
$$
\begin{aligned}
&\|(\pa_Y^2-\a^2)\vartheta\|_{L^2}+\a^{\f12}\|(\pa_Y\vartheta,~\a\vartheta)\|_{L^2}+\a\|\vartheta\|_{L^2}\\
&\qquad\quad\lesssim \|h_2\|_{L^2}+\|(\pa_Y\xi,~\a\xi)\|_{L^2}\lesssim \frac{1}{\a\text{Im}\hat{c}}\|h_1\|_{L^2}+\left(\f{1}{\a^{\f12}n\text{Im}\hat{c}}+1\right)\|h_2\|_{L^2}\\
&\qquad\quad\lesssim \frac{1}{\a\text{Im}\hat{c}}\|h_1\|_{L^2}+\|h_2\|_{L^2},
\end{aligned}
$$
where we have used $\f{1}{\a^{\f12}n\text{Im}\hat{c}}\lesssim \tau_3\lesssim1.$ That is \eqref{6-2.5}. By using \eqref{6-2.14}, \eqref{6-2.15} and \eqref{6.2-17}, we can obtain \eqref{6-2.7}, \eqref{6-2.8} and \eqref{6-2.9} similarly. This completes the proof of the statement of the a priori estimate. With this, the existence and analyticity can be proved by a similar argument as in the proof of Lemma \ref{lem6.1}. We omit the detail for brevity. The proof of Lemma \ref{lem6.2} is completed.\qed

\subsection{Convergence of $\mbox{OS}_d$-$\mbox{OS}_s$ iteration}
In this part, we will construct the solution $(\Phi,\Psi)$ to the Orr-Sommerfeld system of equations \eqref{6.1} with Navier boundary condition \eqref{6.2} and then give the proof of Proposition \ref{prop5.1}. Before doing this, we illustrate the construction of solution as follows. Firstly, we consider the case $\|f_1\|_{L^2_w}+\|f_2\|_{L^2}<\infty$. We then define the zeroth-step approximation $(\phi_0,\varrho_0)$ as the solution to the following equation:
\begin{align}\label{6-3.12}
\mbox{OS}_d(\phi_0,\varrho_0)=(f_1,f_2).
\end{align}
Note that
$$\text{OS}(\phi_0,\varrho_0)=(f_1,f_2)+\bigg(\pa_YR_1(\phi_0,\varrho_0)+i\a R_2(\phi_0,\varrho_0),~ 0\bigg),
$$
where $R_1$ and $R_2$ are the functions defined in \eqref{6-2.2}. Observe that the error term $(\pa_YR_1(\phi_0,\varrho_0)+i\a R_2(\phi_0,\varrho_0),~ 0)$ has a differential structure that
 matches the operator $\mbox{OS}_s$. In order to eliminate this error term, we set $(\xi_1,\vartheta_1)$ to be the solution to the following equation:
$$\mbox{OS}_s(\xi_1,\vartheta_1)=\bigg(-\pa_YR_1(\phi_0,\varrho_0)-i\a R_2(\phi_0,\varrho_0),~ 0\bigg),
$$
which yields an error term 
$$\text{OS}(\xi_1,\vartheta_1)-\mbox{OS}_s(\xi_1,\vartheta_1)=(-\pa_YU_s\pa_Y\xi_1-\xi_1\pa_Y^2U_s,~ 0).
$$
Observe that $-\pa_YU_s\pa_Y\xi_1-\pa_Y^2U_s\xi_1\in L^2_w.$ We set $(\phi_1,\varrho_1)$ to be the solution to the following system
$$\mbox{OS}_d(\phi_1,\varrho_1)=(\pa_YU_s\pa_Y\xi_1+\xi_1\pa_Y^2U_s,~0),
$$
and further choose $(\Phi_1,\Psi_1)=(\xi_1,\vartheta_1)+(\phi_1,\varrho_1)$ as the first-step approximation, which yields 
$$\text{OS}(\phi_0+\Phi_1,\varrho_0+\Psi_1)-(f_1,f_2)=\bigg(\pa_YR_1(\phi_1,\varrho_1)+i\a R_2(\phi_1,\varrho_1),~0\bigg).
$$
By repeating this procedure, we can construct inductively the following solution sequence $$(\Phi_k,\Psi_k)=(\xi_k,\vartheta_k)+(\phi_k,\varrho_k),~ k\geq 1,$$ where $(\xi_k,\vartheta_k)$ solves
\begin{align}\label{6-3.1}
\mbox{OS}_s(\xi_k,\vartheta_k)=\bigg(-\pa_YR_1(\phi_{k-1},\varrho_{k-1})-i\a R_2(\phi_{k-1},\varrho_{k-1}),~0\bigg),~
\end{align}
and
$(\phi_k,\varrho_k)$ solves 
\begin{align}\label{6-3.2}
\mbox{OS}_d(\phi_k,\varrho_k)=(\pa_YU_s\pa_Y\xi_k+\xi_k\pa_Y^2U_s,~0).
\end{align}
For any positive integer $N\geq 1$, it holds that
$$\mbox{OS}\left(\phi_0+\sum_{k=1}^N\Phi_k,\varrho_0+\sum_{k=1}^N\Psi_k\right)=(f_1,f_2)+\bigg(\pa_YR_N(\phi_1,\varrho_1)+i\a R_N(\phi_1,\varrho_1),~0\bigg).
$$
Therefore, it remains to show the convergence of $\sum_{k=1}^\infty \Phi_k$ and $\sum_{k=1}^\infty \Psi_k$ in 
some suitable function space. Indeed, given the convergence, it is straightforward to check that 
\begin{align}\label{solu}
(\Phi,\Psi)=(\phi_0+\sum_{k=1}^\infty \Phi_k, \varrho_0+\sum_{k=1}^\infty \Psi_k)
\end{align}
defines a solution to \eqref{6.1} together with \eqref{6.2}.

We now turn to  consider the second statement in the Proposition \ref{prop5.1}, i.e. $f_1=\pa_Yg_1$ or $f_1=i\a g_1$ with some $g_1\in L^2(\mathbb{R}_+)$. For this case, we only need to have  one more step before  the zeroth-approximation \eqref{6-3.12}. That is, we first solve $(\xi_0,\vartheta_0)$ from the equation
\begin{align}\label{6-3.13}
\mbox{OS}_s(\xi_0,\vartheta_0)=(f_1,f_2),
\end{align}
which yields 
$$\text{OS}(\xi_0,\vartheta_0)-(f_1,f_2)=(-\pa_YU_s\pa_Y\xi_0-\xi_0\pa_Y^2U_s,~0).
$$
By noting that the source term  in the above equation belongs to $L^2_w(\mathbb{R}_+)$, we can define $(\tilde{\Phi},\tilde{\Psi})$ as the solution to the following Orr-Sommerfeld system
\begin{align}
\text{OS}(\tilde{\Phi},\tilde{\Psi})=(\pa_YU_s\pa_Y\xi_0+\xi_0\pa_Y^2U_s,~0),\label{6-3.14-1}
\end{align}
which has been solved in the first part of Proposition \ref{prop5.1}.  In summary, we have
\begin{align}
({\Phi},{\Psi})=(\xi_0,\vartheta_0)+(\tilde{\Phi},\tilde{\Psi})\label{6-3.14-2}
\end{align}
as the desired solution to \eqref{6.1} with \eqref{6.2}.\\

{\bf Proof of Proposition \ref{prop5.1}}: Recall that $\a=A\vep^{\f18}$, $c_*$ is the number given in \eqref{c} and $D_*$ is the disk given in Proposition \ref{prop4.1}. For $c\in D_*$, we have $Im\hat{c}\sim_A |c|\sim_A n^{-\f13}\sim_A \vep^{\f18}$. Therefore, we can take $\vep_3$ suitably small so that for $\vep\in (0,\vep_3)$, there holds
$$\a\in (0,\min\{\a_1,\a_2\}),~ |c|< \gamma_1,~n\text{Im}\hat{c}\gtrsim_A\vep^{-\f14}\geq 2\tau_2^{-1},~n\a^{\f12}\text{Im}\hat{c}\gtrsim_A\vep^{-\f3{16}}\geq 2\tau_3^{-1}.
$$
Here the numbers $\a_2$, $\tau_2$, $\gamma_1$ are given in Lemma \ref{lem6.1} and $\tau_3$ is given in Lemma \ref{lem6.2}. Hence, we have $D_*\subsetneqq \Sigma_{d}\cap\Sigma_{s}$, where $\Sigma_{d}$ and $\Sigma_{s}$ are resolvent sets of operators $\mbox{OS}_d$ and $\mbox{OS}_s$ which are given in Lemma \ref{lem6.1} and Lemma \ref{lem6.2} respectively. Since $(\xi_k,\vartheta_k)$ solves the equation \eqref{6-3.1}, by applying \eqref{6-2.7}, \eqref{6-2.8} and \eqref{6-2.9} to $(\xi_k,\vartheta_k)$ with $$h_1=-\pa_YR_1(\phi_{k-1},\varrho_{k-1})-i\a R_2(\phi_{k-1},\varrho_{k-1}),~h_2=0,$$
we obtain 
\begin{equation}\label{6-3.3}
\begin{aligned}
\|(\pa_Y\xi_k,\a\xi_k)\|_{L^2}&\lesssim \frac{1}{\text{Im}\hat{c}}\left( \|R_1(\phi_{k-1},\varrho_{k-1})\|_{L^2}+\|R_2(\phi_{k-1},\varrho_{k-1})\|_{L^2} \right)\\
&\lesssim \frac{1}{n\text{Im}\hat{c}}\left(\|(\pa_Y\phi_{k-1},\a\phi_{k-1})\|_{L^2}+\|(\pa_Y\varrho_{k-1},\a\varrho_{k-1})\|_{L^2}\right),\\
\|(\pa_Y^2-\a^2)\xi_k\|_{L^2}&\lesssim \frac{n^{\f12}}{(\text{Im}\hat{c})^{\f12}}\left( \|R_1(\phi_{k-1},\varrho_{k-1})\|_{L^2}+\|R_2(\phi_{k-1},\varrho_{k-1})\|_{L^2} \right)\\
&\lesssim \frac{1}{(n\text{Im}\hat{c})^{\f12}}\left(\|(\pa_Y\phi_{k-1},\a\phi_{k-1})\|_{L^2}+\|(\pa_Y\varrho_{k-1},\a\varrho_{k-1})\|_{L^2}\right),
\end{aligned}
\end{equation}
and
\begin{equation}\label{6-3.4}
\begin{aligned}
&\|(\pa_Y^2-\a^2)\vartheta_k\|_{L^2}+\a^{\f12}\|(\pa_Y\vartheta_k,\a\vartheta_k)\|_{L^2}+\a\|\vartheta_k\|_{L^2}\\
&\qquad\quad\lesssim \frac{1}{\text{Im}\hat{c}}\left( \|R_1(\phi_{k-1},\varrho_{k-1})\|_{L^2}+\|R_2(\phi_{k-1},\varrho_{k-1})\|_{L^2} \right)\\
&\qquad\quad\lesssim \frac{1}{n\text{Im}\hat{c}}\left(\|(\pa_Y\phi_{k-1},\a\phi_{k-1})\|_{L^2}+\|(\pa_Y\varrho_{k-1},\a\varrho_{k-1})\|_{L^2}\right).
\end{aligned}
\end{equation}
Here we have used \eqref{6-2.11}. To proceed further, note that $(\phi_k,\varrho_k)$ is the solution to \eqref{6-3.2} and  $$
\begin{aligned}
\|\pa_YU_s\pa_Y\xi_k\|_{L^2_w}+\|\xi_k\pa_Y^2U_s\|_{L^2_w}&\leq \left\|\frac{\pa_YU_s}{|\pa_Y^2U_s|^{\f12}}\right\|_{L^\infty}\|\pa_Y\xi_k\|_{L^2}+\|Y|\pa_Y^2 U_s|^{\f12}\|_{L^\infty}\left\|\frac{\xi_k}{Y}\right\|_{L^2}\\
&\lesssim\|\pa_Y\xi_k\|_{L^2}.
\end{aligned}
$$
Then by applying \eqref{6-1.2}, \eqref{6-1.3} and \eqref{6-1.4} in Lemma \ref{lem6.1} to $(\phi_k,\varrho_k)$ with $q_1=\pa_YU_s\pa_Y\xi_k+\xi_k\pa_Y^2U_s$ and $q_2=0$, we obtain 
\begin{equation}\label{6-3.5}
\begin{aligned}
\|(\pa_Y^2-\a^2)\phi_k\|_{L^2_w}+\|(\pa_Y\phi_k,\a\phi_k)\|_{L^2}&\lesssim \frac{1}{\text{Im}\hat{c}}\|\pa_Y\xi_k\|_{L^2},\\
\|\pa_Y(\pa_Y^2-\a^2)\phi_k\|_{L^2_w}&\lesssim \frac{n^{\f12}}{(\text{Im}\hat{c})^{\f12}}\|\pa_Y\xi_k\|_{L^2},\\
\|(\pa_Y^2-\a^2)\varrho_k\|_{L^2}+\a^{\f12}\|(\pa_Y\varrho_k,\a\varrho_k)\|_{L^2}+\a\|\varrho_k\|_{L^2}&\lesssim \frac{1}{\text{Im}\hat{c}}\|\pa_Y\xi_k\|_{L^2}.
\end{aligned}
\end{equation}
We now define
\begin{equation}\CE_k=\|(\pa_Y^2-\a^2)\phi_k\|_{L^2_w}+\|(\pa_Y\phi_k,\a\phi_k)\|_{L^2}+\|(\pa_Y^2-\a^2)\varrho_k\|_{L^2}+\a^{\f12}\|(\pa_Y\varrho_k,\a\varrho_k)\|_{L^2}+\a\|\varrho_k\|_{L^2}.\nonumber
\end{equation}
From \eqref{6-3.3} and \eqref{6-3.5}, we have 
\begin{align}
\CE_k\leq \frac{C}{\a^{\f12}n(\text{Im}\hat{c})^2}\CE_{k-1}.\label{small1}
\end{align}
Then by taking $\vep_3$  smaller  if necessary so that
\begin{align}\label{small}
\frac{C}{\a^{\f12}n(\text{Im}\hat{c})^2}\leq \frac{C\sqrt{\vep}}{\a(\text{Im}\hat{c})^2}\leq C\a^{\f12}<\f12,
\end{align}
we conclude that
\begin{align}\label{6-3.6}
\sum_{k=1}^\infty\CE_k\leq C\CE_0\lesssim \frac{1}{\text{Im}\hat{c}}\|f_1\|_{L^2_w}+\|f_2\|_{L^2},
\end{align}
which implies
the absolutely convergence of $\sum_{k=1}^\infty\phi_k$ and $\sum_{k=1}^\infty \varrho_k$ in $H^2.$ Moreover, define
$$\CF_k=\|(\pa_Y^2-\a^2)\xi_k\|_{L^2}+\|(\pa_Y\xi_k,\a\xi_k)\|_{L^2}+\|(\pa_Y^2-\a^2)\vartheta_k\|_{L^2}+\a^{\f12}\|(\pa_Y\vartheta_k,\a\vartheta_k)\|_{L^2}+\a\|\vartheta_k\|_{L^2}.
$$
From \eqref{6-3.3} and \eqref{6-3.4} we have
\begin{align}\label{6-3.7}
\sum_{k=1}^\infty \CF_k
\lesssim \left(\frac{1}{\a^{\f12}n\text{Im}\hat{c}}+\frac{1}{(\a n\text{Im}\hat{c})^{\f12}}\right)\times\left(\sum_{k=0}^\infty \CE_k\right)\lesssim \frac{1}{\text{Im}\hat{c}}\|f_1\|_{L^2_w}+\|f_2\|_{L^2},
\end{align}
where we have used \eqref{6-3.6}. The absolute convergence of $\sum_{k=1}^\infty\xi_k$ and $\sum_{k=1}^\infty \vartheta_k$ follows as well. This gives the existence of solution. By recalling \eqref{solu}, the estimate \eqref{6.4} follows from \eqref{6-3.6},  \eqref{6-3.7} and 
the Sobolev embedding. Finally, from Lemma \ref{lem6.1} and Lemma \ref{lem6.2}, we know that for each $k\geq 0$, $(\phi_k,\varrho_k)$ and $(\xi_k,\vartheta_k)$ are analytic in $c$ with values in $H^2(\mathbb{R}_+)$. Then analyticity of $(\Phi,\Psi)$ follows from the absolute convergence, which together with the Sobolev embedding further implies the analyticity of $\Gamma_R(c)$. The proof of the first statement in Proposition \ref{prop5.1} is then completed. 

As for the second statement, since  $(\xi_0,\vartheta_0)$ solves \eqref{6-3.13}, we apply \eqref{6-2.7}, \eqref{6-2.8} and \eqref{6-2.9} to $(\xi_0,\vartheta_0)$ to obtain 
\begin{equation}\nonumber
\begin{aligned}
\|(\pa_Y\xi_0,\a\xi_0)\|_{L^2}&\lesssim \frac{1}{\text{Im}\hat{c}}\|g_1\|_{L^2}+\frac{1}{\a^{\f12}n\text{Im}\hat{c}}\|f_2\|_{L^2},\\
\|(\pa_Y^2-\a^2)\xi_0\|_{L^2}&\lesssim \frac{n^{\f12}}{(\text{Im}\hat{c})^{\f12}}\|g_1\|_{L^2}+\frac{1}{(\a n\text{Im}\hat{c})^{\f12}}\|f_2\|_{L^2},
\end{aligned}
\end{equation}
and
\begin{equation}\nonumber
\|(\pa_Y^2-\a^2)\vartheta_0\|_{L^2}+\a^{\f12}\|(\pa_Y\vartheta_0,\a\vartheta_0)\|_{L^2}+\a\|\vartheta_0\|_{L^2}
\lesssim \frac{1}{\text{Im}\hat{c}}\|g_1\|_{L^2}+\|f_2\|_{L^2}.
\end{equation}
By using these bounds and the first statement of Proposition \ref{prop5.1}, we obtain the following estimates on $(\tilde{\Phi},\tilde{\Psi})$ that solves \eqref{6-3.14-1}:
\begin{equation}
\begin{aligned}
\|(\pa_Y\tilde{\Phi},~\a\tilde{\Phi})\|_{L^2}+\|(\pa_Y^2-\a^2)\tilde{\Phi}\|_{L^2}&\lesssim_A \frac{1}{\text{Im}\hat{c}}\left(\|\pa_YU_s\pa_Y\xi_0\|_{L^2_w}+\|\xi_0\pa_Y^2U_s\|_{L^2_w}\right)\lesssim_A \frac{1}{\text{Im}\hat{c}}\|\pa_Y\xi_0\|_{L^2}\\
&\lesssim_A \frac{1}{(\text{Im}\hat{c})^2}\|g_1\|_{L^2}+\frac{1}{\a^{\f12}n(\text{Im}\hat{c})^2}\|f_2\|_{L^2},
\end{aligned}\nonumber
\end{equation}
and
\begin{equation}
\begin{aligned}
\|(\pa_Y^2-\a^2)\tilde{\Psi}\|_{L^2}+\a^{\f12}\|(\pa_Y\tilde{\Psi},\a\tilde{\Psi})\|_{L^2}+\a\|\tilde{\Psi}\|_{L^2}
&\lesssim_A \frac{1}{\text{Im}\hat{c}}\left(\|\pa_YU_s\pa_Y\xi_0\|_{L^2_w}+\|\xi_0\pa_Y^2U_s\|_{L^2_w}\right)\\
&\lesssim_A \frac{1}{(\text{Im}\hat{c})^2}\|g_1\|_{L^2}+\frac{1}{\a^{\f12}n(\text{Im}\hat{c})^2}\|f_2\|_{L^2}.
\end{aligned}\nonumber
\end{equation}
Recall \eqref{6-3.14-2}. By combining  these bounds and using the Sobolev embedding, we further obtain
\begin{equation}\label{6-3.8}
\begin{aligned}
&\|(\pa_Y\Phi,~\a\Phi)\|_{L^2}+\|(\pa_Y^2-\a^2)\Phi\|_{L^2}+\left|\pa_Y\Phi(0)\right|+\a^{\f12}\|(\pa_Y\Psi,~\a\Psi)\|_{L^2}+\|(\pa_Y^2-
\a^2)\Psi\|_{L^2}\\
&\qquad\lesssim_A \frac{1}{(\text{Im}\hat{c})^2}\left( 1 +n^{\f12}(\text{Im}\hat{c})^{\f32} \right)\|g_1\|_{L^2}+\left(1+\frac{1}{\a^{\f12}n(\text{Im}\hat{c})^2}+\frac{1}{(\a n\text{Im}\hat{c})^{\f12}}\right)\|f_2\|_{L^2}.
\end{aligned}
\end{equation}
For $\a=A\vep^{\f18}$ and $c$  in $D_*$, it holds 
$$n^{\f12}(\text{Im}\hat{c})^{\f32}\lesssim_A 1,~~\frac{1}{\a^{\f12}n(\text{Im}\hat{c})^2}\lesssim_A\vep^{\f1{16}},~\frac{1}{(\a n\text{Im}\hat{c})^{\f12}}\lesssim_A \vep^{\f1{16}}.
$$
By this and \eqref{6-3.8}, we  obtain \eqref{6.5}. The analyticity follows from the first statement and Lemma \ref{lem6.2}. The proof of Proposition \ref{prop5.1} is then completed. \qed\\

Now we are ready to prove the main theorem when $\a =A\vep^{\f18}$.

{\bf Proof of Theorem 1.1:} 
Recall the approximate solution $(\Phi_{\text{app}},\Psi_{\text{app}})(Y;c)$ defined in \eqref{4.3-1.3}  which satisfies \eqref{5.5-2}. We look for the solution to \eqref{eq1.1} with \eqref{BD1} in the form of
$$(\Phi,\Psi)(Y;c)=(\Phi_{\text{app}},\Psi_{\text{app}})(Y;c)-(\Phi_{R,1},\Psi_{R,1})(Y;c)-(\Phi_{R,2},\Psi_{R,2})(Y;c).
$$
Here $(\Phi_{R,1},\Psi_{R,1})$ and $(\Phi_{R,2},\Psi_{R,2})$ are the solutions to the Orr-Sommerfeld system \eqref{6.1}, \eqref{6.2} with inhomogeneous source terms $(f_1,f_2)=(E_3^s+E_3^f,F^f)$ and $(f_1,f_2)=(\pa_YE_1^s+i\a E_2^s+\pa_YE_1^f+i\a E_2^f,0)$ respectively. These two solutions have been constructed in Proposition \ref{prop5.1}. From the construction we know that
$$\text{OS}_c(\Phi,\Psi)=0,~\text{ and }\Phi(Y;c)|_{Y=0}=\Psi(Y;c)|_{Y=0}=0.$$ Now we define a mapping
$$\Gamma(c)\triangleq \pa_Y\Phi(0;c):~D_*\mapsto \mathbb{C},
$$
which is analytic by Proposition \ref{prop5.1}. Recall $\Gamma_0(c)=\pa_Y\Phi_{\text{app}}(0;c)$ defined in \eqref{4.3-1.1}. On one hand, from Proposition \ref{prop4.1},  $\Gamma_0(c)$ has a unique zero $c_{\text{app}}$ in $D_*$. Moreover, on the circle $\pa D_*$ it holds that $|\Gamma_0(c)|\geq \f{1}{2}A^{-\theta}>0.$ On the other hand, by using \eqref{5.8}, \eqref{5.9} in Proposition \ref{prop4.2}, \eqref{3.4-2}, \eqref{3.4-3} in Proposition \ref{prop4.3}, \eqref{6.4} and \eqref{6.5} in Proposition \ref{prop5.1}, we can estimate the difference $|\Gamma(c)-\Gamma_0(c)|$ as follows:
$$
\begin{aligned}
\left|\Gamma(c)-\Gamma_0(c)\right|\leq& \left|\pa_Y\Phi_{R,1}(0;c)\right|+\left|\pa_Y\Phi_{R,2}(0;c)\right|\\
\lesssim_{A}&~\frac{1}{\text{Im}\hat{c}}\left(\||U_s''|^{-\f12}E_3^s\|_{L^2}+\||U_s''|^{-\f12}E_3^f\|_{L^2}\right)+\|F^f\|_{L^2}\\
&~+\frac{1}{(\text{Im}\hat{c})^2}\left(\|E_1^s\|_{L^2}+\|E_1^f\|_{L^2}+\|E_2^s\|_{L^2}+\|E_2^f\|_{L^2}\right)\\
\lesssim_A&~ \frac{\vep^{\f3{16}}}{\text{Im}\hat{c}}+\vep^{\f{5}{16}}+\frac{\vep^{\f5{16}}}{(\text{Im}\hat{c})^2}\lesssim_A \vep^{\f1{16}}.
\end{aligned}
$$
By taking $\vep_0$ suitably small, we have, for any $\vep\in (0,\vep_0),$ that $\left|\Gamma(c)-\Gamma_0(c)\right|<\f12|\Gamma_0(c)|$. Therefore, we  conclude the proof of Theorem \ref{thm1.1} by Rouch\'e's Theorem. \qed\\

\section{The case $\a=M\vep^\beta$ with $\beta\in (3/28,1/8)$}
In this section, we will prove Theorem \ref{thm1.1} for $\a$ in the regime $\a\sim \vep^{\beta}$ with $\beta\in (3/28,1/8).$ 
 Let $\a=M\vep^{\beta}$. First we restrict the consideration to the case $\beta\in (1/12,1/8)$. For brevity, we assume without loss of generality that $M=1$ through out the section. Let 
$\nu_0=\frac{1}{4\beta}(1-8\beta)\in (0,1)$ for $\beta\in (1/12,1/8)$, and
\begin{equation}\label{7.1}
c_{**}=\a+\a^{1+\nu_0}e^{\f{1}{4}\pi i}.
\end{equation}
As shown later, the unstable eigenvalue lies in the disk 
\begin{equation}\label{7.3}
D_{**}=\{c\in \mathbb{C}\mid |c-c_{**}|\leq r_3\a^{1+\nu_0}  \}
\end{equation}
for some $r_3\in (0,\frac{\sqrt{2}}{2})$. In this section, we use $c$ rather than $\hat{c}$ because of the scaling.

Note that the following facts will be frequently used in the following:
\begin{equation}\label{7.4}
n=\frac{\a}{\sqrt{\vep}}=\a^{-3-2\nu_0};
\end{equation}
and there exists $\tau_4\in (0,1)$, such that 
\begin{equation}\label{7.5}
\a(1-\tau_4^{-1}\a^{\nu_0})\leq |c|\leq \a(1+\tau_4^{-1}\a^{\nu_0}),~\text{and }0<\text{arg}c<\tau_4^{-1}\a^{\nu_0},~\text{Im}c>\tau_4\a^{1+\nu_0}~
\text{for }c\in D_{**}.
\end{equation}
Now we define the approximate growing mode. The slow mode of stream function for velocity field is the same as \eqref{4.1.5}, that is
\begin{align}
\Phi_{\text{app}}^s(Y;c)=\psi_{\a,1}(Y;c)+\a\Phi_1^s(Y;c).\label{7.18}
\end{align}
The construction of fast mode is different from the previous case given in subsection 3.2. Instead, we will construct the fast mode around an exponential profile. For this, we rewrite the part of leading order in the first equation in the Orr-Sommerfeld system \eqref{eq1.1} as follows:
\begin{equation}
\begin{aligned}\nonumber
\frac{i}{n}\pa_Y^4\Phi+(U_s-c)\pa_Y^2\Phi-\pa_Y^2U_s\Phi=\pa_Y^2\left(\f{i}{n}\pa_Y^2\Phi+(U_s-c)\Phi-2\pa_Y^{-1}\left( \pa_YU_s\Phi\right)  \right),
\end{aligned}
\end{equation}
where we have denoted $\pa_Y^{-1}f=-\int_Y^\infty f(X)\dd X.$ Then for some positive integer $N\geq 1$ to be determined later, we define the fast mode as $$\Phi_{\text{app}}^f(Y;c)=\sum_{k=0}^N\Phi_{k}^f(Y;c),$$ where $\Phi_0^f$ is the solution to the following equation
\begin{align}\label{7.6}
\frac{i}{n}\pa_Y^2\Phi_0^f-c\Phi_0^f=0,~\Phi_0^f(Y;c)|_{Y=0}=1,
\end{align}
and $\Phi_k^f$ $(k=1,2,\cdots,N)$ solves the following hierarchy of equations 
\begin{align}\label{7.7}
\frac{i}{n}\pa_Y^2\Phi_k^f-c\Phi_k^f=-U_s\Phi_{k-1}^f+2\pa_Y^{-1}\left( \pa_YU_s\Phi_{k-1}^f\right),~ \Phi_{k}^f(Y;c)|_{Y=0}=0. 
\end{align}
From \eqref{7.4} and \eqref{7.5}, we know that
$$(1-\tau_4^{-1}\a^{\nu_0})\a^{-2(1+\nu_0)}\leq |nc|\leq (1+\tau_4^{-1}\a^{\nu_0})\a^{-2(1+\nu_0)},
$$
and 
$$
-\f{\pi}{2}<\arg(-inc)<-\f{\pi}{2}+\tau_4^{-1}\a^{\nu_0}.
$$
Then $\varpi=(-inc)^{\f12}$ is well-defined, analytic in $c$ and satisfies
\begin{align}\label{7.8}
|\varpi|=\a^{-1-\nu_0}(1+O(1)\a^{\nu_0}),~-\f{\pi}{4}<\arg\varpi<-\f{\pi}{4}+\f{\tau_4^{-1}}{2}\a^{\nu_0},~\frac{\sqrt{2}}{3}\a^{-1-\nu_0}\leq \text{Re}\varpi\leq\frac{2\sqrt{2}}{3}\a^{-1-\nu_0}.
\end{align}
Thus we can solve the leading order profile 
\begin{align}
\Phi_0^f(Y;c)=e^{-\varpi Y}\nonumber
\end{align}
from \eqref{7.6} and solve $\Phi_{k}^f$ from \eqref{7.7} inductively by
\begin{equation}\label{7.12}
\Phi_k^f(Y;c)=in\int_0^Ye^{-\varpi (Y-Y')}\dd Y'\int_{Y'}^\infty e^{\varpi(Y'-Y'')}\left(-U_s\Phi_{k-1}^f+2\pa_Y^{-1}( \pa_YU_s\Phi_{k-1}^f)\right)\dd Y'',~k=1,2,\cdots,N.
\end{equation}

The approximate fast mode of magnetic field is defined accordingly as
\begin{align}\label{7.9}
\Psi_{\text{app}}^f(Y;c)=\sum_{k=0}^N\Psi_{k}^f(Y;c),~\text{with }\Psi_{k}^f(Y;c)=-\pa_Y^{-1}\Phi_{k}^f(Y;c).
\end{align}
It is straightforward to check  that the fast mode $(\Phi_{\text{app}}^f,\Psi_{\text{app}}^f)$ satisfies 
\begin{equation}\nonumber
\left\{
\begin{aligned}
&\f{i}{n}\pa_Y^4\Phi_{\text{app}}^f+(U_s-c)\pa_Y^2\Phi_{\text{app}}^f-\pa_Y^2U_s\Phi_{\text{app}}^f=\pa_Y^2\left( U_s\Phi_N^f-2\pa_Y^{-1}(\pa_YU_s\Phi_N^f)\right),~Y>0,\\
&-\pa_Y^2\Psi_{\text{app}}^f-\pa_Y\Phi_{\text{app}}^f=0,~\Phi_{\text{app}}^f(Y;c)|_{Y=0}=1.
\end{aligned}
\right.
\end{equation}

In summary, we define the approximate growing mode as
\begin{equation}
\begin{aligned}\label{7.10}
\Phi_{\text{app}}(Y;c)&=\Phi_{\text{app}}^s(Y;c)-\Phi_{\text{app}}^s(0;c)\Phi_{\text{app}}^f(Y;c),\\
\Psi_{\text{app}}(Y;c)&=\Psi_{\text{app}}^s(Y;c)-\Phi_{\text{app}}^s(0;c)\Psi_{\text{app}}^f(Y;c),
\end{aligned}
\end{equation}
where $\Psi_{\text{app}}^s$ is the solution to \eqref{3.1} with $\varphi_b=\Phi_{\text{app}}^s(0;c)\Psi_{\text{app}}^f(0;c)$ and $f=i\a H_s\Phi_{\text{app}}^s+\pa_Y\Phi_{\text{app}}^s.$\\

In the following lemma, we will give some pointwise estimates on $\Phi_k^f$ and $\Psi_k^f.$
\begin{lemma}
	The following bounds hold 
	\begin{align}
	|\pa_Y^j\Phi_k^f(Y;c)|&\lesssim \a^{k\nu_0-j(1+\nu_0)}e^{-\frac{1}{4}\a^{-(1+\nu_0)}Y},\label{7.13}\\
	|\pa_Y^j\Psi_k^f(Y;c)|&\lesssim \a^{k\nu_0-(j-1)(1+\nu_0)}e^{-\frac{1}{4}\a^{-(1+\nu_0)}Y},\label{7.14}
	\end{align}
	where $Y\geq 0$, $k=0,1,2,\cdots, N$ and $j=0,1,2$.
\end{lemma}
\begin{proof}
	Define
	$$
	\CG_k\triangleq \sup_{Y\geq 0}\left( 1+\a^{-(1+\nu_0)}Y\right)e^{\f14\a^{-(1+\nu_0)}Y }|\Phi_k^f(Y)|.
	$$
	Recall that $\Phi_k^f$ satisfies \eqref{7.12}. Then
	$$
	\begin{aligned}
	\left| U_s\Phi_{k-1}^f(Y)\right|+2\left|\pa_Y^{-1}(\pa_YU_s\Phi_{k-1}^f)(Y)\right|&\lesssim \left\|\f{U_s}{Y}\right\|_{L^\infty}|Y\Phi_{k-1}^f(Y)|+\int_Y^\infty|\Phi_{k-1}^f|(X)\dd X\\
	&\lesssim \a^{1+\nu_0}e^{-\f14\a^{-(1+\nu_0)}Y}\CG_{k-1}.
	\end{aligned}
	$$
	By plugging it  into \eqref{7.12}, we obtain that
	\begin{equation}\label{7.15}
	\begin{aligned}
	|\Phi_k^f(Y)|&\lesssim n\a^{1+\nu_0}\CG_{k-1}\int_0^Ye^{-\text{Re}\varpi(Y-Y')}\dd Y'\int_{Y'}^\infty e^{\text{Re}\varpi(Y'-Y'')}e^{-\frac{1}{4}\a^{-(1+\nu_0)}Y''}\dd Y''.\\
	&\lesssim \frac{n\a^{1+\nu_0}}{(\text{Re}\varpi)^2}\CG_{k-1}e^{-\frac{1}{4}\a^{-(1+\nu_0)}Y}\lesssim \a^{\nu_0}\CG_{k-1}e^{-\f14\a^{-(1+\nu_0)}Y}.
	\end{aligned}
	\end{equation}
	Here we have used \eqref{7.4} and \eqref{7.8}. Applying an inductive argument to \eqref{7.15}, we obtain for $k=1,2,\cdots,N$ that
	$$\sup_{Y\geq 0}e^{\f14\a^{-1-\nu_0}Y}|\Phi_k^f(Y)|\lesssim \a^{k\nu_0}\sup_{Y\geq 0}\left(1+\left( \a^{-(1+\nu_0)}Y\right)^k\right)e^{\f14\a^{-(1+\nu_0)}Y}|\Phi_0^f(Y)|\lesssim \a^{k\nu_0}.
	$$
	This gives \eqref{7.13} when $j=0$. All other inequalities in \eqref{7.13} and \eqref{7.14} can be obtained similarly. We omit the detail for brevity.
\end{proof}
We now define $\tilde{\Gamma}_0(c)=\pa_Y\Phi_{\text{app}}(0;c)$.
With the above pointwise estimates, we can show the existence of $c$ for  $\tilde{\Gamma}_0(c)=0$ in the following lemma.
\begin{lemma}\label{lem7.2}
	Let $\beta\in (1/12,1/8)$ and $r_3\in (0,\frac{\sqrt{2}}{2})$. There exists $\vep_5\in (0,1)$, such that for any $\vep\in (0,\vep_5)$ and $\a=\vep^{\beta}$, the function $\tilde{\Gamma}_0(c)$ has a unique zero point inside $D_{**}$. Moreover, on the circle $\pa D_{**}$, it holds that
	\begin{align}\label{7.16}
	|\tilde{\Gamma}_0(c)|\geq \frac{r_3}{2}.
	\end{align}
\end{lemma}
\begin{remark}
	By  \eqref{7.5},  for $\a=\vep^\beta$ and $c\in D_{**}$, it holds that $$\a\text{Im}c\geq \tau_4\a^{\f1{4\beta}}=\tau_4\vep^{\f14}.$$
\end{remark}
\begin{proof}
	The proof is similar to  the one for  Proposition \ref{prop4.1}. In fact, note that
	\begin{equation}\label{7.17}
	\begin{aligned}
	\tilde{\Gamma}_0(c)&=\pa_Y\Phi_{\text{app}}^s(0;c)-\Phi_{\text{app}}^s(0;c)\pa_Y\Phi_{\text{app}}^f(0;c)\\
	&=\pa_Y\Phi_{\text{app}}^s(0;c)-\Phi_{\text{app}}^s(0;c)\left(-\varpi+\sum_{k=1}^N\pa_Y\Phi_{k}^f(0;c)\right)\\
	&=\pa_Y\Phi_{\text{app}}^s(0;c)+\a^{-(1+\nu_0)}e^{-\frac{1}{4}\pi i}\Phi_{\text{app}}^s(0;c)+\Phi_{\text{app}}^s(0;c)\left( \varpi-\a^{-(1+\nu_0)}e^{-\f14\pi i}  \right)\\
	&\qquad-\Phi_{\text{app}}^s(0;c)\left( \sum_{k=1}^N\pa_Y\Phi_k^f(0;c) \right).
	\end{aligned}
	\end{equation}
	In view of asymptotic behavior \eqref{4.1.7}, \eqref{4.1.7-1} and the pointwise bounds \eqref{7.13}, we define 
	\begin{equation}
	\tilde{\Gamma}_{\text{ref}}(c)=1+\a^{-(1+\nu_0)}e^{-\f14\pi i}(-c+\a).\nonumber
	\end{equation}
	One can see that $\tilde{\Gamma}_{\text{ref}}$ has $c_{**}$ given in \eqref{7.1} as its unique zero point in the whole complex plane. On the circle $\pa D_{**}$, it holds that
	\begin{align}
	|\tilde{\Gamma}_{\text{ref}}(c)|=r_3>0.\nonumber
	\end{align}
	To further estimate the difference $|\tilde{\Gamma}_0(c)-\tilde{\Gamma}_{\text{ref}}(c)|$, we observe from \eqref{4.1.7}, \eqref{7.4} and \eqref{7.5} that
	\begin{equation}\label{7.18-1}
	\left|\Phi_{\text{app}}^s(0;c)\right|\lesssim |c-\a|+\f1n+\a|\hat{c}\log\text{Im}\hat{c}|\lesssim \a^{1+\nu_0}(1+\a^{1-\nu_0}|\log\a|)\lesssim \a^{1+\nu_0},
	\end{equation}
	where we have used $\nu_0\in (0,1)$ in the last inequality. Thus by using \eqref{7.8}, \eqref{7.13} and \eqref{7.18-1} we get
	$$\left|\Phi_{\text{app}}^s(0;c) \left( \varpi-\a^{-(1+\nu_0)}e^{-\f14\pi i}\right)   \right|+\left| \Phi_{\text{app}}^s(0;c)\left( \sum_{k=1}^N\pa_Y\Phi_k^f(0;c) \right)  \right|\lesssim  \a^{\nu_0}.
	$$
	By plugging this bound into \eqref{7.17}, we then obtain that
	$$\begin{aligned}|\tilde{\Gamma}_0(c)-\tilde{\Gamma}_{\text{ref}}(c)|&\lesssim \a|\log \text{Im}\hat{c}|+\a^{-(1+\nu_0)}\left( \f1n+\a|\hat{c}\log\text{Im}\hat{c}| \right)+\a^{\nu_0}\\
	&\lesssim \a^{\nu_0}+\a^{1-\nu_0}|\log \a|,
	\end{aligned}
	$$
	where we have used $\nu_0\in (0,1)$ again. Taking $\vep_5\in (0,1)$ suitably small, such that for any $\vep\in (0,\vep_5)$,
	$$\a^{\nu_0}+\a^{1-\nu_0}|\log\a|<\frac{r_3}{2},
	$$
	we obtain that
	$$|\tilde{\Gamma}_0(c)-\tilde{\Gamma}_{\text{ref}}(c)|<\f12|\tilde{\Gamma}_{\text{ref}}(c)|,~\text{for any }c\in \pa D_{**}.
	$$
	We then complete  the proof by  Rouch\'e's theorem.
\end{proof}

By \eqref{7.18}, \eqref{7.9} and \eqref{7.10}, it is straightforward to check that
\begin{equation}\label{7.20}
\left\{
\begin{aligned}
&\text{OS}_c(\Phi_{\text{app}},\Psi_{\text{app}})=\left( E_{\beta}^s(Y;c)+E^f_{\beta}(Y;c),~F_{\beta}^f(Y;c)  \right),~Y>0,\\
&\Phi_{\text{app}}(Y;c)|_{Y=0}=\Psi_{\text{app}}(Y;c)|_{Y=0}=0.
\end{aligned}
\right.
\end{equation} 
Here $E_\beta^s(Y;c)=\pa_YE_{\beta,1}^s(Y;c)+i\a E_{\beta,2}^s(Y;c)+E_{\beta,3}^s(Y;c)$ is the error term generated by the slow mode $\Phi_{\text{app}}^s$ and it is exactly the same as \eqref{5.5}, that is, $E_{\beta,j}^s(Y;c)=E_{j}^s(Y;c),~j=1,2,3.$ The error term $E_\beta^f(Y;c)$ generated by the fast mode is slightly different from \eqref{5.5-1}. In fact,
\begin{equation}
E_{\beta}^f(Y;c)=\pa_YE_{\beta,1}^f(Y;c)+i\a E_{\beta,2}^f(Y;c)+E_{\beta,3}^f(Y;c),\nonumber
\end{equation}
where
\begin{align}
E^f_{\beta,1}(Y;{c})\triangleq&-\Phi_{\text{app}}^s(0;{c})\bigg(-\f{i}{n}(2\a^2+1)\pa_Y\Phi_{\text{app}}^f+U_s\pa_Y\Phi_N^f-\pa_YU_s\Phi_N^f\nonumber\\
&\qquad-\sqrt{\vep}H_s\pa_Y\Psi_{\text{app}}^f-\f{\a}{n}(U_s-{c})\Psi_{\text{app}}^f+\f{\a}{n}H_s\Phi_{\text{app}}^f\bigg),\nonumber\\
E^f_{\beta,2}(Y;{c})\triangleq&-\Phi_{\text{app}}^s(0;{c})\left( \frac{\a}{n}(\a^2-1)\Phi_{\text{app}}^f+i\a(U_s-\hat{c})\Phi_{\text{app}}^f -i\a\sqrt{\vep}H_s\Psi_{\text{app}}^f  \right),\nonumber\\
E^f_{\beta,3}(Y;{c})\triangleq&-\Phi_{\text{app}}^s(0;{c})\left( \sqrt{\vep}\pa_YH_s\pa_Y\Psi_{\text{app}}^f+\sqrt{\vep}\pa_Y^2H_s\Psi_{\text{app}}^f\right).\nonumber
\end{align}
And the error term $F_{\beta}^f(Y;c)$ in the magnetic equation remains the same as \eqref{5.5-3}, that is, $F_{\beta}^f(Y;c)=F^f(Y;c).$ For these error terms, we have the following bound estimates.
\begin{lemma}\label{lem7.3}
	Let $\beta\in (1/12,1/8)$. There exists $\vep_6\in (0,\vep_5)$, such that if $\a=\vep^{\beta}$ and $c\in D_{**}$ where $D_{**}$ is given in \eqref{7.3}, then we have
	\begin{align}
	&\|E_{\beta,1}^s(\cdot~;c)\|_{L^2}+\|E_{\beta,2}^s(\cdot~;c)\|_{L^2}\lesssim \a^{2+\f12(1+\nu_0)},\label{7.23}\\
	&\||U_s''|^{-\f12}E_{\beta,3}^s(\cdot~;c)\|_{L^2}\lesssim \a^{\f32},\label{7.24}\\
	&\|E_{\beta,1}^f(\cdot~;c)\|_{L^2}+\|E_{\beta,2}^f(\cdot~;c)\|_{L^2}+	\||U_s''|^{-\f12}E_{\beta,3}^f(\cdot~;c)\|_{L^2}\lesssim \a^{\f52(1+\nu_0)},\label{7.24-1}\\
	&\|F_\beta^f(\cdot~;c)\|_{L^2}\lesssim \a^{1+\f32(1+\nu_0)}.\label{7.24-2}
	\end{align}
\end{lemma}
\begin{proof}
	We first estimate the slow mode $\Psi_{\text{app}}^s$ of magnetic field, which solves \eqref{3.1} with $\varphi_b=\Phi_{\text{app}}^s(0;c)\Psi_{\text{app}}^f(0;c)$ and $f=i\a H_s\Phi_{\text{app}}^s+\pa_Y\Phi_{\text{app}}^s.$ Taking $\vep_6\in (0,\vep_5)$ suitably small, we have
	$$D_{**}\subsetneqq \{\text{Im}c>0,~|\hat{c}|\leq\g_2 \}\cap \{ |c|< \g_0 \}~\text{and }\a\in (0,\a_0),
	$$
	where $\g_0$ and $\a_0$ are the numbers given in Proposition \ref{prop3.1} and $\g_2$ is the number given in Lemma \ref{lem4.1}. Using \eqref{7.14} with $j=0$ and \eqref{7.18-1}, we have
	$$|\varphi_b|=\left| \Phi_{\text{app}}^s(0;c)\Psi_{\text{app}}^f(0;c) \right|\lesssim \a^{2(1+\nu_0)}.
	$$
	Then  as the proof of Lemma \ref{lem5.6}, we  obtain
	\begin{align}
	\a\|\Psi_{\text{app}}^s(\cdot~;{c})\|_{L^\infty_\a}+
	\a^{\f12}\|\pa_Y\Psi_{\text{app}}^s(\cdot~;{c})\|_{L^\infty_\a}+
	\|\pa_Y^2\Psi_{\text{app}}^s(\cdot~;{c})\|_{L^\infty_\a}\lesssim 1,\label{7.21}
	\end{align}
	and
	\begin{align}
	\a\|\Psi_{\text{app}}^s(\cdot~;{c})\|_{L^2}+
	\a^{\f12}\|\pa_Y\Psi_{\text{app}}^s(\cdot~;{c})\|_{L^2}+
	\|\pa_Y^2\Psi_{\text{app}}^s(\cdot~;{c})\|_{L^2}\lesssim \a^{-\f12}.\label{7.22}
	\end{align}
	Now we turn to estimate error terms $E_{\beta}^s,~E_{\beta}^f$ and $F_{\beta}^f.$ As for \eqref{5.9-2}, we use the
	 bound estimates  \eqref{5.9-1}, \eqref{5.7}, \eqref{7.21} and \eqref{7.22} to obtain
	$$
	\begin{aligned}
	\|E_1^s(\cdot~;{c})\|_{L^2}\lesssim&\frac{1}{n}\left( \|\pa_Y^3\psi_{\a,1}\|_{L^2}+\a\|\pa_Y^3\Phi_1^s\|_{L^2} +\a^2\|\pa_Y\psi_{\a,1}\|_{L^2}+\a^3\|\pa_Y\Phi_1^s\|_{L^2}\right)\\
	&+\vep^{\f12}\|\pa_Y\Psi_{\text{app}}^s\|_{L^2}+\frac{\a}{n}\left( \|\Psi_{\text{app}}^s\|_{L^2}+\|\psi_{\a,1}\|_{L^2}+\a\|\Phi_1^s\|_{L^2}  \right)\\
	\lesssim &\frac{1}{n}\left(1+\f{\a}{|\text{Im}\hat{c}|^{\f32}}+\a^2+\a^3  \right)+\f{\vep^{\f12}}{\a}+\f{\a}{n}\left(\f{1}{\a^{\f32}}+\f{1}{\a^{\f12}}+\a^{\f12}\right)\\
	\lesssim& \a^{3+2\nu_0}\left(1+\a^{1-\frac{3}{2}(1+\nu_0)} \right)+\a^{\f52+2\nu_0}\lesssim \a^{2+\f12(1+\nu_0)}.
	\end{aligned}
	$$
	Here we have used \eqref{7.4} and \eqref{7.5} in the last line. Similar as Proposition \ref{prop4.2}, one can get
	$$
	\begin{aligned}
	\|E_2^s(\cdot~;{c})\|_{L^2}
	\lesssim \frac{\a^{\f52}}{n}+\frac{\vep^{\f12}}{\a^{\f12}}+\frac{\a^{\f12}}{n}\lesssim  \a^{\frac{7}{2}+2\nu_0},
	\end{aligned}
	$$
	and
	$$
	\||U_s''|^{-\f12}E_{3}^s(\cdot~;{c})\|_{L^2}
	\lesssim \a^{\f32}+\frac{\vep^{\f12}}{\a}+\frac{\vep^{\f12}}{\a^{\f32}}\lesssim \a^{\f32}.
	$$
	This completes the proof of \eqref{7.23} and \eqref{7.24}. As for the fast mode, by applying pointwise estimates \eqref{7.13}, \eqref{7.14} and \eqref{7.18-1} to $E_{\beta,1}^f(Y;c)$, we obtain that
	\begin{equation}
	\begin{aligned}\label{7.25}
	\|E_{\beta,1}^f(\cdot~;c)\|_{L^2}\lesssim& \a^{1+\nu_0}\bigg(\f1n\|\pa_Y\Phi_{\text{app}}^f\|_{L^2}+\left\|\f{U_s}{Y}\right\|_{L^\infty}\|Y\pa_Y\Phi_{N}^f\|_{L^2}+\|\Phi_{N}^f\|_{L^2}
	\\
	&+\sqrt{\vep}\|\pa_Y\Psi_{\text{app}}^f\|_{L^2}+\f{\a}{n} \left\|\f{U_s}{Y}\right\|_{L^\infty}\|Y\Psi_{\text{app}}^f\|_{L^2} +\f{\a|c|}{n}\|\Psi_{\text{app}}^f\|_{L^2} +\f{\a}{n}\|\Phi_{\text{app}}^f\|_{L^2}\bigg)\\
	\lesssim&\a^{1+\nu_0}\left( \frac{1}{n\a^{\f{1+\nu_0}{2}}}+\a^{N\nu_0+\f{1}{2}(1+\nu_0)} +\sqrt{\vep}\a^{\f12(1+\nu_0)}+\f{\a^{1+\f12(1+\nu_0)}}{n}\right)\\
	\lesssim& \a^{\f{5}{2}(1+\nu_0)}\left( \a+\a^{N\nu_0-(1+\nu_0)}  \right),
	\end{aligned}
	\end{equation}
	where we have used \eqref{7.4} in the last inequality. Similarly, we obtain that
	\begin{equation}\label{7.26}
	\begin{aligned}
	\|E_{\beta,2}^f(\cdot~;c)\|_{L^2}&\lesssim \a^{1+\nu_0}\left( \f{\a}{n}\|\Phi_{\text{app}}^f\|_{L^2}+\a\left\|\f{U_s}{Y}\right\|_{L^\infty}\|Y\Phi_{\text{app}}^f\|_{L^2}+\a|\hat{c}|\|\Phi_{\text{app}}^f\|_{L^2}      +\a\sqrt{\vep}\|\Psi_{\text{app}}^f\|_{L^2} \right)\\
	&\lesssim \a^{1+\nu_0}\left( \f{\a^{1+\f12(1+\nu_0)}}{n} +\a^{2+\f12(1+\nu_0)}+\sqrt{\vep}\a^{1+\f{1}{2}(1+\nu_0)} \right)\lesssim \a^{\frac{5}{2}(1+\nu_0)}\a^{1-\nu_0}.
	\end{aligned}
	\end{equation}
	By using \eqref{BL}, we deduce that
	\begin{equation}\label{7.28}
	\begin{aligned}
	\||U_s''|^{-\f12}E_{\beta,3}^f(\cdot~;{c})\|_{L^2}&\lesssim \a^{1+\nu_0}\bigg( \vep^{\f12}\left\|\f{H_s'}{|U_s''|^{\f12}}\right\|_{L^\infty}\|\pa_Y\Psi_{\text{app}}^f\|_{L^2}+\vep^{\f12}\left\|\f{H_s''}{|U_s''|^{\f12}}\right\|_{L^\infty}\|\Psi_{\text{app}}^f\|_{L^2} \bigg)\\
	&\lesssim \a^{2+\f72(1+\nu_0)}.
	\end{aligned}
	\end{equation}
	By taking $N$ suitably large so that $N\nu_0>1+\nu_0$ and by noting that $\nu_0\in (0,1)$, we obtain \eqref{7.24-1} from \eqref{7.25}, \eqref{7.26} and \eqref{7.28}. In summary, we have
	\begin{equation}
	\begin{aligned}
	\|F_{\beta}^f(\cdot~;c)\|_{L^2}&\lesssim \a^{1+\nu_0}\left( \a^2\|\Psi_{\text{app}}^f\|_{L^2}+\a \left\|\f{U_s}{Y}\right\|_{L^\infty}\|Y\Psi_{\text{app}}^f\|_{L^2}+\a|c|\|\Psi_{\text{app}}^f\|_{L^2}+\a\|\Phi_{\text{app}}^f\|_{L^2}  \right)\\
	&\lesssim \a^{1+\f32(1+\nu_0)},\nonumber
	\end{aligned}
	\end{equation}
	which is \eqref{7.24-2}. We then complete the proof of Lemma \ref{lem7.3}.
\end{proof}
{\bf Proof of Theorem \ref{thm1.1} for $\beta\in (3/28,1/8)$.} We look for the solution $(\Phi,\Psi)(Y;c)$ to the Orr-Sommerfeld equation \eqref{eq1.1} with \eqref{BD1} in the following form
\begin{equation}
(\Phi,\Psi)(Y;c)=(\Phi_{\text{app}},\Psi_{\text{app}})(Y;c)-(\Phi_{\beta,1},\Psi_{\beta,1})(Y;c)-(\Phi_{\beta,2},\Psi_{\beta,2})(Y;c).\nonumber
\end{equation}
Here $(\Phi_{\text{app}},\Psi_{\text{app}})$ is the approximate solution constructed in \eqref{7.10} and it satisfies \eqref{7.20},  $(\Phi_{\beta,1},\Psi_{\beta,1})$ satisfies
\begin{equation}
\left\{
\begin{aligned}
&\mbox{OS}_c(\Phi_{\beta,1},\Psi_{\beta,1})=\left( E_{\beta,3}^s+E_{\beta,3}^f,~F_{\beta}^f    \right)\\\nonumber
&\Phi_{\beta,1}(Y;c)|_{Y=0}=(\pa_Y^2-\a^2)\Phi_{\beta,1}(Y;c)|_{Y=0}=\Psi_{\beta,1}(Y;c)|_{Y=0}=0,
\end{aligned}
\right.
\end{equation}
and
$(\Phi_{\beta,2},\Psi_{\beta,2})$ satisfies
\begin{equation}
\left\{
\begin{aligned}
&\mbox{OS}_c(\Phi_{\beta,2},\Psi_{\beta,2})=\left( \pa_YE_{\beta,2}^s+i\a E_{\beta,2}^s+\pa_YE_{\beta,2}^f+i\a E_{\beta,2}^f,~0   \right)\\\nonumber
&\Phi_{\beta,2}(Y;c)|_{Y=0}=(\pa_Y^2-\a^2)\Phi_{\beta,2}(Y;c)|_{Y=0}=\Psi_{\beta,2}(Y;c)|_{Y=0}=0.
\end{aligned}
\right.
\end{equation}
For $\a=\vep^\beta$ with $\beta\in (1/12,1/8)$ and $c\in D_{**}$, by using \eqref{7.4} and \eqref{7.5} 
we can
take $\vep_0$ suitably small so that for $\vep\in (0,\vep_0)$, the following  holds:
\begin{equation}\nonumber
\text{Im}\hat{c}\geq \tau_4\a^{1+\nu_0}\geq \tau_4n^{-1}\a^{-(2+\nu_0)}\geq 2\tau_2^{-1}n^{-1},~\text{Im}\hat{c}\geq \tau_4n^{-1}\a^{-\f12}\a^{-(\f32+\nu_0)}\geq 2\tau_3^{-1}n^{-1},
\end{equation}
and
$$\frac{1}{\a^{\f12}n(\text{Im}\hat{c})^2}\lesssim \a^{3+2\nu_0-\f12-2(1+\nu_0)}\lesssim \a^{\f12}\ll1.
$$
Here $\tau_2$ and $\tau_3$ are the numbers given in Lemma \ref{lem6.1} and Lemma \ref{lem6.2} respectively. Therefore, $D_{**}\subsetneqq\Sigma_{d}\cap\Sigma_{s}$, and in view of \eqref{small1} and Remark \ref{rmk4.2}, the solution sequence defined in \eqref{solu} converges. This gives the existence of  $(\Phi_{\beta,1},\Psi_{\beta,1})(Y;c)$ and $(\Phi_{\beta,2},\Psi_{\beta,2})(Y;c)$. Moreover, in view of \eqref{6-3.6}, \eqref{6-3.7} and \eqref{6-3.8}, since 
$$\frac{1}{\a^{\f12}n(\text{Im}\hat{c})^2}\lesssim \a^{\f12},~\frac{1}{(\a n\text{Im}\hat{c})^{\f12}}\lesssim \a^{\f12(1+\nu_0)},~ n^{\f12}(\text{Im}\hat{c})^{\f32}\lesssim \a^{\f12\nu_0},
$$
we can show that $(\Phi_{\beta,1},\Psi_{\beta,1})(Y;c)$ and $(\Phi_{\beta,2},\Psi_{\beta,2})(Y;c)$ satisfy the same bounds as \eqref{6.4} and \eqref{6.5} respectively. In
particular, it holds that
\begin{equation}\label{7.29}
\begin{aligned}
\left|\pa_Y\Phi_{\beta,1}(0;c)\right|&\lesssim \frac{1}{\text{Im}\hat{c}}\left( \||U_s''|^{-\f12}E_{\beta,3}^s\|_{L^2}+\||U_s''|^{-\f12}E_{\beta,3}^f\|_{L^2}  \right)+\|F^f_\beta\|_{L^2},\\
\left|\pa_Y\Phi_{\beta,2}(0;c)\right|&\lesssim \frac{1}{(\text{Im}\hat{c})^2}\left( \|(E_{\beta,1}^s,E_{\beta,2}^s)\|_{L^2}+\|(E_{\beta,1}^f,E_{\beta,2}^f)\|_{L^2}  \right).
\end{aligned}
\end{equation}

Recall $\tilde{\Gamma}_0(c)$ is the mapping given in Lemma \ref{lem7.2}. Now we define the mapping $$\tilde{\Gamma}(c)\triangleq\pa_Y\Phi(0;c):~ D_{**}\mapsto \mathbb{C}.$$
By using \eqref{7.23}, \eqref{7.24}, \eqref{7.24-1}, \eqref{7.24-2} and \eqref{7.29}, we obtain 
\begin{equation}
\begin{aligned}
\left| \tilde{\Gamma}(c)-\tilde{\Gamma}_0(c)\right|&\leq \left|\pa_Y\Phi_{\beta,1}(0;c)\right|+\left|\pa_Y\Phi_{\beta,2}(0;c)\right|\\
&\lesssim \frac{1}{\text{Im}\hat{c}}\left(\a^{\f32}+\a^{\f52(1+\nu_0)}\right)+\a^{1+\f32(1+\nu_0)}+\f{1}{(\text{Im}\hat{c})^2}\left( \a^{2+\f12(1+\nu_0)}+\a^{\f52(1+\nu_0)}  \right)\\
&\lesssim \a^{\f12-\nu_0}+\a^{\f12(1-3\nu_0)}.\label{7.30}
\end{aligned}
\end{equation}
For $\beta\in (3/28,1/8)$, we have $\nu_0\in (0,1/3)$. Then taking $\vep_0$  smaller if necessary, we deduce that 
\begin{equation}\label{7.31}
\left| \tilde{\Gamma}(c)-\tilde{\Gamma}_0(c)\right|\lesssim \a^{\f12-\nu_0}+\a^{\f12(1-3\nu_0)}\lesssim \f{r_3}{4}.
\end{equation}
The theorem then follows from Lemma \ref{lem7.2}, \eqref{7.16} and an application of Rouch\'e's Theorem. \qed.\\


\noindent{\bf Acknowledgment:} 
The research of C.-J. Liu was supported by the National Key R\&D Program of China (No.2020YFA0712000) and National Natural Science Foundation of China (11801364). The research of T. Yang  was supported by the General Research Fund of Hong Kong CityU No. 11303521.

\section{Appendices}
\subsection{Properties of Airy function}
Recall the classical Airy function, cf. \cite{GMM1},
\begin{align}\label{A1}
Ai(z)=\frac{1}{2\pi i}\int_L \exp\left( zt-\f{t^3}{3}\right)\dd t,
\end{align}
where the contour $L$ is given by
$$L=\left\{  re^{-\f{2\pi i}{3}}, r\in (1,+\infty)    \right\}\cup \left\{  e^{i\theta}, \theta\in[-\f{4\pi}{3},-\f{2\pi}{3}]   \right\}\cup \left\{  re^{\f{2\pi i}{3}}, r\in (1,+\infty)    \right\}.
$$
Its first, second and third order primitive functions are given by
\begin{align}\label{A2}
Ai(k,z)\triangleq \frac{1}{2\pi i} \int_L t^{-k}\exp\left( zt-\f{t^3}{3}\right)\dd t,~k=1,2,3.
\end{align}
It is straightforward to check that 
$$\pa_z^2Ai(z)-zAi(z)=0,
$$
and
$$\pa_zAi(k,z)=Ai(k-1,z),~ k=1,2,3, ~Ai(0,z)=Ai(z).$$ The following lemma gives the asymptotic behavior of $Ai(k,z)$ for large $|z|$, see \cite{DR,GGN,GMM1}. The proof of \eqref{EA1} for the case $k=0,1,2$ can be found in \cite{GMM1}, while the case $k=3$ can be treated similarly. We omit detail of the proof for brevity.
\begin{lemma}
There exists a positive constant $M>1$, such that for $|z|\geq M$ and $|\arg z|\leq \frac{5\pi}{6}$, we have the following asymptotic formula:
\begin{align}
Ai(k,z)=\frac{(-1)^k}{2\sqrt{\pi}}z^{-\f{1+2k}{4}}e^{-\frac{2}{3}z^{\f32}}(1+O(z^{-\f{3}{2}})),~k=0,1,2,3. \label{EA1}
\end{align}
\end{lemma}

\subsection{Some explicit formula}

\begin{lemma}\label{lem5.3}
		Let $\ka\geq1 $ be a positive integer. There exists $\tilde{\gamma}\in (0,1)$, such that for any $c\in$ $\{\text{Im}c>0,~|\hat{c}|\leq \tilde{\gamma} \}$, it holds that
	\begin{align}
	\int_0^1|\log(U_s(Y)-\hat{c})|^\ka\dd Y&\lesssim 1,~\ka\geq 1,\label{5.8-1}\\	
	\int_0^1|U_s(Y)-\hat{c}|^{-1}\dd Y&\lesssim 1+|\log\text{Im}\hat{c}|,\label{5.8-2}\\	
	\int_0^1|U_s(Y)-\hat{c}|^{-\ka}\dd Y&\lesssim 1+|\text{Im}\hat{c}|^{-\ka+1},~\ka>1.\label{5.8-3}
	\end{align}
\end{lemma}
\begin{proof}
We give the proof of  \eqref{5.8-1} only because \eqref{5.8-2} and \eqref{5.8-3} can be proved similarly. Take $\tilde{\gamma}=\min\{\f{1}{2}U_s(1),\frac{1}{2}(1-U_s(1))\}$. Then one has
$$|U_s(Y)-\hat{c}|\leq U_s(1)+|\hat{c}|<1 \text{ for } Y\in [0,1] \text{ and } |\hat{c}|<\tilde{\gamma}.$$
Note that
$$
\int_0^1|\log(U_s(Y)-\hat{c})|^\ka\dd Y\lesssim \int_0^1\left(-\log\left( |U_s(Y)-\text{Re}\hat{c}|+\text{Im}\hat{c}\right)\right)^\ka\dd Y+1.
$$
We now consider two cases.

Case 1: $\text{Re}\hat{c}\in \left(0,\frac{U_s(1)}{2}\right)$. For this case, let $Y_1\in [0,1]$ be the point such that $U_s(Y_1)=\text{Re}\hat{c}$. Then we have
$$
\begin{aligned}
&\int_0^1\left(-\log\left( |U_s(Y)-\text{Re}\hat{c}|+\text{Im}\hat{c}\right)\right)^\ka\dd Y\\
&\qquad=\int_0^{Y_1}\left(-\log\left(\text{Re}\hat{c}-U_s(Y)+\text{Im}\hat{c}\right)\right)^\ka\dd Y+\int_{Y_1}^{1}\left(-\log\left( U_s(Y)-\text{Re}\hat{c}+\text{Im}\hat{c}\right)\right)^\ka\dd Y.
\end{aligned}
$$
For the first integral, denote $Q(Y)\triangleq-\log\left( \text{Re}\hat{c}-U_s(Y)+\text{Im}\hat{c}\right)$. Since $U_s'(Y)>0$, we have
$$\begin{aligned}
\int_0^{Y_1}Q^\ka(Y)\dd Y=&-\int^{Y_1}_{0}\frac{\dd}{\dd Y}\left\{\left( \text{Re}\hat{c}-U_s(Y)+\text{Im}\hat{c}\right)Q^\ka(Y)\right\}\frac{1}{U_s'(Y)}\dd Y+\ka\int^{Y_1}_{0}Q^{\ka-1}(Y)\dd Y\\
\lesssim&1+\ka\int_0^{Y_1}Q^{\ka-1}(Y)\dd Y+\int^{Y_1}_{0}\left|\left( \text{Re}\hat{c}-U_s(Y)+\text{Im}\hat{c}\right)Q^\ka(Y)\right|\left|\frac{U_s''(Y)}{(U_s'(Y))^2}\right|\dd Y\\
\lesssim&1+\ka\int_0^{Y_1}Q^{\ka-1}(Y)\dd Y,
\end{aligned}
$$
where in the last inequality we have used  $x|\log x|^\ka\lesssim_\ka 1$ for $x\in [0,1]$. Inductively, we can obtain
$$\int_0^{Y_1}Q^\ka(Y)\dd Y\lesssim 1.
$$
Similarly, we also have
$$\int_{Y_1}^{1}\left|\log\left( \text{Re}\hat{c}-U_s(Y)+\text{Im}\hat{c}\right)\right|^\ka\dd Y\lesssim 1.$$ 

Case 2: $\text{Re}\hat{c}\in \left(-\f{U_s(1)}{2},0\right]$. In fact, this case can be treated similarly. We then conclude \eqref{5.8-1} and also the proof  the  lemma.
\end{proof}

Next we list some explicit formula related to $\Phi_{1}^s(Y;{c})$ which is defined in \eqref{4.1.4-1}.
\begin{equation}\label{5.6}
\begin{aligned}
\pa_Y\Phi_1^s(Y;{c})=&-2(\pa_Y-\a)\psi_{0,1}(Y)e^{-\a Y}\int_0^YU_s'(X)\psi_{0,2}(X)\dd X\\
&-2(
\pa_Y-\a)\psi_{0,2}(Y)e^{-\a Y}\int_Y^\infty U_s'(X)\psi_{0,1}(X)\dd X,\\
\pa_Y^2\Phi_1^s(Y;{c})=&-2(\pa_Y-\a)^2\psi_{0,1}(Y)e^{-\a Y}\int_0^YU_s'(X)\psi_{0,2}(X)\dd X\\
&-2(\pa_Y-\a)^2\psi_{0,2}(Y)e^{-\a Y}\int_Y^\infty U_s'(X)\psi_{0,1}(X)\dd X+2U_s'e^{-\a Y},\\
\pa_Y^3\Phi_1^s(Y;{c})=&-2(\pa_Y-\a)^3\psi_{0,1}(Y)e^{-\a Y}\int_0^YU_s'(X)\psi_{0,2}(X)\dd X\\
&-2(\pa_Y-\a)^3\psi_{0,2}(Y)e^{-\a Y}\int_Y^\infty U_s'(X)\psi_{0,1}(X)\dd X\\
&+2U_s''(Y)e^{-\a Y}-6\a U_s'(Y)e^{-\a Y}.
\end{aligned}
\end{equation}
Finally we state the following lemma about analyticity.
\begin{lemma}\label{lem.ap1}
	Let $\Omega\subset \mathbb{C}$ be an open subset and $J,K$ be two Banach spaces. Denote by $B(J,K)$ the space of bounded linear operators from $J$ to $K$. Assume that $$L(\cdot~;c):\Omega \mapsto B(J,K)$$ is analytic in $c$ with values in $B(J,K)$ and $$f(c): \Omega\mapsto J$$ is analytic in $c$ with values in $J$. Then $L(f;c)$ is analytic in $c$ with values in $K$.
\end{lemma}
\begin{proof}
Let $c_0\in \Omega$ be a fixed number. Since $L(\cdot~;c)$ is analytic at $c_0$, there exists $r_1>0$ and a sequence of bounded linear operator $\{L_n\}_{n=0}^{\infty}\subset B(J,K)$, such that
$$L(\cdot~;c)=\sum_{n=0}^{+\infty}(c-c_0)^nL_n(\cdot),
$$
with
$$\sum_{n=0}^{+\infty}|c-c_0|^n\|L_n\|_{B(J,K)}<\infty~\text{ for }|c-c_0|<r_1.
$$
Also by analyticity of $f$ we can find $r_2>0$ and a sequence of elements $f_n\in J$ such that
$$f(c)=\sum_{n=0}^{+\infty}(c-c_0)^nf_n,
$$
with
$$\sum_{n=0}^{+\infty}|c-c_0|^n\|f_n\|_{J}<\infty ~\text{ for }|c-c_0|<r_2.
$$
By linearity, for $|c-c_0|<\min\{r_1,r_2\}$ we deduce that
\begin{align}\label{ap1}
L\left(f(c);c\right)=\sum_{n=0}^{+\infty}(c-c_0)^nL_n\left(\sum_{m=0}^{+\infty} (c-c_0)^mf_m\right)=\sum_{n=0}^{+\infty}(c-c_0)^n\left(\sum_{m=0}^nL_{m}f_{n-m}\right).
\end{align}
For $|c-c_0|<\min\{r_1,r_2\}$, it holds
$$
\begin{aligned}\sum_{n=0}^{+\infty}|c-c_0|^n\left\|\sum_{m=0}^nL_{m}f_{n-m}\right\|_{K}\leq& \sum_{n=0}^{+\infty}|c-c_0|^n\left(\sum_{m=0}^n\|L_{m}\|_{B(J,K)}\|f_{n-m}\|_{J}\right)\\
\leq& \left(\sum_{n=0}^{+\infty}|c-c_0|^n\|L_n\|_{B(J,K)}\right)\times\left(\sum_{n=0}^{+\infty}|c-c_0|^n\|f_n\|_{J}\right)<\infty,
\end{aligned}
$$
which implies that the series on the right hand side of \eqref{ap1} absolutely converges in $K$. Therefore, $L(f;c)$ is analytic and the proof of Lemma \ref{lem.ap1} is completed.
\end{proof}

\end{document}